\documentclass[11pt,a4paper]{amsart}

\usepackage[utf8]{inputenc}
\usepackage[T1]{fontenc}
\usepackage[ngerman, english]{babel}
\usepackage{lmodern}
\usepackage{times}

\usepackage{graphicx}
\usepackage{latexsym}
\usepackage{dsfont}
\usepackage{amssymb,amsmath,amsfonts,amsthm}
\usepackage{array}

\usepackage{tikz}
\usetikzlibrary{matrix,arrows}

\newtheorem{Satz}{Satz}

\newtheorem{Lemma}[Satz]{Lemma}	

\newtheorem{Corollary}[Satz]{Corollary}
\newtheorem{Theorem}[Satz]{Theorem}
\newtheorem{Proposition}[Satz]{Proposition}	

\newtheorem*{Lemma*}{Lemma}

\theoremstyle{definition}

\newcommand{\C}{\mathbb{C}}

\newcommand{\R}{\mathbb{R}}

\newcommand{\N}{\mathbb{N}}
\newcommand{\T}{\mathbb{T}}

\newcommand{\cK}{\mathcal{K}}
\newcommand{\cB}{\mathcal{B}}

\newcommand{\cH}{\mathcal{H}}

\newcommand{\cP}{\mathcal{P}}
\newcommand{\cF}{\mathcal{F}}
\newcommand{\cJ}{\mathcal{J}}

\newcommand{\lin}{\mathrm{lin}}

\newcommand{\sign}{\mathrm{sign}}

\newcommand{\relint}{\mathrm{relint \, }}

\newcommand{\intd}{\mathrm{d}}

\newcommand{\1}{\mathds{1}}

\DeclareMathOperator{\SO}{SO}
\DeclareMathOperator{\OO}{O}
\DeclareMathOperator{\GOp}{G}
\DeclareMathOperator{\AOp}{A}

\renewcommand{\S}{\mathbb S^{n-1}}
\newcommand{\Sn}{\mathbb S^{n-1}}

\DeclareMathOperator{\Span}{{\rm span}}

\newcommand{\CM}[1]
  {
    \phi_{#1}
  }

\newcommand{\TenCM}[4]
  {
    \phi_{#1}^{#2,#3,#4}
  }

\begin{document}

  \title{Kinematic formulae for tensorial curvature measures}
  \date{December 26, 2016}

  \author{Daniel Hug \and Jan A. Weis}
  \address{Karlsruhe Institute of Technology (KIT), Department of Mathematics, D-76128 Karls\-ruhe, Germany}
  \email{daniel.hug@kit.edu}
  \address{Karlsruhe Institute of Technology (KIT), Department of Mathematics, D-76128 Karls\-ruhe, Germany}
  \email{jan.weis@kit.edu}

  \thanks{The authors were supported in part by DFG grants FOR 1548 and HU 1874/4-2}

\begin{abstract}
    Tensorial curvature measures are tensor-valued
    generalizations of the curvature measures of convex bodies.
    We prove a complete set of kinematic formulae for such tensorial
    curvature measures on convex bodies and for their (nonsmooth)
    generalizations on convex polytopes.
    These formulae express the integral mean of the tensorial
    curvature measure of the intersection of two given convex
    bodies (resp.~polytopes),  one of which is uniformly moved
    by a proper rigid motion, in terms of linear combinations
    of tensorial curvature measures of the given convex bodies
    (resp.~polytopes).
    We prove these results in a more direct way than in the
    classical proof of the principal kinematic formula for 
    curvature measures, which uses the connection to Crofton
    formulae to determine the involved constants explicitly.
\end{abstract}

\maketitle

  \section{Introduction}\label{secIntro}

    The \textit{principal kinematic formula} is a cornerstone of classical integral geometry. In its basic form
    in Euclidean space, it deals with integral mean values for distinguished geometric functionals with respect to the invariant measure on the group of proper rigid motions;
    see \cite[Chap.~5.1]{SchnWeil08} and \cite[Chap.~4.4]{Schneider14} for background
    information, recent developments and applications.  To be more specific, let
    $\cK^{n}$ denote the space of convex bodies (nonempty, compact, convex sets) in  $\R^{n}$.
    For two convex bodies  $K, K' \in \cK^{n}$ and $j \in \{ 0, \ldots, n \}$, the principal kinematic formula  states that
    \begin{align} \label{Form_Princ_KF}
      \int_{\GOp_n} V_{j} (K \cap g K') \, \mu( \intd g) = \sum_{k = j}^{n} \alpha_{n j k} V_{k}(K) V_{n - k + j}(K'),
    \end{align}
    where $\GOp_{n}$ denotes the group of proper rigid motions of $\R^{n}$, $\mu$ is the motion invariant Haar measure on $\GOp_{n}$, normalized in the usual way (see \cite[p. 586]{SchnWeil08}), and  the constant
    \begin{align*}
      \alpha_{njk}=\frac{\Gamma\left(\frac{k+1}{2}\right)
      \Gamma\left(\frac{n-k+j+1}{2}\right)}{\Gamma\left(\frac{j+1}{2}\right)
      \Gamma\left(\frac{n+1}{2}\right)}
    \end{align*}
    is expressed in terms of specific values of the Gamma function.

    The also appearing functionals $V_{j}: \cK^{n} \rightarrow \R$, $j \in \{ 0, \ldots, n \}$, are the \textit{intrinsic volumes}, which occur as the uniquely determined coefficients of the monomials in the \textit{Steiner formula}
    \begin{equation} \label{Form_Steiner}
      \mathcal{H}^n (K + \epsilon B^{n}) = \sum_{j = 0}^{n} \kappa_{n - j} V_{j} (K) \epsilon^{n - j},\qquad \epsilon \geq 0,
    \end{equation}
    which holds for all convex bodies $K \in \cK^{n}$. As usual in this context, $+$ denotes the Minkowski addition in $\R^{n}$,  $B^{n}$ is the Euclidean unit ball in $\R^{n}$ of $n$-dimensional volume $\kappa_{n}$, and $\mathcal{H}^n$ is the $n$-dimensional Hausdorff measure.
    Properties of the intrinsic volumes such as continuity, isometry invariance, additivity (valuation property) and homogeneity are derived from corresponding properties of the volume functional.
    A key result for the intrinsic volumes is \textit{Hadwiger's characterization theorem}, which states that $V_{0}, \ldots, V_{n}$ form a basis of the vector space of continuous and isometry invariant real-valued valuations on $\cK^{n}$ (see \cite[Theorem~6.4.14]{Schneider14}).
		This theorem can be used to derive not only \eqref{Form_Princ_KF}, but also Hadwiger's general integral geometric theorem (see \cite[Theorem~5.1.2]{SchnWeil08}).

    It is an important feature of the principal kinematic formula that the integral mean values of the intrinsic volumes can be expressed as  sums of products of intrinsic volumes of the two convex bodies involved, and no other functionals are required. In other words, the principal kinematic formulae for the intrinsic volumes constitute a complete (closed) system of integral geometric formulae. As a consequence, these formulae  can be iterated as described in \cite[Chap.~5.1]{SchnWeil08} and applied to the study of Boolean models in
		stochastic geometry (see \cite[Chap.~9.1]{SchnWeil08}). The bilinear structure of the right side of the kinematic formula \eqref{Form_Princ_KF} motivated the introduction and study of kinematic operators, in the recently developed field of algebraic integral geometry, which led to new insights, generalizations, and to a profound understanding of the structure of integral geometric formulae (see \cite{Bernigsurvey}, \cite{Fusurvey}) in connection with the algebraic structure of translation invariant valuations (see \cite{AleskerCourse}, also for further references).

    It is natural to extend the principal kinematic formula by applying the integration over the rigid motion group $\GOp_{n}$ to functionals which generalize the intrinsic volumes. A far reaching generalization is obtained by localizing the intrinsic volumes as measures, associated with convex bodies,  such that the intrinsic volumes are just the total measures. Specifically, this leads to the \emph{support measures} (generalized curvature measures) which are weakly continuous, locally defined and motion equivariant valuations on convex bodies with values in the space of finite measures on Borel subsets of $\R^{n} \times \Sn$, where $\Sn$ denotes the Euclidean unit sphere in $\R^{n}$. They are determined by a local version of the Steiner formula \eqref{Form_Steiner}, and thus they provide a natural example of a localization of the intrinsic volumes. Their marginal measures on Borel subsets of $\R^{n}$ are called \emph{curvature measures}. In 1959, Federer (see \cite[Theorem 6.11]{Federer59}) proved kinematic formulae for the curvature measures, even in the more general setting of sets with positive reach, which contain the classical kinematic formula as a very special case. More recently, kinematic formulae for support measures on convex bodies have been established by Glasauer in 1997 (see \cite[Theorem~3.1]{Glasauer97}). These formulae are based on a special set operation on support elements of the involved bodies, which limits their  usefulness for the present purpose, as explained in \cite{GoodeyHW}.

    Already in the early seventies, integral geometric formulae for \emph{quermassvectors} (curvature centroids) had been found by Hadwiger \& Schneider and Schneider \cite{HadSchn71, Schn72, Schn72b}. Recently, McMullen \cite{McMullen97} initiated a study of tensor-valued generalizations of the (scalar) intrinsic volumes and the vector-valued quermassvectors. This naturally raised the question for an analogue of Hadwiger's characterization theorem and for integral geometric formulae for basic additive, tensor-valued functionals (tensor valuations) on the space of convex bodies. As shown by Alesker \cite{Alesker99, Alesker99b}, and further studied in \cite{HSS07a}, there exist natural tensor-valued functionals, the basic \emph{Minkowski tensors}, which generalize the intrinsic volumes and span the vector space of tensor-valued, continuous, additive functionals on the space of convex bodies which are also isometry covariant. Although the basic {Minkowski tensors} span the corresponding vector space of tensor-valued valuations,  they satisfy nontrivial linear relationships and hence are not a basis. This fact and the inherent difficulty of computing Minkowski tensors explicitly for sufficiently many examples provide an obstacle for computing the constants involved in integral geometric formulae. Nevertheless major progress has been made in various works by different methods. Integral geometric Crofton formulae for general Minkowski tensors have been obtained in \cite{HugSchnSchu08}. A specific case has been further studied and applied to problems in stereology in \cite{KousKideHug15}, for various extensions see \cite{HugWeis16a} and \cite{SvaneJensen}. A quite general study of various kinds of integral geometric formulae for \emph{translation invariant} tensor valuations is carried out in \cite{BernHug15}, where also corresponding algebraic structures are explicitly determined. An approach to Crofton and thus kinematic formulae for translation invariant tensor valuations via integral geometric formulae for area measures (which are of independent interest) follows from \cite{GoodeyHW} and \cite{SchusterWannerer}. Despite all these efforts and substantial progress, a complete set of kinematic and Crofton formulae for general Minkowski tensors has not  been found so far. The current state of the art is described in various contributions of the lecture notes \cite{KidVed16}.

    Surprising new insight into integral geometric formulae can be gained by combining local and tensorial extensions of the classical intrinsic volumes. This setting has recently been studied by Schneider (see \cite{Schneider13}) and further analyzed by Hug \& Schneider in their works on \textit{local tensor valuations} (see \cite{HugSchn14, HugSchn16a, HugSchn16b}). These valuations can be viewed as tensor-valued generalizations of the support measures. On the other hand, they can be considered as localizations of the global tensor valuations introduced and first studied by McMullen (see \cite{McMullen97}) and characterized by Alesker in a Hadwiger style (see \cite{Alesker99, Alesker99b}), as pointed out before. Inspired by the characterization results obtained in \cite{Schneider13, HugSchn14, HugSchn16a, HugSchn16b,Saienko}, we consider tensor-valued curvature measures,  the \textit{tensorial curvature measures}, and their (nonsmooth) generalizations, which basically appear for polytopes, and establish a complete set of kinematic formulae for these (generalized) tensorial curvature measures.

    Kinematic formulae for the generalized tensorial curvature measures on polytopes (which do not have a continuous extension to all convex bodies) have not been considered before. In fact, our results are new even for the tensorial curvature measures which are obtained by integration against  support measures. The constants involved in these formulae are surprisingly simple and can be expressed as a concise  product of Gamma functions. Although some information about tensorial kinematic formulae can be gained from abstract characterization results (as developed in \cite{Schneider13,HugSchn14}), we believe that explicit results cannot be obtained by such an approach, at least not in a simple way.
    In contrast, our argument starts as a tensor-valued version of the proof of the kinematic formula for curvature measures (see \cite[Theorem 4.4.2]{Schneider14}). But instead of first deriving Crofton formulae to obtain the coefficients of the appearing functionals, we have to proceed in a direct way. In fact, the explicit derivation of the constants in related Crofton formulae  via the template method does not seem to be feasible. In an accompanying paper \cite{HWCrofton}, we shall provide explicit Crofton formulae for tensorial curvature measures as a straightforward consequence of our general kinematic formulae,  and relate them to previously obtained special results. The main technical part of the present argument, which requires the calculation of rotational averages over Grassmannians and the rotation group, is new even in the scalar setting.
		
			The current work explores generalizations of the principal kinematic formula to tensorial measure-valued valuations. 
		Various other directions have been taken in extending the classical framework. 
		Kinematic formulae for support functions have been studied by Weil \cite{Weil1995}, 
		Goodey \& Weil  \cite{GoodeyWeil2003},  
		and Schneider \cite{Schneider2003}, recent related work on mean section bodies and Minkowski valuations is due to 
		Schuster \cite{Schuster2010}, Goodey \& Weil \cite{GoodeyWeil2014}, and Schuster \& Wannerer \cite{SchusterWannerer}, a Crofton 
		formula for Hessian measures of convex functions has been established and applied in \cite{ColesantiHug2002}. Instead 
		of changing the functionals involved in the integral geometric formulae, it is also natural and in fact required by 
		applications to stochastic geometry to explore formulae where the integration is extended over subgroups $G$ of the  motion group. The extremal cases are translative and rotational integral geometry, where $G=\R^n$ and $G=\OO(n)$, 
		respectively. The former is described in detail in \cite[Chap.~6.4]{SchnWeil08}, recent progress for scalar- and measure-valued valuations and further references are provided in \cite{WeilI, WeilII, HugRataj}, applications 
		 to stochastic geometry are given in \cite{HoerrmannWeil, HHKM14, HoerrmannDiss15}, where translative integral 
		formulae for tensor-valued measures are established and applied. Rotational Crofton formulae for tensor valuations have recently been developed further by Auneau et al. \cite{ACZJV, ARVJ} and Svane \& Vedel Jensen \cite{SvaneJensen} 
		(see also the literature cited there), applications to stereological estimation and bio-imaging are treated and discussed in \cite{RZNJ, ZNVJ, VJR}. Various other groups of isometries, also in Riemannian isotropic spaces, have been studied in recent years. Major progress has been made, for instance, in Hermitian integral geometry (in curved spaces), where the interplay between 
		global and local results turned out to be crucial (see \cite{BernigFu11, BFS14, 
		Fu90, Fu06, Wannerer1, Wannerer2, Solanes15} and the survey \cite{Bernig17}), 
		but various other group actions have been studied successfully as well (see \cite{AleskerFaifman14, Bernig11, 
		Bernig12, BernigFaifman16, BernigSolanes14, BernigSolanes16, BernigVoide16, Faifman16}).
		
Minkowski tensors, tensorial curvature measures, and general local tensor valuations are useful 
morphological characteristics that allow to describe the geometry of complex spatial structure 
and are particularly well suited for developing structure-property relationships for tensor-valued or orientation-dependent physical 
properties; see  \cite{Schroeder10, Schroeder11} for surveys and Klatt's PhD thesis \cite{Klatt16} for an in depth  analysis of various aspects (including random fields and percolation) of the interplay between physics and Minkowski tensors. These applications cover a wide spectrum ranging from nuclear physics \cite{Schue15}, granular matter 
\cite{Kapfer12, Xia14, Schaller15, Kuhn15}, density functional theory \cite{Wittmann14}, physics of complex plasmas 
\cite{Boebel16}, to physics of materials science \cite{Saadatfar12}. Characterization and classification theorems for tensor valuations, uniqueness and reconstruction results \cite{HK16, KK15, Kousholt16, KousholtPhD16}, which are accompanied by numerical algorithms \cite{Schroeder10, Schroeder11, HKS15, CK16}, stereological estimation procedures   \cite{KousKideHug15, KZKJ16}, and 
integral geometric formulae, as considered in the present work, form the foundation for these and many other 
applications.

    The paper is structured as follows. Section \ref{secPre} contains a brief introduction to the basic concepts and definitions required to state our results. The main theorem (Theorem \ref{Thm_Main}) and its consequences are described in Section \ref{Tmr}, where also further comments on the structure of the obtained formulae are provided. The proof of  Theorem  \ref{Thm_Main}, which is given in Section \ref{secProof}, is preceded by several auxiliary results. These concern integral averages over Grassmannians and the rotation group and are the subject of Section \ref{secAux}. The proof of Theorem  \ref{Thm_Main} is divided into four main steps which are outlined at the beginning of Section \ref{secProof}. In the course of that proof, iterated sums involving Gamma functions build up. In a final step, these expressions have to be simplified again. Some basic tools which are required for this purpose are collected in an appendix.

  \section{Preliminaries}\label{secPre}

    We work in the $n$-dimensional Euclidean space
    $\R^{n}$, $n\ge 2$, equipped with its usual topology generated by the standard
    scalar product $\langle \cdot\,, \cdot \rangle$ and the corresponding Euclidean norm
    $\| \cdot \|$. For a topological space $X$, we denote the Borel
    $\sigma$-algebra on $X$ by $\cB(X)$.

    The algebra of symmetric tensors over
    $\R^n$ is denoted by $\T$ (the underlying $\R^n$ will be clear from the context),
		the vector space of symmetric tensors of rank $p\in\N_0$ is denoted
		by $\T^p$ with $\T^0=\R$. The symmetric tensor product of
    tensors $T_i\in\T$, $i=1,2$, over $\R^{n}$ is denoted by $T_1T_2\in \T$,
		and for $q\in\N_0$ and a tensor $T\in\T$ we write $T^{q}$ for the  $q$-fold tensor
    product of $T$, where $T^0:=1$; see also the
    contributions \cite{HugSchneider16c,BernigHug15b} in the lecture notes \cite{KidVed16}
		for further details and references.
    Identifying $\R^{n}$
    with its dual space via the given scalar product, we interpret a symmetric
    tensor of rank $p$ as a symmetric $p$-linear map from
    $(\R^{n})^{p}$ to $\R$. A special tensor is the \emph{metric
      tensor} $Q \in \T^{2}$, defined by $Q(x, y) := \langle x, y \rangle$ for $x,
    y \in \R^{n}$. For an affine $k$-flat $E\subset\R^n $, $k \in \{0,
    \ldots, n\}$, the metric tensor $Q(E)$ associated with $E$ is defined by $Q(E)(x,
    y) := \langle p_{E^{0}} (x), p_{E^{0}} (y) \rangle$ for $x, y \in \R^{n}$, where $E^0$
		denotes the linear direction space of $E$ (see Section \ref{secAux}) and $p_{E^{0}} (x)$
		is the orthogonal projection of $x$ to $E^0$. If $F\subset\R^n$ is a $k$-dimensional convex 
		set, then we again write $Q(F)$ for the metric tensor $Q(\text{aff}(F))=Q(\text{aff}(F)^0)$ associated
		with the affine subspace $\text{aff}(F)$ spanned by $F$.

    In order to define the tensorial curvature measures and to explain
		how they are related to the support measures,
		we start with the latter.
    For a convex body $K \in \cK^{n}$ and
  	 $x \in \R^{n}$, we denote the
    metric projection of $x$ onto $K$ by $p(K, x)$, and  we define $u(K, x)
    := (x - p(K, x)) / \| x - p(K, x) \|$ for $x \in \R^{n} \setminus
    K$. For $\epsilon >
    0$ and a Borel set $\eta \subset \Sigma^{n}:=\R^{n} \times
    \Sn$,
    \begin{equation*}
      M_{\epsilon}(K, \eta) := \left\{ x \in \left( K + \epsilon B^{n} \right)
      \setminus K \colon \left( p(K, x), u(K, x) \right) \in \eta \right\}
    \end{equation*}
    is a local parallel set of $K$ which satisfies the \emph{local Steiner
      formula}
    \begin{equation} \label{14-Form_Steiner_loc}
		\mathcal{H}^n (M_{\epsilon}(K,
      \eta)) = \sum_{j = 0}^{n - 1} \kappa_{n - j} \Lambda_{j} (K, \eta)
      \epsilon^{n - j}, \qquad \epsilon \geq 0.
    \end{equation}
    This relation determines the \emph{support measures} $\Lambda_{0}
    (K, \cdot), \ldots, \Lambda_{n - 1} (K, \cdot)$ of $K$, which are
    finite Borel measures on $\cB (\Sigma^{n})$. Obviously, a comparison
    of \eqref{14-Form_Steiner_loc} and the Steiner formula yields
    $V_{j}(K) = \Lambda_{j} (K, \Sigma^{n})$. For further information
   see \cite[Chap.~4.2]{Schneider14}.

    Let $\cP^n\subset\cK^n$ denote the space of convex polytopes in $\R^n$.
    For a polytope $P \in \cP^{n}$ and $j \in \{ 0, \ldots, n \}$, we
    denote the set of $j$-dimensional faces of $P$ by $\cF_{j}(P)$ and
    the normal cone of $P$ at a face $F \in \cF_{j}(P)$ by $N(P,F)$.
    Then, the $j$th support measure $\Lambda_{j} (P, \cdot)$ of $P$ is explicitly given by
    \begin{equation*}
      \Lambda_{j} (P, \eta) = \frac {1} {\omega_{n - j}}
      \sum_{F \in \cF_{j}(P)} \int _{F} \int _{N(P,F) \cap \Sn}
      \1_\eta(x,u)
      \, \cH^{n - j - 1} (\intd u) \, \cH^{j} (\intd x)
    \end{equation*}
    for $\eta \in \cB(\Sigma^{n})$ and $j \in \{ 0, \ldots, n - 1 \}$,
		where $\cH^{j}$ denotes the $j$-dimensional Hausdorff measure and
		$\omega_{n}$ is the $(n - 1)$-dimensional volume of $\Sn$.
		
		For a polytope $P\in\mathcal{P}^n$, we define the
    \emph{generalized tensorial curvature measure} 
		$$
		\TenCM{j}{r}{s}{l} (P, \cdot) ,\qquad j \in\{0, \ldots, n - 1\}, \, r, s, l \in \N_{0},
		$$
		as the Borel measure on $\mathcal{B}(\R^n)$ which is given by
    \begin{equation*}
      \TenCM{j}{r}{s}{l} (P, \beta) :=  c_{n, j}^{r, s, l}\,\frac{1}{\omega_{n - j}}
			\sum_{F \in \cF_{j}(P)} Q(F)^{l} \int _{F \cap \beta} x^r \, \cH^{j}(\intd x)
			\int _{N(P,F) \cap \Sn} u^{s} \, \cH^{n - j - 1} (\intd u),
    \end{equation*}
    for $\beta \in \cB(\R^{n})$,   where
    \begin{align*}
      c_{n, j}^{r, s, l} := \frac{1}{r! s!} \frac{\omega_{n - j}}{\omega_{n - j + s}} \frac{\omega_{j + 2l}}{\omega_{j}}  \text{ if }j\neq 0, \quad\text{ $c_{n, 0}^{r, s, 0} := \frac{1}{r! s!}
		\frac{\omega_{n}}{\omega_{n + s}}$},\quad \text{ and $c_{n, 0}^{r, s, l} :=1$ for $l\ge 1$.}
    \end{align*}
		Note that if $j = 0$ and $l \ge 1$,
		then we have $\TenCM{0}{r}{s}{l} \equiv 0$. In all other cases the factor $1/\omega_{n-j}$ in the definition of $\TenCM{j}{r}{s}{l} (P, \beta)$ and the factor $\omega_{n-j}$ involved in the constant $c_{n, j}^{r, s, l}$ cancel.
		
    For a general convex body $K\in\mathcal{K}^n$, we define the \emph{tensorial curvature measure}  
		$$
		\TenCM{n}{r}{0}{l} (K, \cdot),\qquad r,l\in\N_0,
		$$
		as the Borel measure on $\mathcal{B}(\R^n)$ which is given by
    \begin{equation*}
      \TenCM{n}{r}{0}{l} (K, \beta): = c_{n, n}^{r, 0, l} \, Q^{l} \int _{K \cap \beta} x^r \, \cH^{n}(\intd x),
    \end{equation*}
		for $\beta \in \cB(\R^{n})$,  
    where  $c_{n, n}^{r, 0, l} := \frac{1}{r!} \frac{\omega_{n + 2l}}{\omega_{n}}$.
    For the sake of convenience, we extend these definitions by $\TenCM{j}{r}{s}{0} := 0$
		for $j \notin \{ 0, \ldots, n \}$ or $r \notin \N_{0}$ or $s \notin \N_{0}$ or $j = n$ and $s \neq 0$. Finally, we
		observe that for $P\in\mathcal{P}^n$, $r=s=l=0$, and $j=0,\ldots,n-1$, the scalar-valued measures 
		$\TenCM{j}{0}{0}{0}(P,\cdot)$ are just the curvature measures $\CM{j}(P,\cdot)$, that is,
		the marginal measures on $\R^n$ of the support measures $\Lambda_j(P,\cdot)$, which therefore can be extended from 
		polytopes to general convex bodies, and $\TenCM{n}{0}{0}{0}(K,\cdot)$ is the restriction 
		of the $n$-dimensional Hausdorff measure to $K\in\mathcal{K}^n$.

		To put the generalized tensorial curvature measures into their natural context and to emphasize some of their properties,
		we recall the relevant definitions and results from \cite{Schneider13,HugSchn14,HugSchneider16c}.  For $p\in\N_0$,
		let $T_p(\mathcal{P}^n)$ denote the vector space of all mappings $\tilde{\Gamma}:\mathcal{P}^n\times \mathcal{B}(\Sigma^n)\to \mathbb{T}^p$ such that
		\begin{itemize}
		\item $\tilde{\Gamma}(P,\cdot)$ is a $\mathbb{T}^p$-valued measure on
		$\mathcal{B}(\Sigma^n)$, for each $P\in\mathcal{P}^n$;
		\item $\tilde{\Gamma}$ is isometry covariant;
		\item $\tilde{\Gamma}$ is  locally defined.
		\end{itemize}
		We refer to \cite{Schneider13,HugSchn14,HugSchneider16c}
		for explicit definitions of these properties. 
		
		For a polytope $P\in\mathcal{P}^n$, the \emph{generalized local Minkowski tensor} 
		$$
		\tilde{\phi}^{r,s,l}_j (P, \cdot),\qquad j\in\{0,\ldots,n-1\},\, r,s,l\in\N_0,
		$$
		is the Borel measure on $\mathcal{B}(\Sigma^n)$ which is defined by 
		$$
      \tilde{\phi}^{r,s,l}_j (P, \eta) :=  c_{n, j}^{r, s, l}\,\frac{1}{\omega_{n - j}}
			\sum_{F \in \cF_{j}(P)} Q(F)^{l} \int _{F}\int _{N(P,F) \cap \Sn} \1_\eta(x,u) x^r u^{s}\, \cH^{j}(\intd x)
		  \, \cH^{n - j - 1} (\intd u),
		$$
for $\eta\in \mathcal{B}(\Sigma^n)$.

It was shown in 	\cite{Schneider13,HugSchn14} (where a different notation and normalization was used)
 that the mappings $Q^m	\tilde{\phi}^{r,s,l}_j$, where  $m,r,s,l\in\N_0$ satisfy $2m+r+s+2l=p$, where $
j\in\{0,\ldots,n-1\}$, and where $l=0$ if $j\in\{0,n-1\}$, form a basis of $T_p(\mathcal{P}^n)$. 

This fundamental characterization theorem highlights the importance of the generalized local Minkowski tensors. 
In particular, since the mappings $P\mapsto Q^m	\tilde{\phi}^{r,s,l}_j(P,\cdot)$, $P\in \mathcal{P}^n$, are additive, as shown in \cite{HugSchn14}, 
 all mappings in $T_p(\mathcal{P}^n)$ are valuations. 
		
		Noting that 
		$$ \TenCM{j}{r}{s}{l} (P, \beta)=\tilde{\phi}^{r,s,l}_j 
		(P, \beta\times\Sn),\qquad j\in\{0,\ldots,n-1\},\, r,s,l\in\N_0,
		$$
		for  $P\in\mathcal{P}^n$ and all $\beta\in\mathcal{B}(\R^n)$, it is clear that the mappings 
		$$ \TenCM{j}{r}{s}{l} :
		\mathcal{P}^n\times \mathcal{B}(\R^n)\to \mathbb{T}^p,\qquad (P,\beta)\mapsto \TenCM{j}{r}{s}{l}(P,\beta),
		$$
		where
		$p=r+s+2l$, have similar properties as the generalized local Minkowski tensors. In particular, it is easy to see 
		that (including the case $j=n$ where $\mathcal{P}^n$ can be replaced by $\mathcal{K}^n$)
		\begin{enumerate}
		\item[(a)] $\TenCM{j}{r}{s}{l}(P,\cdot)$ is a $\mathbb{T}^p$-valued measure on $\cB(\R^n)$, for each $P\in\mathcal{P}^n$;
		\item[(b)] $\TenCM{j}{r}{s}{l}$ is isometry covariant, that is, translation covariant of degree $r$ in the sense that
		$$
		 \TenCM{j}{r}{s}{l} (P+t, \beta+t)=\sum_{i=0}^r \TenCM{j}{r-i}{s}{l} (P, \beta)\frac{t^i}{i!},
		$$
		for all $P\in\mathcal{P}^n$,  $\beta\in\cB(\R^n)$, and $t\in\R^n$, and rotation covariant in the sense that
		$$
		 \TenCM{j}{r}{s}{l} (\vartheta P, \vartheta\beta)=\vartheta  \TenCM{j}{r}{s}{l} ( P, \beta),
		$$
		for all $P\in\mathcal{P}^n$,  $\beta\in\cB(\R^n)$, and $\vartheta\in \OO(n)$ (the orthogonal group of $\R^n$);
		\item[(c)] $\TenCM{j}{r}{s}{l}$ is locally defined, that is, if $\beta\subset\R^n$ is open and $P,P'\in\mathcal{P}^n$
		are such that $P\cap \beta=P'\cap \beta$, then $\TenCM{j}{r}{s}{l}(P,\gamma)=\TenCM{j}{r}{s}{l}(P',\gamma)$
		for all Borel sets $\gamma\subset\beta$  (This notion is common in this context, though different from what we called ``locally defined'' before.);
		\item[(d)] $P\mapsto \TenCM{j}{r}{s}{l}(P,\cdot)$, $P\in\mathcal{P}^n$, is additive (a valuation).
		\end{enumerate}
	It is an open problem whether the vector space of all mappings
	 $ \Gamma : \mathcal{P}^n\times \mathcal{B}(\R^n)\to \mathbb{T}^p$ satisfying these properties,
	is spanned by the mappings  $Q^m\TenCM{j}{r}{s}{l}$, where $m,r,s,l\in\N_0$ satisfy $2m+r+s+2l=p$,
	where $j\in\{0,\ldots,n-1\}$, and where $l=0$ if $j\in\{0,n-1\}$, or where $j=n$ and $s=l=0$.  
	The linear independence of
	these mappings can be shown in a similar way as Theorem 3.1 in \cite{HugSchn14} was shown. 
	 A characterization 
	theorem for {\em smooth} tensor-valued curvature measures has recently been found by Saienko \cite{Saienko}. 
	Note, however, that the tensorial curvature measures $\TenCM{j}{r}{s}{l}$, for $1\le j\le n-2$ and $l>1$, are not smooth. 
	We also point out 
	that, for every $\beta\in\cB(\R^n)$,  the mapping $\TenCM{j}{r}{s}{l}(\cdot,\beta)$ on $\mathcal{P}^n$
	is measurable.  This is implied by the more general Lemma \ref{lemmeas1} in the appendix.

    It has been shown in \cite{HugSchn14} that the generalized local Minkowski tensor $\tilde{\phi}^{r,s,l}_j$ has a continuous extension
    to $\cK^{n}$ which preserves all other properties if and only if $l \in \{0,1\}$;
		see \cite[Theorem 2.3]{HugSchn14} for a  stronger characterization result. 
		 Globalizing any such continuous extension in the $\Sn$-coordinate,
    we obtain a continuous extension for the generalized tensorial curvature measures.
    For $l = 0$, there exists a natural representation of the extension
		via the support measures. We call these the
    \emph{tensorial curvature measures}. For a convex body $K \in \cK^{n}$, a Borel
		set $\beta \in \cB(\R^{n})$, and $r, s \in \N_{0}$, they are given by
    \begin{equation}\label{alsomeas}
      \TenCM{j}{r}{s}{0} (K, \beta):= c_{n, j}^{r, s, 0} \int _{\beta \times \Sn} x^r u^s \, \Lambda_j(K, \intd (x, u)),
    \end{equation}
    for $j\in\{0,\ldots,n-1\}$, whereas $ \TenCM{n}{r}{0}{l} (K, \beta)$ has 
		already been defined for all $K \in \cK^{n}$. 
		
		For an explicit description of the generalized local Minkowski tensors $\tilde\phi_j^{r,s,1}(K,\cdot)$, 
		for $K \in \cK^{n}$ and  
		$j\in\{0,\ldots,n-1\}$, and hence of  $\TenCM{j}{r}{s}{1} (K, \cdot)$, we refer to \cite{HugSchn14}. 
		There it is  shown that the map $K\mapsto\tilde\phi_j^{r,s,1}(K,\eta)$, $K \in \cK^{n}$, is measurable for all 
		$\eta\in\mathcal{B}(\Sigma^n)$, which yields that the map $K\mapsto\TenCM{j}{r}{s}{1} (K, \beta)$, $K \in \cK^{n}$, is measurable for all 
		$\beta\in\mathcal{B}(\R^n)$. Moreover, the measurability of the map $K\mapsto\phi_j^{r,s,0}(K,\beta)$, $K \in \cK^{n}$,
		is clear from \eqref{alsomeas}.

    In the coefficients of the kinematic formula and in the proof of our main theorem, the classical \emph{Gamma function}
		is involved. It can be defined via the Gaussian product formula
    \begin{align*}
      \Gamma(z) := \lim_{a \rightarrow \infty} \frac{a^{z} a!}{z (z + 1) \cdots (z + a)}
    \end{align*}
    for all $z \in \C \setminus \{ 0, -1, \ldots \}$ (see \cite[(2.7)]{Artin64}). For $c \in \R\setminus \mathbb{Z}$ and $m\in \N_{0}$, this definition
		implies that
		\begin{align} \label{Form_Gam_Cont}
      \frac{\Gamma(-c + m)}{\Gamma(-c)} = (-1)^{m} \frac{\Gamma(c + 1)}{\Gamma(c - m + 1)}.
    \end{align}
    The Gamma function has simple poles at the nonpositive integers. The right side of relation \eqref{Form_Gam_Cont}
		provides a continuation of the left side at  $c \in \N_{0}$, where $\Gamma(c - m + 1)^{-1} = 0$ for $c < m$.
		In the following, we also repeatedly use Legendre's
		duplication formula, which states that
		$$
		\Gamma(c)\Gamma(c + \tfrac{1}{2})=2^{1 - 2c}\sqrt{\pi}\,\Gamma(2c)
		$$
		for $c > 0$.

  \section{The main results}\label{Tmr}

    In the present work, we establish explicit kinematic formulae for the generalized tensorial curvature measures $\TenCM{j}{r}{s}{l}$ of polytopes. In other words, for $P, P' \in \cP^n$ and $\beta, \beta' \in \cB(\R^n)$ we express the integral mean value
    \begin{align*}
      \int_{\GOp_n} \TenCM{j}{r}{s}{l} (P \cap g P', \beta \cap g \beta') \, \mu( \intd g)
    \end{align*}
    in terms of the generalized tensorial curvature measures of $P$ and $P'$, evaluated at $\beta$ and $\beta'$, respectively. In fact, the precise result shows that only  a
		selection of these measures is needed. Furthermore, for $l =0, 1$,
		the tensorial measures $\TenCM{j}{r}{s}{l}$ are defined on $\cK^{n} \times \cB(\R^{n})$, and therefore in these two cases
		we also consider
		 integral means of the form
    \begin{align*}
      \int_{\GOp_n} \TenCM{j}{r}{s}{l} (K \cap g K', \beta \cap g \beta') \, \mu( \intd g),
    \end{align*}
    for general $K, K' \in \cK^n$ and $\beta, \beta' \in \cB(\R^n)$. Although the latter result can be deduced as a consequence of the former,  it came as a surprise that the general formulae simplify for $l\in\{0,1\}$ so that only tensorial curvature measures are involved  which admit a continuous extension.

    \begin{Theorem} \label{Thm_Main}
      For $P, P' \in \cP^n$, $\beta, \beta' \in \cB(\R^n)$, $j, l, r, s \in \N_{0}$ with $j \leq n$, and  $l=0$ if $j=0$,
      \begin{align}
        & \int_{\GOp_n} \TenCM{j}{r}{s}{l} (P \cap g P', \beta \cap g \beta') \, \mu( \intd g) \nonumber \\
        & \qquad = \sum_{k = j}^{n} \sum_{m = 0}^{\lfloor \frac s 2 \rfloor} \sum_{i = 0}^{m}
				c_{n, j, k}^{s, l, i, m} \, Q^{m - i} \TenCM{k}{r}{s - 2m}{l + i} (P, \beta)
				\CM{n - k + j} (P', \beta'),
				\label{Form_Thm_Main}
      \end{align}
      where
      \begin{align*}
        c_{n, j, k}^{s, l, i, m} & : = \frac{(-1)^{i}}{(4\pi)^{m} m!} \frac{\binom{m}{i}}{\pi^i}
				\frac {(i + l - 2)!} {(l - 2)!} \frac {\Gamma(\frac {n - k + j + 1} 2)
				\Gamma(\frac {k + 1} 2)} {\Gamma(\frac {n + 1} 2) \Gamma(\frac {j + 1} 2)} \\
        & \qquad \times \frac {\Gamma(\frac {k} 2 + 1)} {\Gamma(\frac {j} 2 + 1)}
				\frac {\Gamma(\frac {j + s} 2 - m + 1)} {\Gamma(\frac {k + s} 2 + 1)}
				\frac {\Gamma(\frac {k - j} 2 + m)} {\Gamma(\frac{k - j}{2})}.
      \end{align*}
    \end{Theorem}

Several remarkable facts concerning the coefficients $ c_{n, j, k}^{s, l, i, m}$ should be observed. First,
    the ratio ${(i + l - 2)!} /{(l - 2)!}$ has to be interpreted in terms of Gamma functions and relation \eqref{Form_Gam_Cont} if $l\in\{0,1\}$, as described below. The corresponding special cases will be considered separately in the following
		two theorems. Second, the coefficients are indeed independent of the tensorial parameter $r$ and depend only on $l$ through
		the ratio ${(i + l - 2)!} /{(l - 2)!}$. Moreover, only tensors $\TenCM{k}{r}{s - 2m}{p} (P, \beta) $ with $p \ge l$ show up
		on the right side of the kinematic formula.
		Using Legendre's duplication formula, we could shorten the given expressions for the  coefficients $ c_{n, j, k}^{s, l, i, m}$
		even further. However, the present form has the advantage of exhibiting that the factors in the second line cancel each other if $s=0$ (and hence also $m=i=0$). Furthermore,
		the coefficients are signed in contrast to the classical kinematic formula. We shall see below that for $l\in\{0,1\}$ all coefficients
		are nonnegative.

    In Theorem \ref{Thm_Main}, we can simplify the coefficient $c_{n, j, k}^{s, l, i, m}$ for $k \in \{j, n \}$ and $j\le n-1$ such that only one functional remains.  From \eqref{Form_Gam_Cont} we conclude that
    \begin{align*}
      c_{n, j, j}^{s, l, i, m} = \1\{ i = m = 0 \}.
    \end{align*}
    Furthermore, since $\TenCM{n}{r}{s}{l}$ vanishes for $s \neq 0$ and the functionals $Q^{\frac{s}{2} - i}
		\TenCM{n}{r}{0}{l + i}$, $i \in \{ 0, \ldots, \frac{s}{2} \}$, can be combined, we can redefine
      \begin{align*}
        c_{n, j, n}^{s, l, i, m} & := \1 \{ s \text{ even}, m = i = \tfrac{s}{2} \} \frac{1}{(2\pi)^{s}
				(\frac{s}{2})!} \frac{\Gamma(\frac{n - j + s}{2})}{\Gamma(\frac{n - j}{2})};
      \end{align*}
			see \eqref{Plus} and \eqref{PlusPlus} in the proof of Theorem \ref{Thm_Main}.
			
			It should also be observed that
		the functionals $\TenCM{n-1}{r}{s-2m}{l+i}$
		can be expressed in terms of the functionals $Q^{m'}\TenCM{n-1}{r}{s'}{0}$, where $m',s'\in\N_0$ and $2m'+s'=s+2l$. We do
		not pursue this here, since the resulting coefficients do not simplify nicely (see, however, \cite{HWCrofton}).
			
		Theorem \ref{Thm_Main} states an equality for measures, hence the case $r=0$ of the theorem immediately implies
		the general case. In fact, algebraic induction and the inversion invariance of $\mu$ immediately yield the following extension of Theorem \ref{Thm_Main}.
		
		\begin{Corollary}\label{corgenThm1}
		 Let $P, P' \in \cP^n$,  $j, l, r, s \in \N_{0}$ with $j \leq n$, and  $l=0$ if $j=0$. Let
		 $f,h$ be  tensor-valued continuous functions  on $\R^n$. Then
		\begin{align*}
        & \int_{\GOp_n}\int_{\R^n}f(x)h(gx)\, \TenCM{j}{r}{s}{l} (P \cap g^{-1} P', \intd x) \, \mu( \intd g) \nonumber \\
        & \qquad = \sum_{k = j}^{n} \sum_{m = 0}^{\lfloor \frac s 2 \rfloor} \sum_{i = 0}^{m}
				c_{n, j, k}^{s, l, i, m} \, Q^{m - i} \int_{\R^n}f(x)\,\TenCM{k}{r}{s - 2m}{l + i} (P, \intd x)\int_{\R^n} h(y)
				\,\CM{n - k + j} (P', \intd y).
      \end{align*}
			\end{Corollary}
			
		In particular, we could choose $h(y)=y^{\bar r}$, $y\in\R^n$, for $\bar r\in\N_0$. Moreover, since the generalized tensorial curvature measures depend additively on the underlying polytope, all integral formulae remain true if $P$ and $P'$ are replaced by finite unions of polytopes. Similar extensions hold for the following results.
	
		\medskip
		
We state and prove Theorem \ref{Thm_Main} in the present form, since this does not change the argument
		and globalization yields corresponding results for general Minkowski tensors.

    The cases $l = 0, 1$ are of special interest, since we can formulate the kinematic formulae
		for general convex bodies in these cases.

    \begin{Theorem} \label{Cor_Main_l=1}
      For $K, K' \in \cK^n$, $\beta, \beta' \in \cB(\R^n)$ and $j, r, s \in \N_0$ with $1\le j \leq n$,
      \begin{align*}
        & \int_{\GOp_n} \TenCM{j}{r}{s}{1} (K \cap g K', \beta \cap g \beta') \, \mu( \intd g) \\
        & \qquad = \sum_{k = j}^{n} \sum_{m = 0}^{\lfloor \frac s 2 \rfloor}
				c_{n, j, k}^{s, 1, 0, m} \, Q^{m} \TenCM{k}{r}{s - 2m}{1} (K, \beta) \CM{n - k + j} (K', \beta'),
      \end{align*}
      where
      \begin{align*}
        c_{n, j, k}^{s, 1, 0, m} & = \frac{1}{(4\pi)^{m} m!} \frac {\Gamma(\frac {n - k + j + 1} 2) \Gamma(\frac {k + 1} 2)} {\Gamma(\frac {n + 1} 2) \Gamma(\frac {j + 1} 2)} \frac {\Gamma(\frac {k} 2 + 1)} {\Gamma(\frac {j} 2 + 1)} \frac {\Gamma(\frac {j + s} 2 - m + 1)} {\Gamma(\frac {k+s} 2 + 1)} \frac {\Gamma(\frac {k - j} 2 + m)} {\Gamma(\frac{k - j}{2})}.
      \end{align*}
    \end{Theorem}

    \begin{proof}
      We apply \eqref{Form_Gam_Cont} to obtain
      \begin{align*}
        \frac {(i - 1)!} {(- 1)!} = \frac {\Gamma (i)} {\Gamma (0)} = \1\{ i = 0 \}.
      \end{align*}
      Then, Theorem \ref{Thm_Main} yields the assertion in the polytopal case. For a convex body, we conclude the formula by approximating it by polytopes, since the valuations $\TenCM{k}{r}{s - 2m}{1}$ have weakly continuous extensions to $\cK^n$ (and the same is true for the curvature measures).
    \end{proof}

    \begin{Theorem} \label{Cor_Main_l=0}
      For $K, K' \in \cK^n$, $\beta, \beta' \in \cB(\R^n)$ and $j, r, s \in \N_0$ with $j \leq n$,
      \begin{align*}
        & \int_{\GOp_n} \TenCM{j}{r}{s}{0} (K \cap g K', \beta \cap g \beta') \, \mu( \intd g) \\
        & \qquad = \sum_{k = j}^{n} \sum_{m = 0}^{\lfloor \frac s 2 \rfloor} \sum_{i = 0}^{1} c_{n, j, k}^{s, 0, i, m} \, Q^{m - i} \TenCM{k}{r}{s - 2m}{i} (K, \beta) \CM{n - k + j} (K', \beta'),
      \end{align*}
      where
      \begin{align*}
        c_{n, j, k}^{s, 0, i, m}  = \frac{1}{(4\pi)^{m} m!} \frac{\binom{m}{i}}{\pi^{i}} \frac {\Gamma(\frac {n - k + j + 1} 2) \Gamma(\frac {k + 1} 2)} {\Gamma(\frac {n + 1} 2) \Gamma(\frac {j + 1} 2)}
           \frac {\Gamma(\frac {k} 2 + 1)} {\Gamma(\frac {j} 2 + 1)} \frac {\Gamma(\frac {j + s} 2 - m + 1)} {\Gamma(\frac {k+s} 2 + 1)} \frac {\Gamma(\frac {k - j} 2 + m)} {\Gamma(\frac{k - j}{2})}.
      \end{align*}
    \end{Theorem}

    \begin{proof}
      We apply \eqref{Form_Gam_Cont} to obtain
      $$
        \frac {(i - 2)!} {(- 2)!} = \frac {\Gamma (i - 1)} {\Gamma (- 1)}  =
				(-1)^{i} \frac {1} {\Gamma (2 - i)} =  \1\{ i = 0 \}-\1\{ i = 1 \}.
      $$
      Then, Theorem \ref{Thm_Main} yields the assertion in the polytopal case. For a convex body, we conclude the formula by approximating it by polytopes, since for $i\in\{0,1\}$ the valuations $\TenCM{k}{r}{s - 2m}{i}$  have weakly continuous extensions to $\cK^n$.
			Finally, we note that $c_{n, j, k}^{s, 0, 1, 0} =0$.
    \end{proof}

  It is crucial that the right sides of the formulae in Theorem \ref{Cor_Main_l=1} and Theorem \ref{Cor_Main_l=0} only involve the tensorial curvature measures $\TenCM{k}{r}{s}{0}$ and $\TenCM{k}{r}{s}{1}$, which are the ones with weakly continuous extensions to $\cK^{n}$, and not $\TenCM{k}{r}{s}{i}$ with $i>1$.

  \section{Some auxiliary results}\label{secAux}

  		
      Before we start with the proof of the main theorem, we establish several auxiliary integral geometric results in this section.
						As a rule, these results hold for $n\ge 1$. If not stated otherwise, the case $n=1$ (or even $n=0$) can be checked directly.

We recall the following notions.
 The rotation group on $\R^{n}$ is denoted by $\SO(n)$, the orthogonal group on
    $\R^n$ by $\OO(n)$, and we write $\nu$ for the
    Haar probability measure on both spaces. By $\GOp(n, k)$ (resp.~$\AOp(n, k)$),
    for $k \in \{0, \ldots, n\}$, we denote the
    Grassmannian (resp.~affine Grassmannian)  of $k$-dimensional linear
		(resp.~affine) subspaces of $\R^{n}$. We
    write $\nu_k$  for the rotation  invariant Haar probability measure
    on $\GOp(n, k)$. The directional space of an affine $k$-flat $E \in
    \AOp(n, k)$ is denoted by $E^{0} \in \GOp(n, k)$ and its orthogonal complement
    by $E^{\perp} \in \GOp(n, n - k)$.
    For $k \in \{0, \ldots, n\}$ and $F \in \GOp(n, k)$, we denote the group of rotations of $\R^n$ mapping $F$ (and hence also $F^\perp$) into itself  by $\SO(F)$ (which is the same as $\SO(F^\perp)$) and write $\nu^{F}$ for the  Haar probability measure
    on $\SO(F)$. For $l \in \{0, \ldots, n\}$, let $\GOp(F, l) :=  \{ L \in \GOp(n, l): L \subset F \}$ if $l \leq k$, and let
		$\GOp(F, l) :=  \{ L \in \GOp(n, l): L \supset F \}$ if $l > k$.
    Then $\GOp(F, l)$ is a homogeneous $\SO(F)$-space.
    Hence, there exists a unique  Haar probability measure $\nu^{F}_l$ on $\GOp(F, l)$, which is $\SO(F)$ invariant.  An introduction to invariant measures and group operations as needed here is provided in \cite[Chap.~13]{SchnWeil08}, where, however, $\SO(F)$ is defined in a slightly different way.

    Recall that the orthogonal projection of a vector $x \in
    \R^{n}$ to a linear subspace $L$ of $\R^{n}$ is denoted by
    $p_{L}(x)$, and set $\pi_{L}(x):= p_{L}(x)/\|p_{L}(x)\|\in \Sn$ for $x \notin
    L^{\perp}$. For two linear subspaces $L, L'$ of $\R^{n}$, the
    subspace determinant $[L, L']$ is defined as follows (see \cite[Sect.~14.1]{SchnWeil08}). One
    extends an orthonormal basis of $L \cap L'$ (the empty set if  $L \cap L'=\{0\}$) to an orthonormal basis
    of $L$ and to one of $L'$. Then $[L, L']$ is the volume of the
    parallelepiped spanned by all these vectors. Consequently, if $L=\{0\}$ or $L=\R^n$, then
    $[L,L']:=1$. For $F,F'\in\cK^n$, we define
    $[F,F']:=[F^0,(F')^0]$, where $F^0$ is the direction space of the affine hull of $F$.

      A basic tool in this work is the following integral geometric transformation
			formula, which is a special case of \cite[Theorem 7.2.6]{SchnWeil08}.

      \begin{Lemma} \label{Lem_TraFo_SW}
        Let $0\le j \leq k \leq n$ be integers, $F \in \GOp(n,k)$, and let $f:\GOp(n,n-k+j) \rightarrow \R$ be integrable. Then
        \begin{align*}
          \int _{\GOp(n, n - k + j)} f(L) \, \nu_{n - k + j}(\intd L) = d_{n, j, k} \int _{\GOp(F, j)} \int _{\GOp(U, n - k + j)} [F, L]^j f(L) \, \nu_{n - k + j}^U(\intd L) \, \nu_{j}^F(\intd U)
        \end{align*}
        with
        \begin{align*}
          d_{n, j, k} := \prod_{i = 1}^{k - j} \frac {\Gamma(\frac {i} 2) \Gamma(\frac {n - k + j + i} 2)} {\Gamma(\frac {j + i} 2) \Gamma(\frac {n - k + i} 2)}.
        \end{align*}
      \end{Lemma}

      The preceding lemma yields the next result, which is again an integral geometric transformation formula (which will be needed in Section \ref{subsec5.3}). Here we
			(implicitly) require that $n\ge 2$.

      \begin{Lemma}[{\cite[Corollary 4.2]{HugSchnSchu08}}] \label{Lem_TraFo_HSS}
        Let $u \in \Sn$ and let $h: \GOp(n, k) \rightarrow \T$ be an integrable function, where  and $0 < k < n$. Then
        \begin{align*}
          \int _{\GOp(n, k)} h(L) \, \nu_k (\intd L) & = \frac {\omega_{k}} {2 \omega_{n}} \int _{\GOp(u^\perp, k - 1)} \int _{-1}^{1} \int _{U^{\perp} \cap u^{\perp} \cap \Sn} |t|^{k - 1} ( 1 - t^2 )^{\frac {n - k - 2} 2} \\
          & \qquad \times h \bigl( \Span \bigl\{ U, t u + \sqrt{1 - t^2} w \bigr\} \bigr) \, \cH^{n - k - 1} (\intd w) \, \intd t \, \nu^{u^\perp}_{k - 1} (\intd U).
        \end{align*}
      \end{Lemma}

      The following lemmas can be derived from Lemma \ref{Lem_TraFo_HSS} (see \cite[(24)]{SchnSchu02}, \cite[Lemma 4.3 and Proposition 4.5]{HugSchnSchu08}).

      \begin{Lemma}[{\cite[(24)]{SchnSchu02}}] \label{Lem_IntForm_Sph}
        Let $s \in \N_{0}$ and $n\ge 1$. Then
        \begin{equation*}
          \int _{\S} u^{s} \, \cH^{n - 1} (\intd u) = \1\{ s \text{ \rm even} \} \,
					2 \frac{\omega_{n + s}}{\omega_{s + 1}}\, Q^{\frac{s}{2}}.
        \end{equation*}
      \end{Lemma}

The next lemma is used in the proofs of Lemmas \ref{Lem_IntForm_sineMeT} and \ref{Lem_IntForm_MeTc} below.

      \begin{Lemma}[{\cite[Lemma 4.3]{HugSchnSchu08}}] \label{Lem_IntForm_MeT}
        Let $i, k \in \N_0$ with $k \leq n$ and $n\ge 1$. Then
        \begin{equation*}
          \int _{\GOp(n, k)} Q(L)^i \, \nu_{k} (\intd L) = \frac{\Gamma(\frac{n}{2})
					\Gamma(\frac{k}{2} + i)}{\Gamma(\frac{n}{2} + i)\Gamma(\frac{k}{2})} \,Q^{i}.
        \end{equation*}
      \end{Lemma}

The following lemma extends Lemma \ref{Lem_IntForm_MeT} (but the latter is used in the proof
of Lemma \ref{Lem_IntForm_sineMeT}). It
will be needed in the proof of Proposition \ref{Prop_Main}
(of which Lemma \ref{Lem_IntForm_sineMeT}  is a special case).

      \begin{Lemma}[{\cite[Proposition 4.5]{HugSchnSchu08}}] \label{Lem_IntForm_sineMeT}
        Let $a, i \in \N_0$, $k,r \in \{ 0, \ldots, n \}$ with $k + r \geq n\ge 1$,
				and let $F \in \GOp(n, r)$. Then
        \begin{align*}
          \int _{\GOp(n, k)} [F, L]^{a} Q(L)^{i} \, \nu_{k} (\intd L) & = e_{n, k, r, a}\,
					\frac {\Gamma( \frac {n + a} 2)} {\Gamma( \frac {n + a} 2 + i) \Gamma( \frac {k + a} 2)}
					\sum_{\beta = 0}^{i} (-1)^{\beta} \binom{i}{\beta} \Gamma(\tfrac{k + a}{2} + i - \beta) \\
          & \qquad \times \frac{\Gamma(\frac{n - k}{2} + \beta) \Gamma(\frac{a}{2} + 1)
					\Gamma(\frac{r}{2})}{\Gamma(\frac{n - k}{2}) \Gamma(\frac{a}{2} + 1 - \beta)
					\Gamma(\frac{r}{2} + \beta)}\, Q^{i - \beta} Q(F)^{\beta}
        \end{align*}
        with
        \begin{align*}
          e_{n, k, r, a} := \prod_{p = 0}^{n - r - 1} \frac {\Gamma(\frac {n - p} 2)
					\Gamma(\frac {k - p + a} 2)} {\Gamma(\frac {n - p + a} 2) \Gamma(\frac {k - p} 2)}.
        \end{align*}
      \end{Lemma}

\begin{proof}
Although this lemma is stated in \cite[Proposition 4.5]{HugSchnSchu08} only for $k,r\ge 1$, it is easy
to check that it remains true for $k=0$ (and $r=n$) and for $r=0$ (and $k=n$) with $n\ge 1$ as well as for $n=k=r=1$.
The only nontrivial case that has to be checked concerns the assertion for $k=0$, $r=n$ and $i\ge 1$,
where we have to show that the right
side is the zero tensor. For this we can assume that $a>0$, since the case $a=0$ is covered by
Lemma \ref{Lem_IntForm_MeT}. Up to irrelevant constants, the factor on the right side equals
\begin{align*}
&\sum_{\beta=0}^i(-1)^\beta\binom{i}{\beta} \Gamma(\tfrac{a}{2} + i - \beta)
          \frac{\Gamma(\frac{n }{2} + \beta)
					}{\Gamma(\frac{a}{2} + 1 - \beta)
					\Gamma(\frac{n}{2} + \beta)}\\
					&\qquad=(-1)^i\sum_{\beta=0}^i(-1)^\beta\binom{i}{\beta}
          \frac{\Gamma(\tfrac{a}{2} + \beta)
					}{\Gamma(\frac{a}{2} + 1-i + \beta)} =0,
\end{align*}
which follows from  relation \eqref{lemZeilmod}, since $\Gamma(1-i)^{-1}=0$ for $i\ge 1$.
\end{proof}

      From Lemma \ref{Lem_IntForm_MeT} we deduce the next result, which will
			be used in the proofs of Lemma \ref{Lem_IntForm_MeTVec} and Proposition \ref{Prop_Main}.

      \begin{Lemma} \label{Lem_IntForm_MeTc}
        Let $i, j, k \in \N_0$ with $0\le k \leq n$. Then
        \begin{equation*}
          \int _{\GOp(n, k)} Q(L)^i Q(L^{\perp})^j \, \nu_{k} (\intd L) =
					\frac{\Gamma(\frac{n}{2}) \Gamma(\frac{k}{2} + i)
					\Gamma(\frac{n - k}{2} + j)}{\Gamma(\frac{n}{2} + i + j) \Gamma(\frac{k}{2})
					\Gamma(\frac{n - k}{2})} \,Q^{i + j}.
        \end{equation*}
      \end{Lemma}

      \begin{proof}
			The cases where $k\in\{0,n\}$ can be checked easily by distinguishing whether $i,j=0$ or not. Hence
			we can assume that $1\le k\le n-1$.
        Let $I$ denote the integral we are interested in. By expansion of
				$Q(L^{\perp})^{j} = (Q - Q(L))^{j}$ and Lemma \ref{Lem_IntForm_MeT} we obtain
        \begin{align*}
          I & = \sum_{l = 0}^{j} (-1)^{l} \binom{j}{l} Q^{j - l} \int _{\GOp(n, k)}
					Q(L)^{i + l} \, \nu_{k} (\intd L) \\
          & = \frac{\Gamma(\frac{n}{2})}{\Gamma(\frac{k}{2})} \sum_{l = 0}^{j} (-1)^{l}
					\binom{j}{l} \frac{\Gamma(\frac{k}{2} + i + l)}{\Gamma(\frac{n}{2} + i + l)} \,Q^{i + j}.
        \end{align*}
        Then  relation \eqref{lemZeilmod} yields the assertion.
      \end{proof}

The next lemma will be used at the beginning of Section \ref{subsec5.2}.

      \begin{Lemma} \label{Lem_IntForm_MeTVec}
        Let $j, l, s \in \N_{0}$ with $j < n$, $L \in \GOp(n, j)$ and $u \in L^\perp \cap \S$. Then
        \begin{align*}
        	\int_{\SO(n)} Q(\vartheta L)^l (\vartheta u)^{s} \, \nu(\intd\vartheta) = \frac{\Gamma(\frac {n} {2}) \Gamma(\frac {j} {2} + l) \Gamma(\frac {s + 1} {2})}{\sqrt \pi \Gamma(\frac {n + s} {2} + l) \Gamma(\frac {j} {2})}\, Q^{l + {\frac s 2}},
        \end{align*}
        if $s$ is even. The same relation holds if the integration is extended over $\OO(n)$.
				If $s$ is odd and $n\ge 2$ (or $n=1$ and the integration is extended over $\OO(1)$), then the integral
				vanishes.
      \end{Lemma}

      \begin{proof}
			The case $n=1$, $j=0$ is checked directly by distinguishing $l=0$ or $l\neq 0$. Hence let $n\ge 2$.
       Let $I$ denote the integral we are interested in. Due to symmetry, $I = 0$ if $s$ is odd. Therefore, let $s$ be even. Let $\rho\in\SO(L^\perp)$. Then, by the right invariance of $\nu$ and Fubini's theorem we obtain
        \begin{align*}
          I & = \int_{\SO(n)}Q(\vartheta\rho L)^l (\vartheta\rho u)^{s}\, \nu(\intd\vartheta) \\
          & = \int_{\SO(L^\perp)}\int_{\SO(n)}Q(\vartheta\rho L)^l (\vartheta\rho u)^{s}\, \nu(\intd\vartheta)\, \nu^{L^{\perp}}(\intd\rho) \\
          & = \int_{\SO(L^\perp)}\int_{\SO(n)}Q(\vartheta L)^l (\vartheta\rho u)^{s}\, \nu(\intd\vartheta)\, \nu^{L^{\perp}}(\intd\rho) \\
          & = \int_{\SO(n)}Q(\vartheta L)^l\vartheta \int_{\SO(L^\perp)} (\rho u)^{s}\, \nu^{L^{\perp}}(\intd\rho)\, \nu(\intd\vartheta).
        \end{align*}
       Lemma \ref{Lem_IntForm_Sph}, applied in $L^\perp$ with $\text{dim}(L^\perp)\ge 1$, yields
        \begin{align*}
          \int_{\SO(L^\perp)} (\rho u)^{s}\, \nu^{L^{\perp}}(\intd\rho)  = \frac 1 {\omega_{n - j}} \int_{\S \cap L^\perp} v^{s} \, \cH^{n - j - 1} (\intd v) = 2 \frac {\omega_{n - j + s}} {\omega_{s + 1} \omega_{n - j}} \,Q(L^\perp)^{\frac s 2},
        \end{align*}
        and hence we get
        \begin{align}
          I & = 2 \frac {\omega_{n - j + s}} {\omega_{s + 1} \omega_{n - j}} \int_{\SO(n)}Q(\vartheta L)^l Q(\vartheta L^\perp)^{\frac s 2} \, \nu(\intd\vartheta) \nonumber \\
          & = 2 \frac {\omega_{n - j + s}} {\omega_{s + 1} \omega_{n - j}} \int_{\GOp (n, j)} Q(U)^l Q(U^\perp)^{\frac s 2} \, \nu_j(\intd U). \nonumber
        \end{align}
        From Lemma \ref{Lem_IntForm_MeTc} we conclude that
        \begin{align*}
          I = 2 \frac {\omega_{n - j + s}} {\omega_{s + 1} \omega_{n - j}} \frac {\Gamma(\frac {n} {2}) \Gamma(\frac {j} {2} + l) \Gamma(\frac {n - j + s} {2})} {\Gamma(\frac {n + s} {2} + l) \Gamma(\frac {j} {2}) \Gamma(\frac {n - j} {2})}\, Q^{l + {\frac s 2}},
        \end{align*}
         and thus we obtain the assertion.
      \end{proof}

The following lemma will be required in Section \ref{subsec5.3}.

      \begin{Lemma} \label{Lem_IntForm_VecScal}
        Let $u, v \in \mathbb{S}^{n-1}$, $i, t \in \N_0$ and $n\ge 1$. Then
        \begin{align*}
          \int_{\SO(n)} (\rho v)^{i} \langle u, \rho v \rangle^t \, \nu(\intd \rho)
					= \frac {\Gamma(\frac n 2) \Gamma(t + 1)} {2^t \sqrt \pi \Gamma(
					\frac {n + i + t} 2)} \sum_{x = (\frac {i - t} 2)^+}^{\lfloor
					\frac i 2 \rfloor} \binom{i}{2x} \frac {\Gamma(x + \frac 1 2)} {\Gamma(\frac{t - i} 2 + x + 1)} u^{i - 2x} Q^x,
        \end{align*}
        if $i + t$ is even. The same relation holds if the integration is extended over $\OO(n)$.
				If $i + t$ is odd and $n\ge 2$ (or $n=1$ and the integration is extended over $\OO(1)$),
				then the integral on the left side vanishes.
      \end{Lemma}

      \begin{proof} First, we assume that $n\ge 2$.
       Let $I$ denote the integral we are interested in. By symmetry,
				$I = 0$ if $i + t$ is odd. Thus, in the following
				we assume that $i + t$ is even. Applying the transformation
        \begin{align*}
          f: [-1,1] \times (\S \cap u^\perp) \rightarrow \S, \, (z, w) \mapsto z u + \sqrt{1 - z^2} w,
        \end{align*}
        with Jacobian $\cJ f(z,w) = \sqrt{1-z^2}^{n - 3}$ to the integral $I$, we get
        \begin{align*}
          I & = \frac 1 {\omega_{n}} \int_{\S} v^{i} \langle u, v \rangle^t \, \cH^{n - 1} (\intd v) \\
          & = \frac 1 {\omega_{n}} \int_{-1}^1 \int_{\S \cap u^\perp}
					\left( 1 - z^2 \right)^{\frac {n - 3} 2} \left( z u +
					\sqrt{1 - z^2} w \right)^{i} {{\left\langle  u, z u +
					\sqrt{1 - z^2} w  \right\rangle}}^{t} \, \cH^{n - 2}(\intd w) \, \intd z.
        \end{align*}
        Binomial expansion of $(zu + \sqrt{1 - z^2}w)^i$ yields
        \begin{align*}
          I & = \frac 1 {\omega_{n}} \sum_{m = 0}^{i} \binom{i}{m} u^{i - m}
					\underbrace{\int_{-1}^1 z^{t + i - m} \left( 1 - z^2 \right)^{\frac {n + m - 3} 2}
					\, \intd z}_{ = \1 \{ m \text{ even} \} B(\frac {t + i - m + 1} 2, \frac {n + m - 1} 2) }
					\underbrace{\int_{\S \cap u^\perp} w^{m}\, \cH^{n - 2}(\intd w)}_{ =: I' },
        \end{align*}
				where  $B(\cdot,\cdot)$ denotes the Beta function.
        From   Lemma \ref{Lem_IntForm_Sph}, we obtain
        \begin{align*}
          I' & = \1 \{m \text{ even} \} 2 \frac {\omega_{n + m - 1}} {\omega_{m + 1}} \,Q(u^\perp)^{\frac {m} 2},
        \end{align*}
        and thus
        \begin{align*}
          I & = \sum_{m = 0}^{\lfloor \frac i 2 \rfloor} \binom{i}{2m} \frac {\Gamma(\frac {t + i + 1} 2 - m)
					\Gamma(\frac {n - 1} 2 + m)} {\Gamma(\frac {n + i + t} 2)}
					{\frac {2 \omega_{n + 2m - 1}} {\omega_{n} \omega_{2m + 1}}} u^{i - 2m} Q(u^\perp)^{m} \\
          & = \frac {\Gamma(\frac {n} 2)} {\pi \Gamma(\frac {n + i + t} 2)} \sum_{m = 0}^{\lfloor \frac i 2 \rfloor}
					\binom{i}{2m} \Gamma(m + \tfrac {1} 2) \Gamma(\tfrac {t + i + 1} 2 - m) u^{i - 2m} Q(u^\perp)^{m}.
        \end{align*}
        Since $Q(u^\perp) = Q - u^2$, binomial expansion yields
        \begin{align*}
          Q(u^\perp)^{m} & = \sum_{x = 0}^{m} (-1)^{m + x} \binom{m}{x} u^{2m - 2x} Q^{x}.
        \end{align*}
        Legendre's duplication formula gives
        \begin{align*}
          \binom{i}{2m} \binom{m}{x} \Gamma(m + \tfrac {1} 2)
            & = \binom{i}{2x} \Gamma(x + \tfrac{1}{2}) \frac{1}{(m - x)!} \frac{\Gamma(\frac{i + 1}{2} - x) \Gamma(\frac{i}{2} - x + 1)}{\Gamma(\frac{i + 1}{2} - m) \Gamma(\frac{i}{2} - m + 1)} \\
            & = \binom{i}{2x} \Gamma(x + \tfrac{1}{2}) \binom{\lfloor \frac i 2 \rfloor - x}{m - x} \frac{\Gamma(\lfloor \frac{i + 1}{2} \rfloor - x + \frac{1}{2})}{\Gamma(\lfloor \frac{i + 1}{2} \rfloor - m + \frac{1}{2})},
        \end{align*}
        and thus we obtain by a change of the order of summation
        \begin{align*}
          I & = \frac {\Gamma(\frac {n} 2)} {\pi \Gamma(\frac {n + i + t} 2)} \sum_{x = 0}^{\lfloor \frac i 2 \rfloor} \binom{i}{2x} \Gamma(x + \tfrac{1}{2})\Gamma(\lfloor \tfrac{i + 1}{2} \rfloor - x + \tfrac{1}{2}) u^{i - 2x} Q^{x} \\
          & \qquad \qquad \qquad \qquad \times \sum_{m = x}^{\lfloor \frac i 2 \rfloor} (-1)^{m + x} \binom{\lfloor \frac i 2 \rfloor - x}{m - x} \frac{\Gamma(\tfrac {t + i + 1} 2 - m)}{\Gamma(\lfloor \frac{i + 1}{2} \rfloor - m + \frac{1}{2})}.
        \end{align*}
        We denote the sum with respect to $m$ by $S_1$. An index shift by $x$, applied to $S_1$, yields
        \begin{align*}
          S_1 & = \sum_{m = 0}^{\lfloor \frac i 2 \rfloor - x} (-1)^{m} \binom{\lfloor \frac i 2 \rfloor - x}{m} \frac{\Gamma(\tfrac {t + i + 1} 2 - x - m)}{\Gamma(\lfloor \frac{i + 1}{2} \rfloor - x - m + \frac{1}{2})}.
        \end{align*}
        Now we conclude from relation \eqref{lemZeilmod} that
        \begin{align*}
          S_1 & = (-1)^{\lfloor \frac i 2 \rfloor - x} \sum_{m = 0}^{\lfloor \frac i 2 \rfloor - x}
              (-1)^{m}
              \binom{\lfloor \frac i 2 \rfloor - x}{m}
              \frac {\Gamma(\frac {t + i + 1} 2 - \lfloor \frac i 2 \rfloor + m)} {\Gamma(\lfloor \frac{i + 1}{2} \rfloor - \lfloor \frac i 2 \rfloor + m + \frac{1}{2})} \\
            & = (-1)^{\lfloor \frac i 2 \rfloor - x} \frac {\Gamma(\frac {t + i + 1} 2 - \lfloor \frac i 2 \rfloor) \Gamma(\overbrace{\textstyle \lfloor \frac{i + 1}{2} \rfloor + \lfloor \frac i 2 \rfloor}^{ = i} - \frac {t + i + 1} 2 - x + \frac{1}{2})} {\Gamma(\lfloor \frac{i + 1}{2} \rfloor - x + \frac{1}{2}) \Gamma(\lfloor \frac{i + 1}{2} \rfloor - \frac {t + i + 1} 2 + \frac{1}{2})} \\
            & = \overbrace{(-1)^{i + \lfloor \frac{i + 1}{2} \rfloor + \lfloor \frac i 2 \rfloor}}^{ = (-1)^{2i} = 1} \frac {\overbrace{\Gamma(\tfrac {t + i + 1} 2 - \lfloor \tfrac i 2 \rfloor) \Gamma( \tfrac {t + i + 1} 2 - \lfloor \tfrac{i + 1}{2} \rfloor + \tfrac{1}{2})}^{ \scriptstyle = \Gamma(\frac {t + 1} 2) \Gamma(\frac {t} 2 + 1)}} {\Gamma(\lfloor \frac{i + 1}{2} \rfloor - x + \frac{1}{2}) \Gamma(\frac {t - i} 2 + x + 1)},
        \end{align*}
        where we used \eqref{Form_Gam_Cont} with $c = \frac {t + i + 1} 2 -
				\lfloor \frac{i + 1}{2} \rfloor - \frac{1}{2}\in \N_{0}$ and
				$ m = i - \lfloor \frac{i + 1}{2} \rfloor - x \in \N_{0}$.
				We notice that $S_1 = 0$ if $x < \frac {i - t} 2 $. Thus we obtain the
				assertion by another application of Legendre's duplication formula.
				
				It remains to confirm the assertion if $n=1$ and $i+t$ is even (all other assertions are easy to check).
				In this case, $u=\pm v$ and therefore the left-hand side of the asserted equation equals $(\pm 1)^t v^i$.
				Using first Legendre's duplication formula repeatedly, then relation \eqref{lemZeilfund}, and finally again
				Legendre's duplication formula, we see that
				the right-hand side equals
				\begin{align*}
				&\frac { \Gamma(t + 1)} {2^t  \Gamma(
					\frac {1 + i + t} 2)} \sum_{x =0}^{\lfloor
					\frac i 2 \rfloor} \binom{i}{2x} \frac {\Gamma(x + \frac 1 2)} {\Gamma(\frac{t - i} 2 + x + 1)} (\pm 1)^iv^{i} \\
					&\qquad = \frac { \Gamma(t + 1)} {2^t  \Gamma(
					\frac {1 + i + t} 2)}\sqrt\pi \Gamma(\lfloor\tfrac{i+1}{2}\rfloor+\tfrac{1}{2} )
				\sum_{x =0}^{\lfloor
					\frac i 2 \rfloor}\binom{\lfloor\frac{i}{2}\rfloor}{x}\frac{1}
					{\Gamma(\frac{t-i}{2}+1+x)\Gamma(\lfloor\frac{i+1}{2}\rfloor+\frac{1}{2}-x)}\\
					&\qquad =\frac { \Gamma(t + 1)} {2^t  \Gamma(
					\frac {1 + i + t} 2)}\sqrt\pi 
					\frac{\Gamma(\lfloor\tfrac{i+1}{2}\rfloor+\tfrac{1}{2}+\tfrac{t-i}{2}+1+ \lfloor\tfrac{i}{2}
					 \rfloor-1 )}{ 
					\Gamma (
					\tfrac{t-i}{2}+1+ \lfloor\tfrac{i}2 \rfloor ) \Gamma (
					\tfrac{t-i}{2}+ \lfloor\tfrac{i+1}2 \rfloor+\tfrac{1}{2})}(\pm 1)^iv^{i} = (\pm 1)^iv^{i} ,
				\end{align*}
				which confirms the assertion.
      \end{proof}

The next lemma will be useful in the proof of Proposition \ref{Prop_Main}.
    	
      \begin{Lemma} \label{Lem_Trafo_G(U,j)}
        Let $j, k, n \in \N_0$ with $j + k \leq n$, $n\ge 1$, and $U \in \GOp(n, j)$. Then,
				for any integrable function $f: \GOp(U, j + k) \rightarrow \R$,
        \begin{align*}
          \int_{\GOp(U^\perp,k)} f(U + L) \, \nu^{U^\perp}_k(\intd L) = \int_{\GOp(U, j + k)} f(L) \, \nu^{U}_{j + k}(\intd L).
        \end{align*}
      \end{Lemma}

      \begin{proof}
			We consider the map  $H: \GOp(U^\perp, k)  \to \GOp(U, j + k)$, $ L \mapsto U + L$, which is well-defined, since
         $\dim (U \cap L) = 0$ and hence $\dim (L + U) = j + k$ for all $L \in \GOp({U^\perp}, k)$. It is sufficient to show
				that $H(\nu^{U^\perp}_k) = \nu^{U}_{j + k} $, where $H(\nu^{U^\perp}_k)$ is
				the image measure  of $\nu^{U^\perp}_k$ under $H$. Since $H(\nu^{U^\perp}_k)$ and $\nu^{U}_{j + k}$ are probability measures,
				and $\nu^{U}_{j + k}$ is $\SO(U^\perp)$ invariant by definition,
				it is sufficient to show that  $H(\nu^{U^\perp}_k)$ is $\SO(U^\perp)$ invariant.

        To verify this, choose $A \in \cB(\GOp(U, j + k))$ and $\vartheta \in \SO(U^\perp)$. Then we obtain
        \begin{align*}
          H(\nu^{U^\perp}_k) (\vartheta A) & = \nu^{U^\perp}_k \left(  \{ L \in \GOp({U^\perp}, k) : U+L \in \vartheta A \}  \right) \\
          & = \nu^{U^\perp}_k \big( \{ L \in \GOp({U^\perp}, k) :  {\vartheta^{-1} U}  + \vartheta^{-1} L \in A \} \big).
        \end{align*}
        The $\SO(U^\perp)$ invariance of $\nu^{U^\perp}_k$ yields
        $$
          H(\nu^{U^\perp}_k) (\vartheta A)  = \nu^{U^\perp}_k \big( \{ L \in \GOp({U^\perp}, k) : U + L \in A \} \big)
           = H(\nu^{U^\perp}_k) ( A ),
       $$
        which completes the argument.
      \end{proof}

     The following proposition, which is a generalization
		of Lemma \ref{Lem_IntForm_sineMeT} in the case $a=2$, will be applied at the
		end of Section \ref{subsec5.3}. Its proof uses several of the lemmas provided above.
			
      \begin{Proposition} \label{Prop_Main}
        Let $F \in \GOp(n, k)$ with $0\le j \leq k \le n$ and $m,l\in\N_0$. Then
        \begin{align*}
          & \int_{\GOp(n, n - k + j)} [F, L]^2 Q(L)^m Q(F \cap L)^l \, \nu_{n - k + j} (\intd L) \\
          & \qquad = \frac {(n - k + j)! k!} {n! j!} \frac {\Gamma(\frac {n} 2 + 1)
					\Gamma(\frac {j} 2 + l) \Gamma(\frac{k} 2)}{\Gamma(\frac {n} 2 + m + 1)
					\Gamma(\frac{j} 2) \Gamma(\frac {k - j} 2 ) \Gamma(\frac {n - k + j} 2 + 1)} \\
          & \qquad \qquad \times \sum_{i = 0}^{m} \binom{m}{i} \frac {(l + i - 2)!} {(l - 2)!}
					\frac {\Gamma(\frac {k - j} 2 + i) \Gamma(\frac {n - k + j} 2 + m - i + 1)}
					{\Gamma(\frac {k} 2 + l + i)}  Q^{m - i} Q(F)^{l + i}.
        \end{align*}
      \end{Proposition}

      For $l \leq 1$, the factor $\frac {(l + i - 2)!} {(l - 2)!}$ in Proposition \ref{Prop_Main}
			is read as stated in \eqref{Form_Gam_Cont} and discussed in Section \ref{Tmr}. Moreover,
			$\Gamma(l+{j}/ 2 )/\Gamma({j}/{2})$ is zero if $j=0,l\neq0$ and one if $j=l=0$.

      \begin{proof}
        Let $I$ denote the integral in which we are interested. If $j = k$, all summands on  the right
				side of the asserted equation are zero except for $i=0$. Thus it is easy to confirm the assertion.
				Now assume that $0\le j<k\le n$, hence $n\ge 1$.
				If $j=l=0$, then the assertion follows as a special case of Lemma \ref{Lem_IntForm_sineMeT}.
				If $j=0,l\neq0$ then both sides of the asserted equation are zero.

				In the following, we consider the remaining cases where $0<j< k$. Then Lemma \ref{Lem_TraFo_SW} yields
        \begin{align*}
          I = d_{n, j, k} \int_{\GOp(F, j)} \int_{\GOp(U, n - k + j)} [F, L]^{j + 2} Q(L)^m Q(F \cap L)^l \, \nu^{U}_{n - k + j} (\intd L) \, \nu^F_{j} (\intd U).
        \end{align*}
        For fixed $U\in \GOp(F, j)$, we have
				  $\dim (F \cap L) = j = \dim U$ for $\nu^U_{n-k+j}$-almost all $L\in \GOp(U, n - k + j)$ and $U \subset F \cap L$, hence $U = F \cap L$, and therefore
        \begin{align*}
           I= d_{n, j, k} \int_{\GOp(F, j)} Q(U)^l \int_{\GOp(U, n - k + j)} [F, L]^{j + 2} Q(L)^m \, \nu^{U}_{n - k + j} (\intd L) \, \nu^F_{j} (\intd U).
        \end{align*}
        An application of Lemma \ref{Lem_Trafo_G(U,j)} shows that
        \begin{align*}
           I= d_{n, j, k} \int_{\GOp(F, j)} Q(U)^l \int_{\GOp(U^\perp, n - k)}
					[F, U + L]^{j + 2} Q(U + L)^m \, \nu^{U^\perp}_{n - k} (\intd L) \, \nu^F_{j} (\intd U).
        \end{align*}
        As $U \subset F$ and $L \subset U^\perp$, we have
        \begin{align*}
          [F, U + L] = [F \cap U^\perp, L]^{(U^\perp)}
        \end{align*}
        and
        \begin{align*}
          Q(U + L)^m = \big( Q(U) + Q(L) \big)^m = \sum_{\alpha = 0}^{m} \binom{m}{\alpha} Q(L)^\alpha Q(U)^{m - \alpha}.
        \end{align*}
        Thus we obtain
        \begin{align*}
          I & = d_{n, j, k} \sum_{\alpha = 0}^{m} \binom{m}{\alpha} \int_{\GOp(F, j)} Q(U)^{l + m - \alpha} \\
          & \qquad \times \int_{\GOp(U^\perp, n - k)} \left( [F \cap U^\perp, L]^{(U^\perp)} \right)^{j + 2} Q(L)^\alpha \, \nu^{U^\perp}_{n - k} (\intd L) \, \nu^F_{j} (\intd U).
        \end{align*}
				Observe that $\text{dim}(U^\perp)=n-j>n-k\ge 0$, hence $\text{dim}(U^\perp)\ge 1$.
       Therefore Lemma \ref{Lem_IntForm_sineMeT} can be used to see that the integral with respect to $L$ can be expressed as
        \begin{align*}
          &e_{n - j, n - k, k - j, j + 2} \frac {\Gamma(\frac {n} 2 + 1)}{\Gamma(\frac {n - k + j} 2 + 1) \Gamma(\frac {n} 2 + 1 + \alpha)} \sum_{\beta = 0}^{\alpha} (-1)^{\beta} \binom{\alpha}{\beta} \\
          & \qquad \times \frac{\Gamma(\frac{n - k + j} 2 + 1 + \alpha - \beta) \Gamma(\frac{k - j} 2 + \beta) \Gamma(\frac{j} 2 + 2) \Gamma(\frac{k - j} 2)} {\Gamma(\frac{k - j} 2) \Gamma(\frac{j} 2 + 2 - \beta) \Gamma(\frac{k - j} 2 + \beta)} Q(U^\perp)^{\alpha - \beta} Q(F \cap U^\perp)^{\beta},
        \end{align*}
        and thus
        \begin{align*}
          I & = d_{n, j, k} e_{n - j, n - k, k - j, j + 2} \frac {\Gamma(\frac {n} 2 + 1)}{\Gamma(\frac {n - k + j} 2 + 1)} \\
          & \qquad \times \sum_{\alpha = 0}^{m} \sum_{\beta = 0}^{\alpha} (-1)^{\beta} \binom{m}{\alpha} \binom{\alpha}{\beta} \frac{\Gamma(\frac{n - k + j} 2 + 1 + \alpha - \beta) \Gamma(\frac{j} 2 + 2)} {\Gamma(\frac {n} 2 + 1 + \alpha) \Gamma(\frac{j} 2 + 2 - \beta)} \\
          & \qquad \times \int_{\GOp(F, j)} Q(U)^{l + m - \alpha} Q(U^\perp)^{\alpha - \beta}
					Q(F \cap U^\perp)^{\beta} \, \nu^F_{j} (\intd U).
        \end{align*}
        Observing cancellations and using Legendre's duplication formula, we get
		$$
				d_{n, j, k} e_{n - j, n - k, k - j, j + 2}=\frac {(n - k + j)! k!} {n! j!}.
				$$
Expanding $Q(U^{\perp})^{\alpha - \beta} = (Q - Q(U))^{\alpha - \beta}$, we obtain
        \begin{align*}
          I & = \frac {(n - k + j)! k!} {n! j!} \frac {\Gamma(\frac {n} 2 + 1) \Gamma(\frac{j} 2 + 2)}{\Gamma(\frac {n - k + j} 2 + 1)} \sum_{\alpha = 0}^{m} \sum_{\beta = 0}^{\alpha} \sum_{i = 0}^{\alpha - \beta} (-1)^{\alpha + i} \binom{m}{\alpha} \binom{\alpha}{\beta} \binom{\alpha - \beta}{i} \\
          & \qquad \times \frac{\Gamma(\frac{n - k + j} 2 + 1 + \alpha - \beta)} {\Gamma(\frac {n} 2 + 1 + \alpha) \Gamma(\frac{j} 2 + 2 - \beta)} Q^{i} \int_{\GOp(F, j)} Q(U)^{l + m - \beta - i} Q(F \cap U^\perp)^{\beta} \, \nu^F_{j} (\intd U).
        \end{align*}
        Lemma \ref{Lem_IntForm_MeTc}, applied in $F$, yields
        \begin{align*}
          I & = \frac {(n - k + j)! k!} {n! j!} \frac {\Gamma(\frac {n} 2 + 1) \Gamma(\frac{k}{2}) \Gamma(\frac{j} 2 + 2)}{\Gamma(\frac {n - k + j} 2 + 1) \Gamma(\frac{j}{2}) \Gamma(\frac{k - j}{2})} \sum_{\alpha = 0}^{m} \sum_{\beta = 0}^{\alpha} \sum_{i = 0}^{\alpha - \beta} (-1)^{\alpha + i} {\binom{m}{\alpha} \binom{\alpha}{\beta} \binom{\alpha - \beta}{i}} \\
          & \qquad \times \frac{\Gamma(\frac{n - k + j} 2 + 1 + \alpha - \beta)} {\Gamma(\frac {n} 2 + 1 + \alpha)} \frac{\Gamma(\frac{j}{2} + l + m - \beta - i) \Gamma(\frac{k - j}{2} + \beta)}{\Gamma(\frac{k}{2} + l + m - i) \Gamma(\frac{j} 2 + 2 - \beta)} Q^{i} Q(F)^{l + m - i}.
        \end{align*}
        Using the relation
				$$\binom{m}{\alpha} \binom{\alpha}{\beta} \binom{\alpha - \beta}{i}
				 = \binom{m}{i} \binom{m - i}{\beta} \binom{m - i - \beta}{\alpha - i - \beta}
				$$
				and by a change of the order of summation, we conclude that
        \begin{align*}
          I & = \frac {(n - k + j)! k!} {n! j!} \frac {\Gamma(\frac {n} 2 + 1) \Gamma(\frac{k}{2}) \Gamma(\frac{j} 2 + 2)}{\Gamma(\frac {n - k + j} 2 + 1) \Gamma(\frac{j}{2}) \Gamma(\frac{k - j}{2})} \sum_{i = 0}^{m} \binom{m}{i} Q^{i} Q(F)^{l + m - i} \\
          & \qquad \times \frac{1}{\Gamma(\frac{k}{2} + l + m - i)} \sum_{\beta = 0}^{m - i} \binom{m - i}{\beta} \frac{\Gamma(\frac{j}{2} + l + m - \beta - i) \Gamma(\frac{k - j}{2} + \beta)}{ \Gamma(\frac{j} 2 + 2 - \beta)} \\
          & \qquad \times \sum_{\alpha = i + \beta}^{m} (-1)^{\alpha + i} \binom{m - i - \beta}{\alpha - i - \beta} \frac{\Gamma(\frac{n - k + j} 2 + 1 + \alpha - \beta)} {\Gamma(\frac {n} 2 + 1 + \alpha)} .
        \end{align*}
        For the sum with respect to $\alpha$ we obtain from   relation \eqref{lemZeilmod} that
        \begin{align*}
           &  \sum_{\alpha = 0}^{m - i - \beta} (-1)^{\alpha + \beta} \binom{m - i - \beta}{\alpha} \frac{\Gamma(\frac{n - k + j} 2 + i + 1 + \alpha)} {\Gamma(\frac {n} 2 + i + \beta + 1 + \alpha)} \\
          & \qquad = (-1)^{\beta} \frac{\Gamma(\frac{n - k + j} 2 + i + 1) \Gamma(\frac {k - j} 2 + m - i)} {\Gamma(\frac {n} 2 + m + 1) \Gamma(\frac {k - j} 2 + \beta)}.
        \end{align*}
        Next, for  the resulting sum with respect to $\beta$, we obtain again from  relation \eqref{lemZeilmod} that
        \begin{align*}
           &  \sum_{\beta = 0}^{m - i} (-1)^{m + i + \beta} \binom{m - i}{\beta} \frac{\Gamma(\frac{j}{2} + l + \beta)}{\Gamma(\frac{j} 2 + 2 - m + i + \beta)} \\
          & \qquad = (-1)^{m + i} \frac{\Gamma(\frac{j}{2} + l) \Gamma(2 - l)}{\Gamma(\frac{j} 2 + 2) \Gamma(2 - l - m + i)} \\
          & \qquad = \frac{\Gamma(\frac{j}{2} + l) \Gamma(l + m - i - 1)}{\Gamma(\frac{j} 2 + 2) \Gamma(l - 1)},
        \end{align*}
				where we used $j>0$ for the first equation and \eqref{Form_Gam_Cont} for the second (and distinguished the cases $l=0$, $l=1$, $l\ge 2$).
       Thus we get
        \begin{align*}
          I & = \frac {(n - k + j)! k!} {n! j!} \frac {\Gamma(\frac {n} 2 + 1) \Gamma(\frac{j}{2} + l) \Gamma(\frac{k}{2})}{\Gamma(\frac {n} 2 + m + 1) \Gamma(\frac{j}{2}) \Gamma(\frac{k - j}{2}) \Gamma(\frac {n - k + j} 2 + 1)} \\
          & \qquad \times \sum_{i = 0}^{m} \binom{m}{i} \frac{(l + i - 2)!}{(l - 2)!} \frac{\Gamma(\frac {k - j} 2 + i) \Gamma(\frac{n - k + j} 2 + m - i + 1)} {\Gamma(\frac{k}{2} + l + i)} Q^{m - i} Q(F)^{l + i},
        \end{align*}
        where we reversed the order of summation.
      \end{proof}

    \section{Proof of  Theorem \ref{Thm_Main}}\label{secProof}
		
		We divide the proof into several steps. First, we treat the translative part of the kinematic integral. This can be done similarly
		as in the proof of the translative integral formula for curvature measures. Then we consider two ``boundary cases'' separately.
		The main and third step requires  the explicit calculation of a rotational integral for a tensor-valued function.
		Once this is accomplished,
		the proof is  finished, except for the asserted description of the coefficients,
		which at this point are still given in terms of iterated sums of products of
		binomial coefficients and Gamma functions. In a final step, these coefficients are then shown to have the simple form
		provided in the statement of the theorem.
		
		\subsection{The translative part}\label{subsec5.1}

The case $j=n$ is easy to check directly (since then $s=0$). Hence we may assume that $j\le n-1$ in the following.
   Let $I_{1}$ denote the integral in which we are interested. We start by decomposing the measure $\mu$ and by substituting
		the definition of $\TenCM{j}{r}{s}{l}$ for polytopes to get
    \begin{align*}
      I_1 & = \int_{\GOp_n} \TenCM{j}{r}{s}{l} (P \cap g P', \beta \cap g \beta') \, \mu( \intd g) \\
      & = \int_{\SO(n)} \int_{\R^n} \TenCM{j}{r}{s}{l} (P \cap (\vartheta P' + t), \beta \cap (\vartheta \beta' + t)) \, \cH^n(\intd t) \, \nu( \intd \vartheta) \\
      & = c_{n, j}^{r, s, l}\frac{1}{\omega_{n - j}}  \int_{\SO(n)} \int_{\R^n} \sum_{G \in \cF_j(P \cap (\vartheta P' + t))} Q(G)^l \int_{G \cap \beta \cap (\vartheta \beta' + t)} x^r \, \cH^j(\intd x) \\
      & \qquad \qquad \qquad \qquad \qquad \times \int_{N(P \cap (\vartheta P' + t), G) \cap \S} u^s \, \cH^{n - j - 1} (\intd u) \, \cH^n(\intd t) \, \nu( \intd \vartheta).
    \end{align*}
    Let $\vartheta \in \SO(n)$ be fixed for the moment. Neglecting a set of translations $t\in\R^n$ of measure zero, we
		can assume that the following is true (see \cite[p.~241]{Schneider14}).
		For every $j$-face $G \in \cF_j(P \cap (\vartheta P' + t))$
		there are a unique $k \in \{j, \ldots,  n\}$, a unique  $F \in \cF_k(P)$ and a unique
		 $F' \in \cF_{n - k + j}( P')$ such that $G = F \cap (\vartheta F'+t)$.
		Conversely, for every $k \in \{j, \ldots,  n\}$, every $F \in \cF_k(P)$ and every
		  $F' \in \cF_{n - k + j}(  P'  )$, we have $F \cap (\vartheta F'+t)\in \cF_j(P \cap (\vartheta P' + t))$ for almost all
			$t\in\R^n$ such that $F \cap (\vartheta F'+t)\neq \emptyset$.
		 Hence, we get
    \begin{align}
      I_1 & = c_{n, j}^{r, s, l}\frac{1}{\omega_{n - j}}  \int_{\SO(n)} \sum_{k = j}^n \sum_{F \in \cF_k(P)} \sum_{F' \in \cF_{n - k + j}(P')} Q(F^0 \cap (\vartheta F')^0)^l \nonumber \\
      & \qquad \qquad \qquad \times \int_{N(P \cap (\vartheta P' + t), F \cap (\vartheta F' + t)) \cap \S} u^s \, \cH^{n - j - 1} (\intd u) \nonumber \\
      & \qquad \qquad \qquad \times \int_{\R^n} \int_{F \cap (\vartheta F' + t) \cap \beta \cap (\vartheta \beta' + t)} x^r \, \cH^j(\intd x)\, \cH^n(\intd t) \, \nu( \intd \vartheta), \label{form_I1_mit_I2}
    \end{align}
    where we use that the integral with respect to $u$ is independent of the choice of a vector $t \in \R^n$ such that $\relint F \cap \relint (\vartheta F' + t) \neq \emptyset$.

    As a next step, we calculate the integral with respect to $t$, which we denote by $I_2$.
		The argument is essentially the same as in  \cite[p.~241-2]{Schneider14}. We include it for the the sake
		of completeness.
		For this, we can again assume that $F$ and $\vartheta F'$ are in general position, that is, $[F,\vartheta F']\neq 0$. We set $\alpha:= F \cap \beta$ and
		$\alpha' := \vartheta F' \cap \vartheta \beta'$. We can assume that $\alpha\neq \emptyset$ and $\alpha'\neq\emptyset$, since otherwise both sides of the equation to be derived are zero. Let $s_0,t_0\in\R^n$ be such  that $s_0\in \alpha\cap(\alpha'+t_0)\neq \emptyset$.
		Then we define $L_1 := F^0 \cap (\vartheta F')^0$, $L_2 := F^0 \cap L_1^\perp$, $L_3 := (\vartheta F')^0 \cap L_1^\perp$. Hence,
		for every $t\in\R^n$, there are uniquely determined vectors $t_i \in L_i$, $i = 1,2,3$, such that $t = t_0+t_1 + t_2 + t_3$. The transformation
		formula for integrals then yields that
    \begin{align*}
      I_2 & = [F, \vartheta F'] \int_{L_3} \int_{L_2} \int_{L_1} \int_{\alpha \cap (\alpha' + t_0+t_1 + t_2 + t_3)} x^r \, \cH^j(\intd x) \, \cH^j(\intd t_1) \, \cH^{k - j}(\intd t_2) \, \cH^{n - k}(\intd t_3).
    \end{align*}
    Since
		$$
		(\alpha-s_0-t_2)\cap(\alpha'+t_0-s_0+t_1+t_3)\subset F^0\cap (\vartheta F')^0=L_1,
		$$
		we have
		$\alpha \cap ( \alpha' + t_0+t_1 + t_2 + t_3) \subset s_0+L_1 + t_2$, and hence, from Fubini's theorem, we obtain for the two inner integrals
    \begin{align*}
      &\int_{L_1} \int_{\alpha \cap (\alpha' + t_0+t_1 + t_2 + t_3)\cap (s_0+L_1 + t_2) } x^r \, \cH^j(\intd x) \, \cH^j(\intd t_1) \\
			& \qquad= \int_{L_1} \int_{L_1 \cap (\alpha - t_2-s_0)\cap (\alpha' + t_0-s_0+ t_1 + t_3)} \left( x +s_0+ t_2 \right)^r \, \cH^j(\intd x) \, \cH^j(\intd t_1) \\
			&\qquad=\int \1\{x\in (\alpha - t_2-s_0)\cap L_1\}\left( x +s_0+ t_2 \right)^r\\
			&\qquad\qquad\qquad \times \int \1\{t_1\in L_1,x\in \alpha' + t_0-s_0+ t_1 + t_3\} \, \cH^j(\intd t_1)\,
			\cH^j(\intd x)\\
			&\qquad=\int \1\{x\in (\alpha - t_2-s_0)\cap L_1\} \left( x +s_0+ t_2 \right)^r \\
			&\qquad\qquad\qquad \times\cH^{j}([(\alpha'+t_0-s_0)+t_3]\cap L_1)\, \cH^j(\intd x)\\
			&\qquad=\cH^{j}([(\alpha'+t_0-s_0)+t_3]\cap L_1)\int_{(\alpha-s_0-t_2)\cap L_1}\left( x +s_0+ t_2 \right)^r\, \cH^j(\intd x)\\
			&\qquad=\cH^{j}((\alpha'+t_0-s_0)\cap (L_1+t_3))\int_{(\alpha-s_0 )\cap (L_1+t_2)}\left( x +s_0\right)^r\, \cH^j(\intd x),\\
    \end{align*}
    which yields
    \begin{align*}
      I_2 & = [F, \vartheta F'] \int_{L_3} \cH^j \left( (\alpha'+ t_0-s_0)\cap (L_1+t_3) \right) \, \cH^{n - k}(\intd t_3) \\
			&\qquad\qquad\qquad\times \int_{L_2} \int_{(\alpha-s_0)\cap (L_1 + t_2)} (x+s_0)^r \, \cH^j(\intd x) \, \cH^{k - j}(\intd t_2)\\
			&=[F, \vartheta F']\cH^{n-k+j}(\alpha'+t_0-s_0)\int_{\alpha-s_0}(x+s_0)^r\, \cH^k(\intd x)\\
			&=[F, \vartheta F']\cH^{n-k+j}(F' \cap \beta' )\int_{F\cap \beta} x^r\, \cH^k(\intd x).
    \end{align*}
   Thus, (\ref{form_I1_mit_I2}) can be rewritten as
    \begin{align*}
      I_1 & = c_{n, j}^{r, s, l}\frac{1}{\omega_{n - j}}  \sum_{k = j}^n \sum_{F \in \cF_k(P)} \sum_{F' \in \cF_{n - k + j}(P')} \cH^{n - k + j}(F' \cap \beta') \int _{F \cap \beta} x^r \, \cH^k(\intd x) \\
      & \qquad \qquad \times \int_{\SO(n)} [F, \vartheta F'] Q \left( F^0 \cap (\vartheta F')^0 \right)^l \\
      & \qquad \qquad \times \int _{N(P \cap (\vartheta P' + t), F \cap (\vartheta F' + t)) \cap \S} u^s \, \cH^{n - j - 1} (\intd u) \, \nu( \intd \vartheta),
    \end{align*}
		where $t\in\R^n$ is chosen as described after \eqref{form_I1_mit_I2}.

\subsection{The cases $k\in\{j,n\}$}\label{subsec5.2}

  We have to consider the cases $k = j$ and $k = n$ separately, since the application of some of the auxiliary results
	requires that $k-j\ge 1$ and $k\le n-1$. Starting with $k = j$, we get
    \begin{align*}
      & c_{n, j}^{r, s, l}\frac{1}{\omega_{n - j}}  \sum_{F \in \cF_j(P)} \sum_{F' \in \cF_n(P')} \cH^n (F' \cap \beta') \int _{F \cap \beta} x^r \, \cH^j (\intd x) \int_{\SO(n)} [F, \vartheta F'] \\
      & \qquad \qquad \times Q \left( F^0 \cap (\vartheta F')^0 \right)^l \int _{N(P \cap (\vartheta P' + t), F \cap (\vartheta F' + t)) \cap \S} u^s \, \cH^{n - j - 1} (\intd u) \, \nu( \intd \vartheta) \\
      & \qquad = c_{n, j}^{r, s, l}\frac{1}{\omega_{n - j}}  \sum_{F \in \cF_j(P)} \underbrace{\cH^n (P' \cap \beta')}_{ = \CM{n}(P', \beta')} \int _{F \cap \beta} x^r \, \cH^j (\intd x) \int_{\SO(n)} \underbrace{[F, \vartheta P']}_{ = [F, \R^n] = 1} \\
      & \qquad \qquad \times Q \big( F^0 \cap \underbrace{(\vartheta P')^0}_{ = \R^n} \big)^l \int _{\underbrace{\scriptstyle N(P \cap (\vartheta P' + t), F \cap (\vartheta P' + t))}_{ = N(P, F)} \cap \S} u^s \, \cH^{n - j - 1} (\intd u) \, \nu( \intd \vartheta) \\
      &\qquad = \TenCM{j}{r}{s}{l} (P, \beta) \CM{n}(P', \beta').
    \end{align*}

    For $k = n$, we conclude from Fubini's theorem
    \begin{align*}
       &  c_{n, j}^{r, s, l}\frac{1}{\omega_{n - j}}  \sum_{F \in \cF_n(P)} \sum_{F' \in \cF_{j}(P')} \cH^{j}(F' \cap \beta') \int _{F \cap \beta} x^r \, \cH^n(\intd x) \int_{\SO(n)} \underbrace{[F, \vartheta F']}_{ = [\R^n, \vartheta F'] = 1} \\
      & \qquad\qquad \times Q \big( \underbrace{F^0}_{ = \R^n} \cap (\vartheta F')^0 \big)^l \int _{\underbrace{\scriptstyle N(P \cap (\vartheta P' + t), P \cap (\vartheta F' + t))}_{ = \vartheta N(P', F')} \cap \S} u^s \, \cH^{n - j - 1} (\intd u) \, \nu( \intd \vartheta) \\
      & \qquad = c_{n, j}^{r, s, l}\frac{1}{\omega_{n - j}}  \int _{P \cap \beta} x^r \, \cH^n(\intd x) \sum_{F' \in \cF_{j}(P')} \cH^{j}(F' \cap \beta') \\
      & \qquad\qquad \times \int _{N(P', F') \cap \S} \int_{\SO(n)} Q \left( \vartheta F' \right)^l (\vartheta u)^s \, \nu( \intd \vartheta) \, \cH^{n - j - 1} (\intd u).
    \end{align*}
    For this, we obtain from Lemma \ref{Lem_IntForm_MeTVec}
    \begin{align*}
       &  \1 \{ s \text{ even} \}  c_{n, j}^{r, s, l} {\frac{1}{\omega_{n - j}} \frac{\Gamma(\frac {n} {2}) \Gamma(\frac {j} {2} + l) \Gamma(\frac {s + 1} {2})}{\sqrt \pi \Gamma(\frac {n + s} {2} + l) \Gamma(\frac {j} {2})}}  Q^{l + \frac s 2} \int _{P \cap \beta} x^r \, \cH^n(\intd x) \\
      & \qquad \qquad \qquad \times \sum_{F' \in \cF_{j}(P')} \cH^{j}(F' \cap \beta') \int _{N(P', F') \cap \S} \cH^{n - j - 1} (\intd u) \\
      & \qquad = c_{n, j}^{s} \, \TenCM{n}{r}{0}{\frac s 2 + l} (P, \beta) \CM{j}(P', \beta'),
    \end{align*}
    where
    \begin{equation}\label{Plus}
      c_{n, j}^{s}  : = \1 \{ s \text{ even} \} \frac{2 \omega_{n - j}}{s! \omega_{s + 1} \omega_{n - j + s}} = \1 \{ s \text{ even} \} \frac{1}{(2\pi)^{s} \left(\frac{s}{2}\right)!} \frac{\Gamma(\frac{n - j + s}{2})}{\Gamma(\frac{n - j}{2})}.
    \end{equation}

    Hence, we get
    \begin{align*}
      I_1 & = c_{n, j}^{r, s, l}\frac{1}{\omega_{n - j}}  \sum_{k = j + 1}^{n - 1} \sum_{F \in \cF_k(P)} \sum_{F' \in \cF_{n - k + j}(P')} \cH^{n - k + j}(F' \cap \beta') \\
      & \qquad \qquad \qquad \times \int _{F \cap \beta} x^r \, \cH^k(\intd x) \int_{\SO(n)} [F, \vartheta F'] Q\left( F^0 \cap (\vartheta F')^0 \right)^l \\
      & \qquad \qquad \qquad \times \int _{N(P \cap (\vartheta P' + t), F \cap (\vartheta F' + t)) \cap \S} u^s \, \cH^{n - j - 1} (\intd u) \, \nu( \intd \vartheta) \\
      &	\qquad + \TenCM{j}{r}{s}{l} (P, \beta) \CM{n}(P', \beta') + c_{n, j}^{s} \, \TenCM{n}{r}{0}{\frac s 2 + l}(P, \beta) \CM{j}(P', \beta').
    \end{align*}
    Furthermore, for any $t \in \R^n$ such that $\relint F \cap \relint (\vartheta F' + t) \neq \emptyset$ we obtain from  \cite[Theorem 2.2.1]{Schneider14} that
    \begin{align*}
      N\left( P \cap (\vartheta P' + t), F \cap (\vartheta F' + t) \right) & = N(P, F) + \vartheta N(P', F'),
    \end{align*}
    and thus
    \begin{align}
      I_1 & = c_{n, j}^{r, s, l}\frac{1}{\omega_{n - j}}  \sum_{k = j + 1}^{n - 1} \sum_{F \in \cF_k(P)} \sum_{F' \in \cF_{n - k + j}(P')} \cH^{n - k + j}(F' \cap \beta') \nonumber \\
      & \qquad \qquad \qquad \times \int _{F \cap \beta} x^r \, \cH^k(\intd x) \int_{\SO(n)} [F, \vartheta F'] Q\left( F^0 \cap (\vartheta F')^0 \right)^l \nonumber \\
      & \qquad \qquad \qquad \times \int _{(N(P, F) + \vartheta N(P', F')) \cap \S} u^s \, \cH^{n - j - 1} (\intd u) \, \nu( \intd \vartheta) \nonumber \\
      &	\qquad + \TenCM{j}{r}{s}{l} (P, \beta) \CM{n}(P', \beta') + c_{n, j}^{s} \, \TenCM{n}{r}{0}{\frac s 2 + l} (P, \beta) \CM{j}(P', \beta') . \label{form_I1_mit_J}
    \end{align}
	 In the following, we denote by $C(\omega) := \{ \lambda x \in \R^n : x \in \omega, \lambda > 0\}$
	the cone spanned by a set  $\omega\subset\mathbb{S}^{n-1}$. Moreover, if $F$ is a face of $P$, we write
	$F^\perp$ for the linear subspace orthogonal to $F^0$.
    For the next and main step, we may assume that $j\le n-2$ (since $j<k\le n-1$). We define the mapping
    \begin{align*}
      J: \cB(F^\perp \cap \S) \times \cB(F'^\perp \cap \S) \rightarrow \T^{2l + s}
    \end{align*}
    by
    \begin{align*}
      J(\omega, \omega')& := \int_{\SO(n)} [F, \vartheta F'] Q\left( F^0 \cap (\vartheta F')^0 \right)^l \\
      & \qquad \qquad \times \int _{(C(\omega) + \vartheta C(\omega')) \cap \S} u^s
			\, \cH^{n - j - 1} (\intd u) \, \nu( \intd \vartheta)
    \end{align*}
    for $\omega \in \cB(F^\perp \cap \S)$ and $\omega' \in \cB(F'^\perp \cap \S)$.
		Then $J$ is a finite signed measure on $\cB(F^\perp \cap \S)$ in the first variable and
		a finite signed measure on $\cB(F'^\perp \cap \S)$ in the second variable, but this will not be needed in the following.
		In fact, we could specialize to the case $\omega=N(P,F)\cap \S$ and  $\omega'=N(P',F')\cap \S$ throughout the proof.
		
		Since  $[F, \vartheta F'] Q\left( F^0 \cap (\vartheta F')^0 \right)^l$ depends only on the linear subspaces
		$F^0$ and $(\vartheta F')^0$, we can assume that $F\in \GOp(n,k)$ and $F'\in\GOp(n,n-k+j)$ for determining $J(\omega,\omega')$.   Moreover, for $\nu$-almost all $\vartheta\in \SO(n)$, the linear subspaces $F^\perp$
		and $\vartheta (F'^\perp)$ are linearly independent. This will be tacitly used in the following.

    \subsection{The rotational part}\label{subsec5.3}
In this section, $\omega,\omega'$ are fixed and as described above.
    Due to the right invariance of $\nu$ and Fubini's theorem, we obtain for $\rho \in \SO(F'^{\perp})$
    \begin{align*}
      J(\omega, \omega') & = \int_{\SO(n)} [F, {\vartheta \rho F'}] Q(F \cap {\vartheta \rho F'})^l \int _{(C(\omega) + \vartheta \rho C(\omega')) \cap \S} u^s \, \cH^{n - j - 1} (\intd u) \, \nu( \intd \vartheta) \\
      & = \int_{\SO(n)} [F, \vartheta F'] Q(F \cap \vartheta F')^l \\
      & \qquad \qquad \times \int _{\SO(F'^\perp)} \int _{(C(\omega) + \vartheta \rho C(\omega')) \cap \S} u^s \, \cH^{n - j - 1} (\intd u) \, \nu^{F'^\perp}(\intd \rho) \, \nu( \intd \vartheta).
    \end{align*}
    Next, we introduce a multiple $J_1$ of the inner integral of $J(\omega,\omega')$ and rewrite it by means of a polar coordinate transformation, that is,
    \begin{align*}
      J_1 : = & \tfrac 1 2 \Gamma({\textstyle\frac {n - j + s} 2}) \int _{(C(\omega) + \vartheta \rho C(\omega')) \cap \S} u^s \, \cH^{n - j - 1}(\intd u)\\
      = & \int _0^\infty \int _{(C(\omega) + \vartheta \rho C(\omega')) \cap \S} (ru)^s e^{-\| ru\|^2} r^{n - j - 1} \, \cH^{n - j - 1}(\intd u) \, \intd r \\
      = & \int _{C(\omega) + \vartheta \rho C(\omega')} x^s e^{-\| x\|^2} \, \cH^{n - j}(\intd x) .
    \end{align*}
    The bijective transformation (here we assume that $\vartheta \in \SO(n)$ is such that  $F^\perp$
		and $\vartheta (F'^\perp)$ are linearly independent subspaces)
    \begin{align*}
      T: \omega \times \omega' \times (0, \infty)^2 \rightarrow C(\omega) + \vartheta \rho C(\omega'), \quad (u, v, t_1, t_2) \mapsto t_1 u + t_2 \vartheta \rho v,
    \end{align*}
    has the Jacobian
    \begin{align*}
      \cJ T(u,v,t_1,t_2) = t_1^{n - k - 1} t_2^{k - j - 1} [ F^\perp , \vartheta F'^\perp] = t_1^{n - k - 1} t_2^{k - j - 1} [ F , \vartheta F'].
    \end{align*}
    Hence, we obtain
    \begin{align*}
      J_1 & = \int _{\omega} \int _{\omega'} \int _{(0,\infty)^2} t_1^{n - k - 1} t_2^{k - j - 1} [ F , \vartheta F'] \left( t_1 u + t_2 \vartheta\rho v \right)^s e^{-\| t_1 u + t_2 \vartheta \rho v \|^2} \\ & \qquad \qquad \qquad \qquad \times \cH^{2} ( \intd (t_1,t_2)) \, \cH^{k - j - 1}(\intd v) \, \cH^{n - k - 1}(\intd u).
    \end{align*}
    Applying a polar coordinate transformation to the inner integral and then using binomial expansion, we get
    \begin{align*}
      J_1 & = [ F , \vartheta F'] \int _{\omega} \int _{\omega'} \int _0^{\frac \pi 2}
			\int _0^\infty (r \cos(\alpha))^{n - k - 1} (r \sin(\alpha))^{k - j - 1} e^{-\|
			r \cos(\alpha) u + r \sin(\alpha) \vartheta \rho v \|^2} \\
      & \qquad \times \left( r \cos(\alpha) u + r \sin(\alpha) \vartheta \rho v
			\right)^s r \, \intd r \, \intd \alpha \, \cH^{k - j - 1}(\intd v) \, \cH^{n - k - 1}(\intd u) \\
      & = [ F , \vartheta F'] \int _{\omega} \int _{\omega'} \int _0^{\frac \pi 2}
			\int _0^\infty r^{n - j + s - 1} e^{-r^2 (1 +  { \scriptstyle 2
			\sin(\alpha)\cos(\alpha)}  \langle u, \vartheta \rho v \rangle)}
			\sum_{i = 0}^{s} \binom{s}{i} \\
      & \qquad \times \cos(\alpha)^{n - k + s - i - 1} \sin(\alpha)^{k - j + i - 1} u^{s - i}
			(\vartheta \rho v)^{i} \, \intd r \, \intd \alpha \, \cH^{k - j - 1}(\intd v) \, \cH^{n - k - 1}(\intd u).
    \end{align*}
    Since $\sin(2\alpha)=2\sin(\alpha)\cos(\alpha)$ and
    \begin{align*}
      & \int_{\SO(n)} \int _{\SO(F'^\perp)} \int _{\omega} \int _{\omega'} \int _0^{\frac \pi 2} \int _0^\infty \bigg \| [F, \vartheta F']^2 Q(F \cap \vartheta F')^l r^{n - j + s - 1} \\
      & \qquad\qquad \times e^{-r^2 (1 + \sin(2 \alpha) \langle u, \vartheta \rho v \rangle)} \sum_{i = 0}^{s} \binom{s}{i} u^{s - i} (\vartheta \rho v)^{i} \cos(\alpha)^{n - k + s - i - 1} \sin(\alpha)^{k - j + i - 1}\bigg \| \\
      & \qquad\qquad \times \, \intd r \, \intd \alpha \, \cH^{k - j - 1}(\intd v) \, \cH^{n - k - 1}(\intd u) \, \nu^{F'^\perp}(\intd \rho) \, \nu( \intd \vartheta) \\
      &\qquad \leq   \int_{\SO(n)} \int _{\SO(F'^\perp)} \int _{\omega} \int _{\omega'} \int _{(0,\infty)^2} t_1^{n - k - 1} t_2^{k - j - 1} [ F , \vartheta F'] \big\| t_1 u + t_2 \vartheta\rho v \big\| ^s e^{-\| t_1 u + t_2 \vartheta \rho v \|^2} \\
      & \qquad\qquad \times \cH^{2} ( \intd (t_1,t_2)) \, \cH^{k - j - 1}(\intd v) \, \cH^{n - k - 1}(\intd u) \, \nu^{F'^\perp}(\intd \rho) \, \nu( \intd \vartheta) \\
      &\qquad = \frac{2}{\Gamma(\frac {n - j + s} 2)} \int_{\SO(n)} \int _{\SO(F'^\perp)} \int _{(C(\omega) + \vartheta \rho C(\omega')) \cap \S} \|u \|^s \, \cH^{n - j - 1} (\intd u) \, \nu^{F'^\perp}(\intd \rho) \, \nu( \intd \vartheta)
    \end{align*}
 is finite,   Fubini's theorem can be applied and yields
    \begin{align*}
      J(\omega, \omega') & = \frac{2}{\Gamma(\frac {n - j + s} 2)} \int _{\omega} \int_{\SO(n)} [F, \vartheta F']^2 Q(F \cap \vartheta F')^l \int _0^{\frac \pi 2} \int _{\omega'} \int _{\SO(F'^\perp)} \sum_{i = 0}^{s} \binom{s}{i} u^{s - i} (\vartheta \rho v)^{i} \\
      & \qquad \times \cos(\alpha)^{n - k + s - i - 1} \sin(\alpha)^{k - j + i - 1} \int _0^\infty r^{n - j + s - 1} e^{-r^2 (1 + \sin(2 \alpha) \langle u, \vartheta \rho v \rangle)} \, \intd r \\
      & \qquad \times \, \nu^{F'^\perp}(\intd \rho) \, \cH^{k - j - 1}(\intd v) \, \intd \alpha \, \nu( \intd \vartheta) \, \cH^{n - k - 1}(\intd u).
    \end{align*}

    Now we substitute $x := r \sqrt{1 + \sin(2\alpha) \langle u, \vartheta \rho v \rangle}$ in the integration
		with respect to $r$ (which is admissible, since $1 + \sin(2\alpha) \langle u, \vartheta \rho v \rangle \neq 0$ holds
		for almost all $\alpha, u, v, \vartheta, \rho$). Denoting the corresponding integral by $J_{2}$, we obtain
    \begin{align*}
      J_2 & = \frac{1}{\sqrt{1 + \sin(2\alpha) \langle u, \vartheta \rho v \rangle}^{n - j + s}} \int _0^\infty x^{n - j + s - 1} e^{-x^2} \, \intd x = \frac{\Gamma(\frac {n - j + s} 2)}{2\sqrt{1 + \sin(2\alpha) \langle u, \vartheta \rho v \rangle}^{n - j + s}}.
    \end{align*}
    Hence, we arrive at
    \begin{align*}
      J(\omega, \omega') & = \int _{\omega} \int_{\SO(n)} [F, \vartheta F']^2 Q(F \cap \vartheta F')^l \int _0^{\frac \pi 2} \int _{\omega'} \int _{\SO(F'^\perp)} \\
      & \qquad  \times \sum_{i = 0}^{s} \binom{s}{i} u^{s - i} (\vartheta \rho v)^{i} \frac {\cos(\alpha)^{n - k + s - i - 1} \sin(\alpha)^{k - j + i - 1}} {\sqrt{1 + \sin(2\alpha) \langle u, \vartheta \rho v \rangle}^{n - j + s}} \\
       & \qquad \times \nu^{F'^\perp}(\intd \rho) \, \cH^{k - j - 1}(\intd v) \, \intd \alpha \, \nu( \intd \vartheta) \, \cH^{n - k - 1}(\intd u).
    \end{align*}
    Expanding $(1 + \sin(2\alpha) \langle u, \vartheta \rho v \rangle)^{- (n - j + s)/2}$ as a binomial series gives
    \begin{align*}
      J(\omega, \omega') & = \int _{\omega} \int_{\SO(n)} [F, \vartheta F']^2 Q(F \cap \vartheta F')^l \int _0^{\frac \pi 2} \int _{\omega'} \int _{\SO(F'^\perp)} \sum_{i = 0}^{s} \binom{s}{i} u^{s - i} \\
      & \qquad \times \sum_{t = 0}^{\infty} (-2)^t \binom{\frac {n - j + s} 2 + t - 1}{t} (\vartheta \rho v)^{i} \langle u, \vartheta \rho v \rangle^t \cos(\alpha)^{n - k + s - i + t - 1}\\
      & \qquad \times \sin(\alpha)^{k - j + i + t - 1} \nu^{F'^\perp}(\intd \rho) \, \cH^{k - j - 1}(\intd v) \, \intd \alpha \, \nu( \intd \vartheta) \, \cH^{n - k - 1}(\intd u).
    \end{align*}
    The series with respect to $t$ converges absolutely for almost all $u, v, \vartheta, \alpha$; in fact,
				\begin{align*}
      & \, \sum_{t = 0}^{\infty} \bigg\| (-2)^t \binom{\frac {n - j + s} 2 + t - 1}{t}  \langle u, \vartheta \rho v \rangle^t \cos(\alpha)^{n - k + s - i + t - 1} \sin(\alpha)^{k - j + i + t - 1} \bigg\| \\
      &\qquad\le \, \sum_{t = 0}^{\infty} \binom{\frac {n - j + s} 2 + t - 1}{t} {\sin(2\alpha)}^{t}<\infty,
			\end{align*}
			since $\sin(2\alpha)<1$ for $\alpha\in(0,\frac{\pi}{2})\setminus\{\frac{\pi}{4}\}$. Therefore
    we can interchange the integration with respect to $\rho$ and the summation with respect to $t$. We denote the resulting integral with respect to $\rho$ by $J_{3}$ and obtain
    \begin{align*}
      J_3 & = \vartheta \int _{\SO(F'^\perp)} (\rho v)^{i} \left\langle p_{F'^\perp} (\vartheta^{-1} u) + p_{F'} (\vartheta^{-1} u), \rho v \right\rangle^t \, \nu^{F'^\perp}(\intd \rho) \\
      & = \| p_{\vartheta F'^\perp} (u) \|^t \, \vartheta \int _{\SO(F'^\perp)} (\rho v)^{i} \bigg\langle \underbrace{\frac {p_{F'^\perp} (\vartheta^{-1} u)} {\| p_{F'^\perp} (\vartheta^{-1} u) \|}}_{ = \pi_{F'^\perp} (\vartheta^{-1} u)}, \rho v \bigg\rangle ^t \, \nu^{F'^\perp}(\intd \rho)
    \end{align*}
		if $\vartheta^{-1}u\notin F'$ (which holds for almost all choices of $u,\vartheta$).  Note that the integration over
		$\SO(F'^\perp)$ yields the same value as an integration over all $\vartheta\in \OO(n)$ which fix $F'^0$ pointwise, since
		$\text{dim}(F'^\perp)\in\{1,\ldots,n-1\}$ and $n\ge 2$.
    Hence, an application of Lemma \ref{Lem_IntForm_VecScal} in $F'^\perp$ yields
		\begin{align*}
      J_3 & = \1\{ i + t \text{ even} \} \frac {\Gamma(\frac {k - j} 2)} {\sqrt \pi} \frac {\Gamma(t + 1)} {2^t \Gamma(\frac {k - j + i + t} 2)} \| p_{\vartheta F'^\perp} (u) \|^t \\
      & \qquad \times \sum_{x = (\frac {i - t} 2)^+ }^{\lfloor \frac i 2 \rfloor} \binom{i}{2x} \frac {\Gamma(x + \frac 1 2) } {\Gamma(\frac {t - i} 2 + x + 1)} \pi_{\vartheta F'^\perp} (u)^{i - 2x} Q(\vartheta F'^\perp)^{x}.
    \end{align*}
		This also holds if $\vartheta^{-1}u\in F'$, where both sides are zero (even if $\pi_{\vartheta F'^\perp} (u)$ is undefined).
    Hence, we conclude
    \begin{align*}
      J(\omega, \omega') & = \frac {\Gamma(\frac {k - j} 2)} {\sqrt \pi} \cH^{k - j - 1}(\omega') \int _{\omega} \int_{\SO(n)} [F, \vartheta F']^2 Q(F \cap \vartheta F')^l \int _0^{\frac \pi 2} \sum_{i = 0}^{s} (-1)^i \binom{s}{i} u^{s - i} \\
      & \qquad \times \sum_{t = 0}^{\infty} \1 \{ i + t \text{ even} \} \binom{\frac {n - j + s} 2 + t - 1}{t} \frac {\Gamma(t + 1)} {\Gamma(\frac {k - j + i + t} 2)} \sum_{x = (\frac{i - t} 2)^+}^{\lfloor \frac i 2 \rfloor} \binom{i}{2x} \\
      & \qquad \times \frac {\Gamma(x + \frac 1 2)} {\Gamma(\frac {t - i} 2 + x + 1)} \| p_{\vartheta F'^\perp} (u) \|^t \pi_{\vartheta F'^\perp} (u)^{i - 2x} Q(\vartheta F'^\perp)^{x} \\
      & \qquad \times \cos(\alpha)^{n - k + s - i + t - 1} \sin(\alpha)^{k - j + i + t - 1} \, \intd \alpha \, \nu( \intd \vartheta) \, \cH^{n - k - 1}(\intd u),
    \end{align*}
    where we used that $(-1)^t = (-1)^i$ provided that $i + t$ is even.

    As the series with respect to $t$ converges absolutely for almost all $u, \vartheta$ (by a similar argument as before,
		which uses that $\| p_{\vartheta F'^\perp} (u) \| < 1$, that is, $\vartheta^{-1}u\notin F'^\perp$,  almost surely), we can rearrange the order of integration and summation to get
    \begin{align*}
      J(\omega, \omega') & = \frac {\Gamma(\frac {k - j} 2)} {\sqrt \pi} \cH^{k - j - 1}(\omega') \int _{\omega} \int_{\SO(n)} [F, \vartheta F']^2 Q(F \cap \vartheta F')^l \sum_{i = 0}^{s} \sum_{x = 0}^{\lfloor \frac i 2 \rfloor} (-1)^i \binom{s}{i} \binom{i}{2x} \\
      & \qquad \times \Gamma(x + \tfrac 1 2) u^{s - i} \sum_{t = i - 2x}^{\infty} \1 \{ i + t \text{ even} \} \binom{\frac {n - j + s} 2 + t - 1}{t} \\
      & \qquad \times \frac {\Gamma(t + 1)} {\Gamma(\frac {k - j + i + t} 2) \Gamma(\frac {t - i} 2 + x + 1)} \| p_{\vartheta F'^\perp} (u) \|^t \pi_{\vartheta F'^\perp} (u)^{i - 2x} Q(\vartheta F'^\perp)^{x} \\
      & \qquad \times  {\int _0^{\frac \pi 2} \cos(\alpha)^{n - k + s - i + t - 1} \sin(\alpha)^{k - j + i + t - 1} \, \intd \alpha}  \, \nu( \intd \vartheta) \, \cH^{n - k - 1}(\intd u) \\
      & = \frac {\Gamma(\frac {k - j} 2)} {2\sqrt \pi \Gamma(\frac {n - j + s} 2)} \cH^{k - j - 1}(\omega') \int _{\omega} \int_{\SO(n)} [F, \vartheta F']^2 Q(F \cap \vartheta F')^l \sum_{i = 0}^{s} \sum_{x = 0}^{\lfloor \frac i 2 \rfloor} (-1)^i \\
      & \qquad \times \binom{s}{i} \binom{i}{2x} \Gamma(x + \tfrac 1 2) u^{s - i} \pi_{\vartheta F'^\perp} (u)^{i - 2x} Q(\vartheta F'^\perp)^{x} \\
      & \qquad \times \sum_{t = i - 2x}^{\infty} \1 \{ i + t \text{ even} \} \frac {\Gamma(\frac {n - k + s - i + t} 2)} {\Gamma(\frac {t - i} 2 + x + 1)} \| p_{\vartheta F'^\perp} (u) \|^t \, \nu( \intd \vartheta) \, \cH^{n - k - 1}(\intd u) .
    \end{align*}
    We denote the series with respect to $t$ by $S_{2}$. Then, for $\vartheta^{-1}u\notin F'^\perp$, we obtain (after an index shift)
    \begin{align*}
      S_2 & = \sum_{t = 0}^{\infty} \underbrace{\1 \{ 2i - 2x + t \text{ even} \}}_{ = \1\{t \text{ even} \} } \frac {\Gamma(\frac {n - k + s + t - 2x} 2)} {\Gamma(\frac {t} 2 + 1)} \| p_{\vartheta F'^\perp} (u) \|^{i - 2x + t} \\
      & = \| p_{\vartheta F'^\perp} (u) \|^{i - 2x} \sum_{t = 0}^{\infty} \frac {\Gamma(\frac {n - k + s} 2 + t - x)} {\Gamma(t + 1)} \| p_{\vartheta F'^\perp} (u) \|^{2t} \\
      & = \Gamma(\tfrac {n - k + s} 2 - x) \| p_{\vartheta F'^\perp} (u) \|^{i - 2x} \sum_{t = 0}^{\infty} \binom{ - \frac{n - k + s} 2 + x}{t} (- \| p_{\vartheta F'^\perp} (u) \|^{2})^t ,
    \end{align*}
    where the remaining series is just a binomial series. Hence, we get
    \begin{align*}
      S_2 & = \Gamma(\tfrac {n - k + s} 2 - x) \| p_{\vartheta F'^\perp} (u) \|^{i - 2x} (1 - \| p_{\vartheta F'^\perp} (u) \|^{2})^{- \frac{n - k + s} 2 + x} \\
      & = \Gamma(\tfrac {n - k + s} 2 - x) \| p_{\vartheta F'^\perp} (u) \|^{i - 2x} \| p_{\vartheta F'} (u) \|^{- n + k - s + 2x}.
    \end{align*}
    Expanding $Q(\vartheta F'^\perp)^x = (Q - Q(\vartheta F'))^x$ in $J(\omega, \omega')$, we obtain
    \begin{align*}
      J(\omega, \omega') & = \frac {\Gamma(\frac {k - j} 2)} {2\sqrt \pi \Gamma(\frac {n - j + s} 2)} \cH^{k - j - 1}(\omega') \int _{\omega} \int_{\SO(n)} \sum_{i = 0}^{s} \sum_{x = 0}^{\lfloor \frac i 2 \rfloor} \sum_{y = 0}^x (-1)^{i + y} \binom{s}{i} \binom{i}{2x} \binom{x}{y} \Gamma(x + \tfrac 1 2) \\
      & \qquad \times \Gamma(\tfrac {n - k + s} 2 - x) u^{s - i} Q^{x - y} [F, \vartheta F']^2 \| p_{\vartheta F'} (u) \|^{- n + k - s + 2x} \\
      & \qquad \times  p_{\vartheta F'^\perp} (u) ^{i - 2x}  Q(\vartheta F')^{y} Q(F \cap \vartheta F')^l \, \nu( \intd \vartheta) \, \cH^{n - k - 1}(\intd u).
    \end{align*}
    Changing the order of the summation under the integral gives
    \begin{align*}
      J(\omega, \omega') & = \frac {\Gamma(\frac {k - j} 2)} {2\sqrt \pi \Gamma(\frac {n - j + s} 2)} \cH^{k - j - 1}(\omega') \int _{\omega} \int_{\SO(n)} \sum_{x = 0}^{\lfloor \frac s 2 \rfloor} \sum_{y = 0}^x \sum_{i = 2x}^{s} (-1)^{i + y} \binom{s}{i} \binom{i}{2x} \binom{x}{y} \Gamma(x + \tfrac 1 2) \\
      & \qquad \times \Gamma(\tfrac {n - k + s} 2 - x) u^{s - i} Q^{x - y} [F, \vartheta F']^2 \| p_{\vartheta F'} (u) \|^{- n + k - s + 2x} p_{\vartheta F'^\perp} (u) ^{i - 2x} \\
      & \qquad \times Q(\vartheta F')^{y} Q(F \cap \vartheta F')^l \, \nu( \intd \vartheta) \, \cH^{n - k - 1}(\intd u).
    \end{align*}
    We denote the integral with respect to $\vartheta$ in $J(\omega, \omega')$ by $J_{4}$. Since $n\ge 2$ and $1\le n-k+j\le n-1$, Lemma \ref{Lem_TraFo_HSS} yields
    \begin{align*}
      J_{4} & = \int_{\GOp(n, n - k + j)} \sum_{x = 0}^{\lfloor \frac s 2 \rfloor} \sum_{y = 0}^x \sum_{i = 2x}^{s} (-1)^{i + y} \binom{s}{i} \binom{i}{2x} \binom{x}{y} \Gamma(x + \tfrac 1 2) \Gamma(\tfrac {n - k + s} 2 - x) \\
      & \qquad \times u^{s - i} Q^{x - y} [F, G]^2 \| p_{G} (u) \|^{- n + k - s + 2x} p_{G^\perp} (u) ^{i - 2x} Q(G)^{y} Q(F \cap G)^l \, \nu_{n - k + j}( \intd G) \\
      & = \frac {\omega_{n - k + j}} {2 \omega_n} \int_{\GOp(u^\perp, n - k + j - 1)} \int _{-1}^1 \int _{U^\perp \cap u^\perp \cap \S} \sum_{x = 0}^{\lfloor \frac s 2 \rfloor} \sum_{y = 0}^x \sum_{i = 2x}^{s} (-1)^{i + y} \binom{s}{i} \binom{i}{2x} \binom{x}{y} \Gamma(x + \tfrac 1 2) \\
      & \qquad \times \Gamma(\tfrac {n - k + s} 2 - x) u^{s - i} Q^{x - y} | z |^{n - k + j - 1} \left( 1 - z^2 \right)^{\frac {k - j - 2} 2} [F, \lin \{ U, z u + \sqrt{1 - z^2} w \} ]^2 \\
      & \qquad \times \| p_{\lin \{ U, z u + \sqrt{1 - z^2} w \}} (u) \|^{- n + k - s + 2x} Q(\lin \{ U, z u + \sqrt{1 - z^2} w \})^{y} \\
      & \qquad \times Q(F \cap \lin \{ U, z u + \sqrt{1 - z^2} w \})^l p_{\lin \{ U, z u + \sqrt{1 - z^2} w \}^\perp} (u)^{i - 2x} \, \\
      & \qquad \times \cH^{k - j - 1}(\intd w) \, \intd z \, \nu^{u^\perp}_{n - k + j - 1}( \intd U).
    \end{align*}
		The required integrability will be clear from \eqref{absint} below.
    Since $u,w \in U^\perp$, we obtain
    \begin{align*}
      p_{\lin \{ U, z u + \sqrt{1 - z^2} w \}^\perp} (u) & = p_{U^\perp \cap (z u + \sqrt{1 - z^2} w)^\perp} (u)
       = p_{(z u + \sqrt{1 - z^2} w)^\perp} (u) \\
      & = u -  p_{z u + \sqrt{1 - z^2} w} (u)  = u -  z(z u + \sqrt{1 - z^2} w) \\
      & = (1 - z^2) u - z \sqrt{1 - z^2} w \\
			&=\sqrt{1 - z^2}\cdot  (\sqrt{1 - z^2} u - |z| \sign(z) w)\\
      & =  \| p_{\lin \{ U, z u + \sqrt{1 - z^2} w \}^\perp} (u) \| \cdot  \pi_{\lin \{ U, z u + \sqrt{1 - z^2} w \}^\perp} (u)
    \end{align*}
    and  $\| p_{\lin \{ U, z u + \sqrt{1 - z^2} w \}} (u) \| = |z|$.
    Furthermore, since also $F\subset u^\perp$, we have
    \begin{align*}
      [F, \lin \{ U, z u + \sqrt{1 - z^2} w \}] & = [F, U]^{(u^\perp)}  |z| , \\
      Q(\lin \{ U, z u + \sqrt{1 - z^2} w \}) & = Q(U) + (|z| u + \sqrt{1 - z^2} \sign(z) w )^2,
    \end{align*}
    and, for all $z  \in [-1, 1]\setminus\{0\}$ and $w \in U^\perp \cap u^\perp \cap \S$,
    \begin{align*}
       Q(F \cap \lin \{ U, z u + \sqrt{1 - z^2} w \}) = Q( F \cap U),
    \end{align*}	
    as $F\subset u^\perp$ and $U = \lin \{ U, z u + \sqrt{1 - z^2} w \} \cap u^\perp$.
    Using the fact that the integration with respect to $w$ is invariant under reflection in the origin, we obtain
    \begin{align*}
      J_{4} & = \frac {\omega_{n - k + j}} {2 \omega_n} \int_{\GOp(u^\perp, n - k + j - 1)} \int _{-1}^1 \int _{U^\perp \cap u^\perp \cap \S} \sum_{x = 0}^{\lfloor \frac s 2 \rfloor} \sum_{y = 0}^x \sum_{i = 2x}^{s} (-1)^{i + y} \binom{s}{i} \binom{i}{2x} \binom{x}{y} \Gamma(x + \tfrac 1 2) \\
      & \qquad \times \Gamma(\tfrac {n - k + s} 2 - x) u^{s - i} Q^{x - y} | z |^{j - s + 2x + 1} \left( 1 - z^2 \right)^{\frac {k - j + i - 2x - 2} 2} \left( \sqrt{1 - z^2} u - |z| w \right)^{i - 2x} \\
      & \qquad \times \left([F, U]^{(u^\perp)}\right)^2 \left( Q(U) + (|z| u + \sqrt{1 - z^2} w )^2 \right)^{y} Q(F \cap U)^l \\
      & \qquad \times \, \cH^{k - j - 1}(\intd w) \, \intd z \, \nu^{u^\perp}_{n - k + j - 1}( \intd U).
    \end{align*}
    Binomial expansion yields
    \begin{align*}
      \left( \sqrt{1 - z^2} u - |z| w \right)^{i - 2x} = \sum_{\alpha = 0}^{i - 2x} (-1)^{\alpha} \binom{i - 2x}{\alpha} \big( \sqrt{1 - z^2} u \big)^{i - 2x - \alpha} \big( |z| w \big)^{\alpha}.
    \end{align*}
    A change of the order of summation gives
    \begin{align*}
      J_{4} & = \frac {\omega_{n - k + j}} {2 \omega_n} \int_{\GOp(u^\perp, n - k + j - 1)} \left([F, U]^{(u^\perp)}\right)^2 \int _{-1}^1 \int _{U^\perp \cap u^\perp \cap \S} \sum_{x = 0}^{\lfloor \frac s 2 \rfloor} \sum_{y = 0}^x \sum_{\alpha = 0}^{s - 2x} (-1)^{y} \binom{x}{y} \\
      & \qquad \times w^\alpha \sum_{i = 2x + \alpha}^{s} (-1)^{i + \alpha} \binom{s}{i} \binom{i}{2x} \binom{i - 2x}{\alpha} \left( 1 - z^2 \right)^{i} \Gamma(\tfrac {n - k + s} 2 - x) \\
      & \qquad \times \Gamma(x + \tfrac 1 2) u^{s - 2x - \alpha} Q^{x - y} | z |^{j - s + 2x + \alpha + 1} \left( 1 - z^2 \right)^{\frac {k - j - 4x - \alpha - 2} 2} Q(F \cap U)^l \\
      & \qquad \times \left( Q(U) + (|z| u + \sqrt{1 - z^2} w)^2 \right)^{y} \, \cH^{k - j - 1}(\intd w) \, \intd z \, \nu^{u^\perp}_{n - k + j - 1}( \intd U).
    \end{align*}
    With Lemma \ref{Lem_Zeil_is} we conclude
    \begin{align}\label{absint}
      J_{4} & = \frac {\omega_{n - k + j}} {2 \omega_n} \sum_{x = 0}^{\lfloor \frac s 2 \rfloor}
			\sum_{y = 0}^x \sum_{\alpha = 0}^{s - 2x} (-1)^{y} \binom{x}{y} \binom{s}{2x}
			\binom{s - 2x}{\alpha} \Gamma(x + \tfrac 1 2) \Gamma(\tfrac {n - k + s} 2 - x)\nonumber \\
      & \qquad \times u^{s - 2x - \alpha} Q^{x - y} \int_{\GOp(u^\perp, n - k + j - 1)}
			\left([F, U]^{(u^\perp)}\right)^2 \int _{-1}^1 \int _{U^\perp \cap u^\perp \cap \S} \nonumber\\
      & \qquad \times | z |^{j + s - 2x - \alpha + 1} \left( 1 - z^2 \right)^{\frac {k - j + \alpha - 2} 2} Q(F \cap U)^l w^\alpha \nonumber\\
      & \qquad \times \left( Q(U) + (|z| u + \sqrt{1 - z^2} w)^2 \right)^{y} \, \cH^{k - j - 1}(\intd w) \,
			\intd z \, \nu^{u^\perp}_{n - k + j - 1}( \intd U).
    \end{align}
		 At this point we easily see that the integrals in $J_{4}$ are finite, since $j+s-2x-\alpha+1\ge 0$ and $k-j+\alpha-2\ge -1$.
		In fact, the absolute values of the integrands have finite integral,
		which also justifies the application of Lemma \ref{Lem_TraFo_HSS} above.
		Therefore, we can change the order of summation and integration from now on.
		We write $J_5$ for the integral with respect to $U$ multiplied by the factor
		$ {\omega_{n - k + j}} /{(2 \omega_n)} $.
		
By (twofold) binomial expansion of $( Q(U) + (|z| u + \sqrt{1 - z^2} w \})^2 )^{y}$ we obtain
    \begin{align*}
      J_{5} & = \frac {\omega_{n - k + j}} {2 \omega_n} \sum_{\beta = 0}^{y} \binom{y}{\beta} \int_{\GOp(u^\perp, n - k + j - 1)} \left([F, U]^{(u^\perp)}\right)^2 Q(U)^{y - \beta} Q(F \cap U)^l \\
      & \qquad \times \int _{-1}^1 | z |^{j + s - 2x - \alpha + 1} \left( 1 - z^2 \right)^{\frac {k - j + \alpha - 2} 2} \int _{U^\perp \cap u^\perp \cap \S} w^\alpha (|z| u + \sqrt{1 - z^2} w)^{2\beta} \\
      & \qquad \times \, \cH^{k - j - 1}(\intd w) \, \intd z \, \nu^{u^\perp}_{n - k + j - 1}( \intd U) \\
      & = \frac {\omega_{n - k + j}} {2 \omega_n} \sum_{\beta = 0}^{y} \sum_{\gamma = 0}^{2\beta} \binom{y}{\beta} \binom{2\beta}{\gamma} u^{2\beta - \gamma} \int_{\GOp(u^\perp, n - k + j - 1)} \left([F, U]^{(u^\perp)}\right)^2 Q(U)^{y - \beta} \\
      & \qquad \times Q(F \cap U)^l \ \int _{-1}^1 | z |^{j + s - 2x - \alpha + 2\beta - \gamma + 1} \left( 1 - z^2 \right)^{\frac {k - j + \alpha + \gamma - 2} 2} \, \intd z  \\
      & \qquad \times  \int _{U^\perp \cap u^\perp \cap \S} w^{\alpha + \gamma} \, \cH^{k - j - 1}(\intd w)  \, \nu^{u^\perp}_{n - k + j - 1}( \intd U).
    \end{align*}
		Using Lemma \ref{Lem_IntForm_Sph} and expressing the involved   spherical volumes in terms of Gamma functions, we get
    \begin{align*}
      J_{5} & = \frac {\Gamma( \frac {n} 2)} {\sqrt \pi \Gamma( \frac {n - k + j} 2)} \sum_{\beta = 0}^{y} \sum_{\gamma = 0}^{2\beta} \1\{\alpha + \gamma \text{ even}\} \binom{y}{\beta} \binom{2\beta}{\gamma} \\
      & \qquad \times \frac{\Gamma(\frac {j + s - \alpha - \gamma} 2 - x + \beta + 1) \Gamma(\frac {\alpha + \gamma + 1} 2)} {\Gamma(\frac {k + s} 2 - x + \beta + 1)} u^{2\beta - \gamma} \int_{\GOp(u^\perp, n - k + j - 1)} \\
      & \qquad \times \left([F, U]^{(u^\perp)}\right)^2 Q(U)^{y - \beta} Q(F \cap U)^l Q(U^\perp \cap u^\perp)^{\frac{\alpha + \gamma} 2} \, \nu^{u^\perp}_{n - k + j - 1}( \intd U).
    \end{align*}
    With an index shift in the summation with respect to $\gamma$ we obtain
    \begin{align*}
      J_{5} & = \frac {\Gamma( \frac {n} 2)} {\sqrt \pi \Gamma( \frac {n - k + j} 2)} \sum_{\beta = 0}^{y} \sum_{\gamma = \alpha}^{\alpha + 2\beta} \1\{\gamma \text{ even}\} \binom{y}{\beta} \binom{2\beta}{\gamma - \alpha} \\
      & \qquad \times \frac{\Gamma(\frac {j + s - \gamma} 2 - x + \beta + 1) \Gamma(\frac {\gamma + 1} 2)} {\Gamma(\frac {k + s } 2 - x + \beta + 1)} u^{\alpha + 2\beta - \gamma} \int_{\GOp(u^\perp, n - k + j - 1)} \\
      & \qquad \times \left([F, U]^{(u^\perp)}\right)^2 Q(U)^{y - \beta} Q(F \cap U)^l Q(U^\perp \cap u^\perp)^{\frac{\gamma} 2} \, \nu^{u^\perp}_{n - k + j - 1}( \intd U).
    \end{align*}
		We plug $J_5$ into $J_4$ and change the order of summation to get
    \begin{align*}
      J_{4} & = \frac {\Gamma( \frac {n} 2)} {\sqrt \pi \Gamma( \frac {n - k + j} 2)} \sum_{x = 0}^{\lfloor \frac s 2 \rfloor} \sum_{y = 0}^x \sum_{\beta = 0}^{y} \sum_{\gamma = 0}^{s - 2x + 2\beta} (-1)^{y} \1\{\gamma \text{ even}\} \binom{s}{2x} \binom{x}{y} \binom{y}{\beta} \\
      & \qquad \times \Gamma(x + \tfrac 1 2) \Gamma(\tfrac {n - k + s} 2 - x) \sum_{\alpha = (\gamma - 2\beta)^+}^{\min\{s - 2x, \gamma\}} \binom{s - 2x}{\alpha} \binom{2\beta}{\gamma - \alpha} \\
      & \qquad \times \frac{\Gamma(\frac {j + s-\gamma} 2 - x + \beta + 1) \Gamma(\frac {\gamma + 1} 2)} {\Gamma(\frac {k + s} 2 - x + \beta + 1)} u^{s - 2x + 2\beta - \gamma} Q^{x - y} \int_{\GOp(u^\perp, n - k + j - 1)} \\
      & \qquad \times \left([F, U]^{(u^\perp)}\right)^2 Q(U)^{y - \beta} Q(F \cap U)^l Q(U^\perp \cap u	^\perp)^{\frac{\gamma} 2} \, \nu^{u^\perp}_{n - k + j - 1}( \intd U).
    \end{align*}
   From Vandermonde's identity we conclude that
    \begin{align*}
      &\sum_{\alpha = (\gamma - 2\beta)^+}^{\min\{s - 2x, \gamma\}} \binom{s - 2x}{\alpha} \binom{2\beta}{\gamma - \alpha} = \binom{s - 2x + 2\beta}{\gamma},
    \end{align*}
    and thus
    \begin{align*}
      J_{4} & = \frac {\Gamma( \frac {n} 2)} {\sqrt \pi \Gamma( \frac {n - k + j} 2)} \sum_{x = 0}^{\lfloor \frac s 2 \rfloor} \sum_{y = 0}^x \sum_{\beta = 0}^{y} \sum_{\gamma = 0}^{\lfloor \frac {s} 2 \rfloor - x + \beta} (-1)^y \binom{s}{2x} \binom{x}{y} \binom{y}{\beta} \binom{s - 2x + 2\beta}{2\gamma} \Gamma(x + \tfrac 1 2) \\
      & \qquad \times \Gamma(\tfrac {n - k + s} 2 - x) \frac{\Gamma(\frac {j + s} 2 - x + \beta - \gamma + 1) \Gamma(\gamma + \frac {1} 2)} {\Gamma(\frac {k + s} 2 - x + \beta + 1)} u^{s - 2x + 2\beta - 2\gamma} Q^{x-y}\\
      & \qquad \times \int_{\GOp(u^\perp, n - k + j - 1)} \left([F, U]^{(u^\perp)}\right)^2 Q(U)^{y - \beta} Q(F \cap U)^l Q(U^\perp \cap u	^\perp)^{\gamma} \, \nu^{u^\perp}_{n - k + j - 1}( \intd U).
    \end{align*}
    Furthermore, the term $Q(U^\perp \cap u	^\perp)^{\gamma} = (Q(u^\perp) - Q(U) )^{\gamma}$
		can be expanded so that we obtain
    \begin{align*}
      J(\omega, \omega') & = \frac {\Gamma( \frac {n} 2) \Gamma(\frac {k - j} 2)} {2\pi \Gamma( \frac {n - k + j} 2) \Gamma(\frac {n - j + s} 2)} \cH^{k - j - 1}(\omega') \sum_{x = 0}^{\lfloor \frac s 2 \rfloor} \sum_{y = 0}^x \sum_{\beta = 0}^{y} \sum_{\gamma = 0}^{\lfloor \frac {s} 2 \rfloor - x + \beta} \sum_{\delta = 0}^{\gamma} (-1)^{y + \delta} \\
      & \qquad \times \binom{s}{2x} \binom{x}{y} \binom{y}{\beta} \binom{s - 2x + 2\beta}{2\gamma} \binom{\gamma}{\delta} {\textstyle \Gamma(x + \frac 1 2) \Gamma(\frac {n - k + s} 2 - x)} \\
      & \qquad \times \frac{\Gamma(\frac {j + s} 2 - x + \beta - \gamma + 1) \Gamma(\gamma + \frac {1} 2)} {\Gamma(\frac {k + s} 2 - x + \beta + 1)} Q^{x - y} \int _{\omega} u^{s - 2x + 2\beta - 2\gamma} Q(u^\perp)^{\gamma - \delta} \\
      & \qquad \times \int_{\GOp(u^\perp, n - k + j - 1)} \left([F, U]^{(u^\perp)}\right)^2 Q(U)^{y - \beta + \delta} Q(F \cap U)^l \, \nu^{u^\perp}_{n - k + j - 1}( \intd U) \\
      & \qquad \times \, \cH^{n - k - 1}(\intd u).
    \end{align*}
    Reversing the order of   summation, first with respect to $\beta$, and then with respect to $y$, we get
    \begin{align*}
      J(\omega, \omega') & = \frac {\Gamma( \frac {n} 2) \Gamma(\frac {k - j} 2)} {2\pi \Gamma( \frac {n - k + j} 2) \Gamma(\frac {n - j + s} 2)} \cH^{k - j - 1}(\omega') \sum_{x = 0}^{\lfloor \frac s 2 \rfloor} \sum_{y = 0}^x \sum_{\beta = 0}^{x - y} \sum_{\gamma = 0}^{\lfloor \frac {s} 2 \rfloor - y - \beta} \sum_{\delta = 0}^{\gamma} (-1)^{x + y + \delta} \\
      & \qquad \times \binom{s}{2x} \binom{x}{y} \binom{x - y}{\beta} \binom{s - 2y - 2\beta}{2\gamma} \binom{\gamma}{\delta} \Gamma(x + \tfrac 1 2) \Gamma(\tfrac {n - k + s} 2 - x) \\
      & \qquad \times \frac{\Gamma(\frac {j + s} 2 - y - \beta - \gamma + 1) \Gamma(\gamma + \frac {1} 2)} {\Gamma(\frac {k + s} 2 - y - \beta + 1)} Q^{y} \int _{\omega} u^{s - 2y - 2\beta - 2\gamma} Q(u^\perp)^{\gamma - \delta} \\
      & \qquad \times \int_{\GOp(u^\perp, n - k + j - 1)} \left([F, U]^{(u^\perp)}\right)^2 Q(U)^{\beta + \delta} Q(F \cap U)^l \, \nu^{u^\perp}_{n - k + j - 1}( \intd U)  \\
      & \qquad \times \cH^{n - k - 1}(\intd u).
    \end{align*}
      A change of the order of summation yields
    \begin{align*}
      J(\omega, \omega') & = \frac {\Gamma( \frac {n} 2) \Gamma(\frac {k - j} 2)} {2\pi \Gamma( \frac {n - k + j} 2) \Gamma(\frac {n - j + s} 2)} \cH^{k - j - 1}(\omega') \sum_{y = 0}^{\lfloor \frac s 2 \rfloor} \sum_{\beta = 0}^{\lfloor \frac s 2 \rfloor - y} \sum_{\gamma = 0}^{\lfloor \frac {s} 2 \rfloor - y - \beta} \sum_{\delta = 0}^{\gamma} (-1)^{y + \delta} \\
      & \qquad \times \sum_{x = y + \beta}^{\lfloor \frac s 2 \rfloor} (-1)^x \binom{s}{2x} \binom{x}{y} \binom{x - y}{\beta} {\textstyle \Gamma(x + \frac 1 2) \Gamma(\frac {n - k + s} 2 - x)} \\
      & \qquad \times \binom{s - 2y - 2\beta}{2\gamma} \binom{\gamma}{\delta} \frac{\Gamma(\frac {j + s} 2 - y - \beta - \gamma + 1) \Gamma(\gamma + \frac {1} 2)} {\Gamma(\frac {k + s} 2 - y - \beta + 1)} \\
      & \qquad \times Q^{y} \int _{\omega} u^{s - 2y - 2\beta - 2\gamma} Q(u^\perp)^{\gamma - \delta} \int_{\GOp(u^\perp, n - k + j - 1)} \left([F, U]^{(u^\perp)}\right)^2 \\
      & \qquad \times Q(U)^{\beta + \delta} Q(F \cap U)^l \, \nu^{u^\perp}_{n - k + j - 1}( \intd U)  \, \cH^{n - k - 1}(\intd u).
    \end{align*}
    By Legendre's duplication formula, applied several times, we obtain
    \begin{align*}
      & \binom{s}{2x} \binom{x}{y} \binom{x - y}{\beta} \Gamma(x + \tfrac 1 2) \\
      & \qquad = \binom{s}{2y + 2\beta} \binom{y + \beta}{y} \Gamma(y + \beta + \tfrac 1 2) \frac{\Gamma(\frac{s + 1}{2} - y - \beta) \Gamma(\frac{s}{2} - y - \beta + 1)}{\Gamma(\frac{s + 1}{2} - x) \Gamma(\frac{s}{2} - x + 1) (x - y - \beta)!} \\
      & \qquad = \binom{s}{2y + 2\beta} \binom{y + \beta}{y} \Gamma(y + \beta + \tfrac 1 2) \binom{\lfloor \frac s 2 \rfloor - y - \beta}{x - y - \beta} \frac{\Gamma(\lfloor \tfrac{s + 1}{2} \rfloor - y - \beta + \tfrac{1}{2})}{\Gamma(\lfloor \frac{s + 1}{2} \rfloor - x + \tfrac{1}{2})}.
    \end{align*}
    We denote the resulting sum with respect to $x$ by $S_3$. An index shift and a change of the order of summation imply that
    \begin{align*}
      S_{3} & = \sum_{x = 0}^{\lfloor \frac s 2 \rfloor - y - \beta} (-1)^{\lfloor \frac s 2 \rfloor + x} \binom{\lfloor \frac s 2 \rfloor - y - \beta}{x} \frac{\Gamma(\tfrac {n - k + s} 2 - \lfloor \frac s 2 \rfloor + x)}{\Gamma(\lfloor \frac{s + 1}{2} \rfloor - \lfloor \frac s 2 \rfloor + x + \tfrac{1}{2})}.
    \end{align*}
    Hence, an application of relation \eqref{lemZeilmod} and then of relation \eqref{Form_Gam_Cont} with $c = \tfrac{n-k+s}{2}-\lfloor \tfrac{s + 1}{2} \rfloor-\tfrac{1}{2}$
		and $m=\lfloor \frac s 2 \rfloor-y-\beta\in\N_0$ yield
    \begin{align*}
      S_{3} & = (-1)^{\lfloor \frac s 2 \rfloor} \frac{\Gamma(\tfrac {n - k + s} 2 - \lfloor \frac s 2 \rfloor) \Gamma( \overbrace{\textstyle \lfloor \frac{s + 1}{2} \rfloor + \lfloor \frac s 2 \rfloor}^{ = s} - \tfrac {n - k + s} 2 - y - \beta + \tfrac{1}{2})}{\Gamma(\lfloor \frac{s + 1}{2} \rfloor - y - \beta + \tfrac{1}{2}) \Gamma(\lfloor \frac{s + 1}{2} \rfloor - \tfrac {n - k + s} 2 + \tfrac{1}{2})} \\
      & = \overbrace{(-1)^{s + \lfloor \frac s 2 \rfloor + \lfloor \frac{s + 1}{2} \rfloor + y + \beta}}^{ = (-1)^{y + \beta}} \frac{\overbrace{\textstyle \Gamma(\tfrac {n - k + s} 2 - \lfloor \frac s 2 \rfloor) \Gamma(\tfrac {n - k + s} 2 - \lfloor \frac{s + 1}{2} \rfloor + \tfrac{1}{2})}^{ = \Gamma(\frac {n - k} 2) \Gamma(\frac {n - k + 1} 2)}}{\Gamma(\lfloor \frac{s + 1}{2} \rfloor - y - \beta + \tfrac{1}{2}) \Gamma( \tfrac {n - k + 1} 2 + y + \beta - \tfrac{s}{2})},
    \end{align*}
		where we used that $c\ge 0$, except for $k=n-1$ and odd $s$ when $c=-1/2$.
    Note that $S_3 = 0$ if $n - k + s$ is odd and $n - k + 1 \leq s - 2y - 2\beta$. Thus, we obtain
    \begin{align*}
      J(\omega, \omega') & = \frac {\Gamma( \frac {n} 2) \Gamma(\frac{n - k} 2) \Gamma(\frac {n - k + 1} 2) \Gamma(\frac {k - j} 2)} {2\pi \Gamma( \frac {n - k + j} 2) \Gamma(\frac {n - j + s} 2)} \cH^{k - j - 1}(\omega') \sum_{y = 0}^{\lfloor \frac s 2 \rfloor} \sum_{\beta = 0}^{\lfloor \frac s 2 \rfloor - y} \sum_{\gamma = 0}^{\lfloor \frac {s} 2 \rfloor - y - \beta} \sum_{\delta = 0}^{\gamma} (-1)^{\beta + \delta} \\
      & \quad \times \binom{s}{2y + 2\beta} \binom{y + \beta}{y} \binom{s - 2y - 2\beta}{2\gamma} \binom{\gamma}{\delta} \frac{\Gamma(\frac {j + s} 2 - y - \beta - \gamma + 1) \Gamma(\gamma + \frac {1} 2)} {\Gamma(\frac {k + s} 2 - y - \beta + 1)} \\
      & \quad \times \frac {\Gamma(y + \beta + \frac 1 2)} {\Gamma(\frac {n - k + 1} 2 + y + \beta - \frac s 2)} Q^{y} \int _{\omega} u^{s - 2y - 2\beta - 2\gamma} Q(u^\perp)^{\gamma - \delta} \int_{\GOp(u^\perp, n - k + j - 1)} \\
      & \quad \times \left([F, U]^{(u^\perp)}\right)^2 Q(U)^{\beta + \delta} Q(F \cap U)^l \, \nu^{u^\perp}_{n - k + j - 1}( \intd U)  \, \cH^{n - k - 1}(\intd u).
    \end{align*}
We conclude from Proposition \ref{Prop_Main} that
    \begin{align*}
      & \int_{\GOp(u^\perp, n - k + j - 1)} \left([F, U]^{(u^\perp)}\right)^2 Q(U)^{\beta + \delta} Q(F \cap U)^l \, \nu^{u^\perp}_{n - k + j - 1}( \intd U) \\
      & \qquad = \frac {(n - k + j - 1)! k!} {(n - 1)! j!} \frac {\Gamma(\frac {n + 1} 2) \Gamma(\frac {j} 2 + l) \Gamma(\frac{k} 2)}{\Gamma(\frac {n + 1} 2 + \beta + \delta) \Gamma(\frac{j} 2) \Gamma(\frac {k - j} 2 ) \Gamma(\frac {n - k + j + 1} 2)} \sum_{i = 0}^{\beta + \delta} \binom{\beta + \delta}{i} \\
      & \qquad\qquad \times \frac {(i + l - 2)!} {(l - 2)!} \frac {\Gamma(\frac {k - j} 2 + i) \Gamma(\frac {n - k + j + 1} 2 + \beta + \delta - i)} {\Gamma(\frac {k} 2 + l + i)} Q(u^\perp)^{\beta + \delta - i} Q(F)^{l + i},
    \end{align*}
    and hence we get
    \begin{align*}
      J(\omega, \omega') & = {\frac {\Gamma( \frac {n} 2) \Gamma(\frac {n + 1} 2) (n - k + j - 1)!} {(n - 1)! \Gamma( \frac {n - k + j} 2) \Gamma(\frac {n - k + j + 1} 2)}} \frac {k! \Gamma(\frac{k} 2) {\textstyle\Gamma(\frac{n - k} 2) \Gamma(\frac {n - k + 1} 2)} \Gamma(\frac {j} 2 + l)} {2 \pi j! \Gamma(\frac{j} 2) \Gamma(\frac {n - j + s} 2)} \cH^{k - j - 1}(\omega') \\
      & \qquad \times \sum_{y = 0}^{\lfloor \frac s 2 \rfloor} \sum_{\beta = 0}^{\lfloor \frac s 2 \rfloor - y} \sum_{\gamma = 0}^{\lfloor \frac {s} 2 \rfloor - y - \beta} \sum_{\delta = 0}^{\gamma} \sum_{i = 0}^{\beta + \delta} (-1)^{\beta + \delta} \binom{s}{2y + 2\beta} \binom{y + \beta}{y} \binom{s - 2y - 2\beta}{2\gamma} \binom{\gamma}{\delta} \\
      & \qquad \times \binom{\beta + \delta}{i} \frac{\Gamma(\frac {j + s} 2 - y - \beta - \gamma + 1) \Gamma(\gamma + \frac {1} 2)} {\Gamma(\frac {k + s} 2 - y - \beta + 1)} \frac {\Gamma(y + \beta + \frac 1 2)} {\Gamma(\frac {n - k + 1} 2 + y + \beta - \frac s 2)} \\
      & \qquad \times \frac {(i + l - 2)!} {(l - 2)!} \frac {\Gamma(\frac {k - j} 2 + i) } {\Gamma(\frac {k} 2 + l + i)} \frac {\Gamma(\frac {n - k + j + 1} 2 + \beta + \delta - i)}{\Gamma(\frac {n + 1} 2 + \beta + \delta)} Q^{y} Q(F)^{l + i} \\
      & \qquad \times \int _{\omega} u^{s - 2y - 2\beta - 2\gamma} Q(u^\perp)^{\beta + \gamma - i} \, \cH^{n - k - 1}(\intd u).
    \end{align*}
		To simplify the right-hand side  we apply Legendre's duplication formula three times.
    Then binomial expansion of $Q(u^{\perp})^{\beta + \gamma - i} = (Q - u^{2})^{\beta + \gamma - i}$ and an index shift in the resulting sum yield
    \begin{align*}
      J(\omega, \omega') & = \frac {k! (n - k - 1)! \Gamma(\frac{k} 2) \Gamma(\frac {j} 2 + l)} {2^{n - j} \sqrt \pi j! \Gamma(\frac{j} 2) \Gamma(\frac {n - j + s} 2)} \cH^{k - j - 1}(\omega') \sum_{y = 0}^{\lfloor \frac {s} 2 \rfloor} \sum_{\beta = 0}^{\lfloor \frac {s} 2 \rfloor - y} \sum_{\gamma = 0}^{\lfloor \frac {s} 2 \rfloor - y - \beta} \sum_{\delta = 0}^{\gamma} \sum_{i = 0}^{\beta + \delta} \\
      & \qquad \times \sum_{m = y + i}^{y + \beta + \gamma} (-1)^{m + y + \gamma + \delta} \binom{s}{2y + 2\beta} \binom{y + \beta}{y} \binom{s - 2y - 2\beta}{2\gamma} \binom{\gamma}{\delta} \\
      & \qquad \times \binom{\beta + \delta}{i} \binom{\beta + \gamma - i}{m - y - i} \frac {(i + l - 2)!} {(l - 2)!} \frac {\Gamma(y + \beta + \frac 1 2)} {\Gamma(\frac {n - k + 1} 2 + y + \beta - \frac s 2)} \\
      & \qquad \times \frac{\Gamma(\frac {j + s} 2 - y - \beta - \gamma + 1) \Gamma(\gamma + \frac {1} 2)} {\Gamma(\frac {k + s} 2 - y - \beta + 1)} \frac {\Gamma(\frac {n - k + j + 1} 2 + \beta + \delta - i)}{\Gamma(\frac {n + 1} 2 + \beta + \delta)} \\
      & \qquad \times \frac {\Gamma(\frac {k - j} 2 + i) } {\Gamma(\frac {k} 2 + l + i)} Q^{m - i} Q(F)^{l + i} \int _{\omega} u^{s - 2m} \, \cH^{n - k - 1}(\intd u).
    \end{align*}
    An index shift in the summation with respect to $\beta$ implies that
    \begin{align*}
      J(\omega, \omega') & = \frac {k! (n - k - 1)! \Gamma(\frac{k} 2) \Gamma(\frac {j} 2 + l)} {2^{n - j} \sqrt \pi j! \Gamma(\frac{j} 2) \Gamma(\frac {n - j + s} 2)} \cH^{k - j - 1}(\omega') \sum_{y = 0}^{\lfloor \frac {s} 2 \rfloor} \sum_{\beta = y}^{\lfloor \frac {s} 2 \rfloor} \sum_{\gamma = 0}^{\lfloor \frac {s} 2 \rfloor - \beta} \sum_{\delta = 0}^{\gamma} \sum_{i = 0}^{\beta + \delta - y} \\
      & \qquad \times \sum_{m = y + i}^{\beta + \gamma} (-1)^{m + y + \gamma + \delta} \binom{s}{2\beta} \binom{\beta}{y} \binom{s - 2\beta}{2\gamma} \binom{\gamma}{\delta} \binom{\beta + \delta - y}{i} \\
      & \qquad \times \binom{\beta + \gamma - y - i}{m - y - i} \frac {(i + l - 2)!} {(l - 2)!} \frac {\Gamma(\beta + \frac 1 2)} {\Gamma(\frac {n - k + 1} 2 + \beta - \frac s 2)} \\
      & \qquad \times \frac{\Gamma(\frac {j + s} 2 - \beta - \gamma + 1) \Gamma(\gamma + \frac {1} 2)} {\Gamma(\frac {k + s} 2 - \beta + 1)} \frac {\Gamma(\frac {n - k + j + 1} 2 + \beta + \delta - y - i)}{\Gamma(\frac {n + 1} 2 + \beta + \delta - y)} \\
      & \qquad \times \frac {\Gamma(\frac {k - j} 2 + i) } {\Gamma(\frac {k} 2 + l + i)} Q^{m - i} Q(F)^{l + i} \int _{\omega} u^{s - 2m} \, \cH^{n - k - 1}(\intd u).
    \end{align*}
    By a change of the order of summation we finally obtain
    \begin{align*}
      J(\omega, \omega') & = \cH^{k - j - 1}(\omega') \sum_{m = 0}^{\lfloor \frac s 2 \rfloor} \sum_{i = 0}^{m} b_{n, j, k}^{s, l, i} \, \hat a_{n, j, k}^{s, i, m} \, Q^{m - i} Q(F)^{l + i} \int _{\omega} u^{s - 2m} \, \cH^{n - k - 1}(\intd u),
    \end{align*}
    where
    \begin{align*}
      b_{n, j, k}^{s, l, i} & : = \frac {\Gamma(\frac{k} 2)} {2^{n - j} \sqrt \pi \Gamma(\frac{j} 2) \Gamma(\frac {n - j + s} 2)} \frac{k! (n - k - 1)!}{j!} \frac {(i + l - 2)!} {(l - 2)!} \frac { \Gamma(\frac {j} 2 + l) \Gamma(\frac {k - j} 2 + i) } {\Gamma(\frac {k} 2 + l + i)}, \\
      \hat a_{n, j, k}^{s, i, m} & : = \sum_{y = 0}^{m - i} \sum_{\beta = y}^{\lfloor \frac s 2 \rfloor} \sum_{\gamma = (m - \beta)^+}^{\lfloor \frac {s} 2 \rfloor - \beta} \sum_{\delta = (i - \beta + y)^+}^{\gamma} (-1)^{m + y + \gamma + \delta} \binom{s}{2\beta} \binom{\beta}{y} \binom{s - 2\beta}{2\gamma} \binom{\gamma}{\delta} \\
      & \qquad \times \binom{\beta + \delta - y}{i} \binom{\beta + \gamma - y - i}{m - y - i} \Gamma(\beta + \tfrac 1 2) \Gamma(\gamma + \tfrac {1} 2) \\
      & \qquad \times \frac{\Gamma(\frac {j + s} 2 - \beta - \gamma + 1)} {\Gamma(\frac {k + s} 2 - \beta + 1) \Gamma(\frac {n - k + 1} 2 + \beta - \frac s 2)} \frac {\Gamma(\frac {n - k + j + 1} 2 + \beta + \delta - y - i)}{\Gamma(\frac {n + 1} 2 + \beta + \delta - y)}.
    \end{align*}

  \subsection{Simplifying the coefficients}\label{subsec5.4}

    In this section, we simplify the coefficients $\hat a_{n, j, k}^{s, i, m}$ by a repeated change of the order of summation and by repeated application of relations \eqref{lemZeilfund} and \eqref{lemZeilmod}.
		
    First, an index shift by $\beta$, applied to the summation with respect to $\gamma$, gives
    \begin{align*}
      \hat a_{n, j, k}^{s, i, m} & = \sum_{y = 0}^{m - i} \sum_{\beta = y}^{\lfloor \frac s 2 \rfloor} \sum_{\gamma = \max \{ \beta, m \}}^{\lfloor \frac {s} 2 \rfloor} \sum_{\delta = (i - \beta + y)^+}^{\gamma - \beta} (-1)^{m + y + \gamma + \beta + \delta} \binom{s}{2\beta} \binom{\beta}{y} \binom{s - 2\beta}{2\gamma - 2\beta} \\
      & \qquad \times \binom{\gamma - \beta}{\delta} \binom{\beta + \delta - y}{i} \binom{\gamma - y - i}{m - y - i} \Gamma(\beta + \tfrac 1 2) \Gamma(\gamma - \beta + \tfrac {1} 2) \\
      & \qquad \times \frac{\Gamma(\frac {j + s} 2 - \gamma + 1)} {\Gamma(\frac {k + s} 2 - \beta + 1) \Gamma(\frac {n - k + 1} 2 + \beta - \frac s 2)} \frac {\Gamma(\frac {n - k + j + 1} 2 + \beta + \delta - y - i)}{\Gamma(\frac {n + 1} 2 + \beta + \delta - y)}.
    \end{align*}
    A change of the order of summation yields
    \begin{align*}
      \hat a_{n, j, k}^{s, i, m} & = \sum_{y = 0}^{m - i} \sum_{\gamma = m}^{\lfloor \frac {s} 2 \rfloor} \sum_{\beta = y}^{\gamma} \sum_{\delta = (i - \beta + y)^+}^{\gamma - \beta} (-1)^{m + y + \gamma + \beta + \delta} \binom{s}{2\beta} \binom{\beta}{y} \binom{s - 2\beta}{2\gamma - 2\beta} \binom{\gamma - \beta}{\delta} \\
      & \qquad \times \binom{\beta + \delta - y}{i} \binom{\gamma - y - i}{m - y - i} \Gamma(\beta + \tfrac 1 2) \Gamma(\gamma - \beta + \tfrac {1} 2) \\
      & \qquad \times \frac{\Gamma(\frac {j + s} 2 - \gamma + 1)} {\Gamma(\frac {k + s} 2 - \beta + 1) \Gamma(\frac {n - k + 1} 2 + \beta - \frac s 2)} \frac {\Gamma(\frac {n - k + j + 1} 2 + \beta + \delta - y - i)}{\Gamma(\frac {n + 1} 2 + \beta + \delta - y)}.
    \end{align*}
    Shifting the index of the summation with respect to $\delta$ by $\beta$, we obtain
    \begin{align*}
      \hat a_{n, j, k}^{s, i, m} & = \sum_{y = 0}^{m - i} \sum_{\gamma = m}^{\lfloor \frac {s} 2 \rfloor} \sum_{\beta = y}^{\gamma} \sum_{\delta = \max \{ \beta, i + y \}}^{\gamma} (-1)^{m + y + \gamma + \delta} \binom{s}{2\beta} \binom{\beta}{y} \binom{s - 2\beta}{2\gamma - 2\beta} \\
      & \qquad \times \binom{\gamma - \beta}{\delta - \beta} \binom{\delta - y}{i} \binom{\gamma - y - i}{m - y - i} \Gamma(\beta + \tfrac 1 2) \Gamma(\gamma - \beta + \tfrac {1} 2) \\
      & \qquad \times \frac{\Gamma(\frac {j + s} 2 - \gamma + 1)} {\Gamma(\frac {k + s} 2 - \beta + 1) \Gamma(\frac {n - k + 1} 2 + \beta - \frac s 2)} \frac {\Gamma(\frac {n - k + j + 1} 2 + \delta - y - i)}{\Gamma(\frac {n + 1} 2 + \delta - y)}.
    \end{align*}
    A change of the order of summation gives
    \begin{align*}
      \hat a_{n, j, k}^{s, i, m} & = \sum_{y = 0}^{m - i} \sum_{\gamma = m}^{\lfloor \frac {s} 2 \rfloor} \sum_{\delta = i + y }^{\gamma} \sum_{\beta = y}^{\delta} (-1)^{m + y + \gamma + \delta} \binom{s}{2\beta} \binom{\beta}{y} \binom{s - 2\beta}{2\gamma - 2\beta} \\
      & \qquad \times \binom{\gamma - \beta}{\delta - \beta} \binom{\delta - y}{i} \binom{\gamma - y - i}{m - y - i} \Gamma(\beta + \tfrac 1 2) \Gamma(\gamma - \beta + \tfrac {1} 2) \\
      & \qquad \times \frac{\Gamma(\frac {j + s} 2 - \gamma + 1)} {\Gamma(\frac {k + s} 2 - \beta + 1) \Gamma(\frac {n - k + 1} 2 + \beta - \frac s 2)} \frac {\Gamma(\frac {n - k + j + 1} 2 + \delta - y - i)}{\Gamma(\frac {n + 1} 2 + \delta - y)}.
    \end{align*}
    We conclude from an index shift by $y$, applied to the summation with respect to $\beta$,
    \begin{align*}
      \hat a_{n, j, k}^{s, i, m} & = \sum_{y = 0}^{m - i} \sum_{\gamma = m}^{\lfloor \frac {s} 2 \rfloor} \sum_{\delta = i + y }^{\gamma} \sum_{\beta = 0}^{\delta - y} (-1)^{m + y + \gamma + \delta} \binom{s}{2y + 2\beta} \binom{y + \beta}{y} \binom{s - 2y - 2\beta}{2\gamma - 2y - 2\beta} \\
      & \qquad \times \binom{\gamma - y - \beta}{\delta - y - \beta} \binom{\delta - y}{i} \binom{\gamma - y - i}{m - y - i} \Gamma(y + \beta + \tfrac 1 2) \Gamma(\gamma - y - \beta + \tfrac {1} 2) \\
      & \qquad \times \frac{\Gamma(\frac {j + s} 2 - \gamma + 1)} {\Gamma(\frac {k + s} 2 - y - \beta + 1) \Gamma(\frac {n - k + 1} 2 + y + \beta - \frac s 2)} \frac {\Gamma(\frac {n - k + j + 1} 2 + \delta - y - i)}{\Gamma(\frac {n + 1} 2 + \delta - y)},
    \end{align*}
    and by $- i - y$, applied to the summation with respect to $\delta$,
    \begin{align*}
      \hat a_{n, j, k}^{s, i, m} & = \sum_{y = 0}^{m - i} \sum_{\gamma = m}^{\lfloor \frac {s} 2 \rfloor} \sum_{\delta = 0}^{\gamma - y - i} \sum_{\beta = 0}^{i + \delta} (-1)^{i + m + \gamma + \delta} \binom{s}{2y + 2\beta} \binom{y + \beta}{y} \binom{s - 2y - 2\beta}{2\gamma - 2y - 2\beta} \\
      & \qquad \times \binom{\gamma - y - \beta}{i + \delta - \beta} \binom{i + \delta}{i} \binom{\gamma - y - i}{m - y - i} \Gamma(y + \beta + \tfrac 1 2) \Gamma(\gamma - y - \beta + \tfrac {1} 2) \\
      & \qquad \times \frac{\Gamma(\frac {j + s} 2 - \gamma + 1)} {\Gamma(\frac {k + s} 2 - y - \beta + 1) \Gamma(\frac {n - k + 1} 2 + y + \beta - \frac s 2)} \frac {\Gamma(\frac {n - k + j + 1} 2 + \delta)}{\Gamma(\frac {n + 1} 2 + i + \delta)}.
    \end{align*}
    With Legendre's duplication formula (applied three times) we obtain
    \begin{align*}
      & \binom{s}{2y + 2\beta} \binom{y + \beta}{y} \binom{s - 2y - 2\beta}{ 2\gamma - 2y - 2\beta} \binom{\gamma - y - \beta}{i + \delta - \beta} \binom{i + \delta}{i} \binom{\gamma - y - i}{m - y - i} \\
      & \qquad \qquad \times \Gamma(y + \beta + \tfrac 1 2) \Gamma(\gamma - y - \beta + \tfrac {1} 2) \\
      & \qquad = \binom{s}{2i} \binom{m - i}{y} \binom{\gamma - i}{m - i} \binom{s - 2i}{2\gamma - 2i} \binom{\gamma - y - i}{\delta} \binom{i + \delta}{\beta} \Gamma(i + \tfrac {1} 2) \Gamma(\gamma - i + \tfrac {1} 2),
    \end{align*}
    and hence
    \begin{align*}
      \hat a_{n, j, k}^{s, i, m} & = \Gamma(i + \tfrac {1} 2) \binom{s}{2i} \sum_{y = 0}^{m - i} \sum_{\gamma = m}^{\lfloor \frac {s} 2 \rfloor} \sum_{\delta = 0}^{\gamma - y - i} \sum_{\beta = 0}^{i + \delta} (-1)^{i + m + \gamma + \delta} \\
      & \qquad \times \binom{m - i}{y} \binom{\gamma - i}{m - i} \binom{s - 2i}{2\gamma - 2i} \binom{\gamma - y - i}{\delta} \binom{i + \delta}{\beta} \Gamma(\gamma - i + \tfrac {1} 2) \\
      & \qquad \times \frac{\Gamma(\frac {j + s} 2 - \gamma + 1)} {\Gamma(\frac {k + s} 2 - y - \beta + 1) \Gamma(\frac {n - k + 1} 2 + y + \beta - \frac s 2)} \frac {\Gamma(\frac {n - k + j + 1} 2 + \delta)}{\Gamma(\frac {n + 1} 2 + i + \delta)}.
    \end{align*}

    Now we define $a_{n, j, k}^{s, i, m} : = ( \Gamma(i + \tfrac 1 2) \binom{s}{2i} )^{-1} \hat a_{n, j, k}^{s, i, m}$.
    We first use relation \eqref{lemZeilfund} and then apply relation \eqref{lemZeilmod} twice. Thus we obtain
    \begin{align*}
      & \sum_{\beta = 0}^{i + \delta} \binom{i + \delta}{\beta} \frac{1} {\Gamma(\frac {k + s} 2 - y - \beta + 1) \Gamma(\frac {n - k - s + 1} 2 + y + \beta)} \\
      & \qquad \qquad \qquad = \frac{\Gamma(\frac{n + 1}{2} + i + \delta)} {\Gamma(\frac {k + s} 2 - y + 1) \Gamma(\frac {n - k - s + 1} 2 + i + y + \delta) \Gamma(\frac{n + 1}{2})},
    \end{align*}
    \begin{align*}
      & \sum_{\delta = 0}^{\gamma - y - i} (-1)^{\delta} \binom{\gamma - y - i}{\delta} \frac {\Gamma(\frac {n - k + j + 1} 2 + \delta)}{\Gamma(\frac {n - k - s + 1} 2 + i + y + \delta)} \\
      & \qquad \qquad \qquad = \frac {\Gamma(\frac {n - k + j + 1} 2) \Gamma(-\frac {j + s} 2 + \gamma)}{\Gamma(\frac {n - k - s + 1} 2 + \gamma) \Gamma(- \frac {j + s} 2 + i + y)} \\
      & \qquad \qquad \qquad = (-1)^{i + \gamma + y} \frac {\Gamma(\frac {n - k + j + 1} 2) \Gamma(\frac {j + s} 2 - i - y + 1)}{\Gamma(\frac {n - k - s + 1} 2 + \gamma) \Gamma(\frac {j + s} 2 - \gamma + 1)},
    \end{align*}
    where we used \eqref{Form_Gam_Cont} with $c = \frac {j + s} 2 - i - y\ge 0$ and $m = \gamma - i - y \in \N_{0}$ in the second step, and
    \begin{align*}
      \sum_{y = 0}^{m - i} (-1)^{m + y} \binom{m}{y} \frac{\Gamma(\frac {j + s} 2 - i - y + 1)} {\Gamma(\frac {k + s} 2 - y + 1)} & = \sum_{y = 0}^{m - i} (-1)^{i + y} \binom{m}{y} \frac{\Gamma(\frac {j + s} 2 - m + y + 1)} {\Gamma(\frac {k + s} 2 + i - m + y + 1)} \\
      & = (-1)^{i} \frac{\Gamma(\frac {j + s} 2 - m + 1) \Gamma(\frac {k - j} 2 + m)} {\Gamma(\frac {k + s} 2 + 1) \Gamma(\frac {k - j} 2 + i)}.
    \end{align*}
    This gives
    \begin{align*}
      a_{n, j, k}^{s, i, m} & = (-1)^{i} \frac{\Gamma(\frac {n - k + j + 1} 2) \Gamma(\frac {j + s} 2 - m + 1) \Gamma(\frac {k - j} 2 + m)} {\Gamma(\frac {n + 1} 2) \Gamma(\frac {k + s} 2 + 1) \Gamma(\frac {k - j} 2 + i)} \\
      & \qquad \times \sum_{\gamma = m}^{\lfloor \frac {s} 2 \rfloor} \binom{\gamma - i}{m - i} \binom{s - 2i}{2\gamma - 2i} \frac {\Gamma(\gamma - i + \tfrac {1} 2)}{\Gamma(\frac {n - k - s + 1} 2 + \gamma)}.
    \end{align*}
    We deduce from Legendre's duplication formula that
    \begin{align*}
      \binom{\gamma - i}{m - i} \binom{s - 2i}{2\gamma - 2i} \Gamma(\gamma-i + \tfrac {1} 2) & = \frac{\sqrt \pi}{(m - i)!} \frac{\Gamma(\frac {s + 1} 2 - i) \Gamma(\frac {s} 2 - i + 1)}{(\gamma - m)! \Gamma(\frac {s + 1} 2 - \gamma) \Gamma(\frac {s} 2 - \gamma + 1)} \\
      & = \sqrt \pi \binom{\lfloor \frac {s} 2 \rfloor - i}{m - i} \binom{\lfloor \frac {s} 2 \rfloor - m}{\gamma - m} \frac{\Gamma(\lfloor \frac {s + 1} 2 \rfloor - i + \frac{1}{2})}{\Gamma(\lfloor \frac {s + 1} 2 \rfloor - \gamma + \tfrac{1}{2})}.
    \end{align*}
    Denoting the remaining sum in $a_{n, j, k}^{ s, i, m}$ with respect to $\gamma$ by $S_{4}$, we obtain
    \begin{align*}
      S_{4} & = \sum_{\gamma = 0}^{\lfloor \frac {s} 2 \rfloor - m} \binom{\lfloor \frac {s} 2 \rfloor - m}{\gamma} \frac {1}{\Gamma(\lfloor \frac {s + 1} 2 \rfloor - m - \gamma + \tfrac{1}{2}) \Gamma(\frac {n - k - s + 1} 2 + m + \gamma)},
    \end{align*}
  for which relation \eqref{lemZeilfund}  yields
    \begin{align*}
      S_{4} & = \frac {\Gamma(\frac {n - k - s} 2 +  {\lfloor \tfrac {s + 1} 2 \rfloor + \lfloor \tfrac {s} 2 \rfloor}  - m)}{\Gamma(\lfloor \frac {s + 1} 2 \rfloor - m + \tfrac{1}{2}) {\textstyle \Gamma(\frac {n - k + 1} 2 + \lfloor \frac {s} 2 \rfloor - \frac{s}{2}) \Gamma(\frac {n - k} 2 + \lfloor \frac {s + 1} 2 \rfloor - \frac{s}{2})} }. \\
      & = \frac{\Gamma(\frac {n - k + s} 2 - m)}{\Gamma(\frac {n - k + 1} 2) \Gamma(\frac {n - k} 2) \Gamma(\lfloor \frac {s + 1} 2 \rfloor - m + \tfrac{1}{2})}.
    \end{align*}
    We obtain from Legendre's duplication formula
    \begin{align*}
      & \sqrt \pi \binom{\lfloor \frac {s} 2 \rfloor - i}{m - i} \Gamma(\lfloor \tfrac {s + 1} 2 \rfloor - i + \tfrac{1}{2}) S_{4} \\
      & \qquad = \frac{\Gamma(\frac {n - k + s} 2 - m)}{\Gamma(\frac {n - k + 1} 2) \Gamma(\frac {n - k} 2)} \frac{\sqrt \pi}{(m - i)!} \frac {\Gamma(\frac {s} 2 - i + 1) \Gamma(\tfrac {s + 1} 2 - i)}{\Gamma(\frac {s} 2 - m + 1) \Gamma(\frac {s + 1} 2 - m)} \\
      & \qquad = \frac{\Gamma(\frac {n - k + s} 2 - m)}{\Gamma(\frac {n - k + 1} 2) \Gamma(\frac {n - k} 2)} \binom{s - 2i}{2m - 2i} \Gamma(m - i + \tfrac{1}{2}).
    \end{align*}
    This gives
    \begin{align*}
      a_{n, j, k}^{s, i, m} & = (-1)^{i} \binom{s - 2i}{2m - 2i} \Gamma(m - i + \tfrac{1}{2}) \frac{\Gamma(\frac {n - k + j + 1} 2)} {\Gamma(\frac {n + 1} 2) \Gamma(\frac {n - k + 1} 2) \Gamma(\frac {n - k} 2) \Gamma(\frac {k + s} 2 + 1)} \\
      & \qquad \times \frac{\Gamma(\frac {n - k + s} 2 - m) \Gamma(\frac {j + s} 2 - m + 1) \Gamma(\frac {k - j} 2 + m)}{\Gamma(\frac {k - j} 2 + i)}.
    \end{align*}
Next, using
	$$	 {\binom{s}{2i} \binom{s - 2i}{2m - 2i} \Gamma(i + \tfrac 1 2) \Gamma(m - i + \tfrac{1}{2})}  = \frac{s!}{(s - 2m)!} \frac{\pi}{2^{2m} i! (m - i)!} ,
	$$
$$\frac{(n - k - 1)! k!}{\Gamma(\frac {n - k + 1} 2) \Gamma(\frac {n - k} 2) j!} = \frac{2^{n - j - 1}} {\sqrt \pi} \frac {\Gamma(\frac {k} 2 + 1) \Gamma(\frac {k + 1} 2)} {\Gamma(\frac {j} 2 + 1) \Gamma(\frac {j + 1} 2)} ,
		$$
		and
	$$\frac { \Gamma(\frac {j} 2 + l) \Gamma(\frac {n - k + s} 2 - m) \Gamma(\frac{k} 2)} { \Gamma(\frac {n - j + s} 2) \Gamma(\frac{j} 2) \Gamma(\frac {k} 2 + l + i)}  = \sqrt\pi^{j - k - 2i - 2m} \frac{\omega_{n - j + s} \omega_{j} \omega_{k + 2l + 2i}}{\omega_{j + 2l} \omega_{n - k + s - 2m} \omega_{k}} ,
	$$
we get
    \begin{align*}
      c_{n, j, k}^{s, l, i, m}  :&= \frac{\omega_{n - k} \omega_{k - j}}{\omega_{n - j}} \frac{c_{n, j}^{r, s, l}}{c_{n, k}^{r, s - 2m, l + i}} \binom{s}{2i} \Gamma(i + \tfrac{1}{2}) b_{n, j, k}^{s, l, i} a_{n, j, k}^{s, i, m} \\
      &  = (-1)^{i}  {\binom{s}{2i} \binom{s - 2i}{2m - 2i} \Gamma(i + \tfrac 1 2) \Gamma(m - i + \tfrac{1}{2})} {\frac{(n - k - 1)! k!}{\Gamma(\frac {n - k + 1} 2) \Gamma(\frac {n - k} 2) j!}}  \\
      & \qquad \qquad \times \frac {(i + l - 2)!} {(l - 2)!} \frac {\Gamma(\frac {n - k + j + 1} 2) \Gamma(\frac {j + s} 2 - m + 1) \Gamma(\frac {k - j} 2 + m)} {2^{n - j} \sqrt \pi \Gamma(\frac {n + 1} 2) \Gamma(\frac {k + s} 2 + 1)} \\
      & \qquad \qquad \times  \frac { \Gamma(\frac {j} 2 + l) \Gamma(\frac {n - k + s} 2 - m) \Gamma(\frac{k} 2)} { \Gamma(\frac {n - j + s} 2) \Gamma(\frac{j} 2) \Gamma(\frac {k} 2 + l + i)}   \frac{2 \sqrt\pi^{k - j} }{\Gamma(\frac{k - j}{2})} \frac{\omega_{n - k}}{\omega_{n - j}} \frac{c_{n, j}^{r, s, l}}{c_{n, k}^{r, s - 2m, l + i}} \\
      &   = (-1)^{i} \frac{1}{4^{m} i! (m - i)!} \frac{1}{\pi^{i + m}} \frac {(i + l - 2)!} {(l - 2)!} \frac {\Gamma(\frac {n - k + j + 1} 2) \Gamma(\frac {k + 1} 2)} {\Gamma(\frac {n + 1} 2) \Gamma(\frac {j + 1} 2)} \\
      & \qquad \qquad \times \frac {\Gamma(\frac {k} 2 + 1)} {\Gamma(\frac {k + s} 2 + 1)} \frac {\Gamma(\frac {j + s} 2 - m + 1)} {\Gamma(\frac {j} 2 + 1)} \frac {\Gamma(\frac {k - j} 2 + m)} {\Gamma(\frac{k - j}{2})},
    \end{align*}
which yields the assertion.

    Finally, returning to (\ref{form_I1_mit_J}) and using the definition of the tensorial curvature measures, we get
    \begin{align*}
      I_1 & = \sum_{k = j + 1}^{n - 1} \sum_{m = 0}^{\lfloor \frac s 2 \rfloor} \sum_{i = 0}^{m} c_{n, j, k}^{s, l, i, m} \, Q^{m - i} \frac{1}{\omega_{n - k}} c_{n, k}^{r, s - 2m, l + i} \\
      & \qquad \qquad \times \sum_{F \in \cF_k(P)} Q(F)^{l + i} \int _{F \cap \beta} x^r \, \cH^k(\intd x) \int _{N(P, F) \cap \S} u^{s - 2m} \, \cH^{n - k - 1}(\intd u) \\
      & \qquad \qquad \times \frac{1}{\omega_{k - j}} \sum_{F' \in \cF_{n - k + j}(P')} \cH^{n - k + j}(F' \cap \beta') \cH^{k - j - 1} \left( N(P',F') \cap \S \right) \\
      &	\qquad + \TenCM{j}{r}{s}{l}(P, \beta) \CM{n}(P', \beta') + c_{n, j}^{s} \, \TenCM{n}{r}{0}{\frac s 2 + l} (P, \beta)
			\CM{j}(P', \beta')\\
      & = \sum_{k = j}^{n} \sum_{m = 0}^{\lfloor \frac s 2 \rfloor} \sum_{i = 0}^{m} c_{n, j, k}^{s, l, i, m} \, Q^{m - i} \TenCM{k}{r}{s - 2m}{l + i} (P, \beta) \CM{n - k + j}(P', \beta').
    \end{align*}
		In the last step, we use that for $k=j$ we have
      $c_{n, j, j}^{s, l, i, m} = \1\{ i = m = 0 \}$. Moreover, in the case $k=n$ we use that
		 $\TenCM{n}{r}{s - 2m}{l + i}$ vanishes for $m \neq \frac{s}{2}$. Hence, for even $s$ we have to simplify the sum
      \begin{align*}
        &  \sum_{i = 0}^{\frac s 2} c_{n, j, n}^{s, l, i, \frac s 2} \,  Q^{\frac s 2 - i} \TenCM{n}{r}{0}{l + i} (P, \beta)
				\CM{j}(P', \beta')\\
				&\qquad =\sum_{i = 0}^{\frac s 2} c_{n, j, n}^{s, l, i, \frac s 2} \,
				\frac{\omega_{n + 2l + 2i}}{\omega_{n + s + 2l}} \TenCM{n}{r}{0}{\frac s 2 + l} (P, \beta) \CM{j}(P', \beta').
      \end{align*}
    For this, an application of relation  \eqref{lemZeilmod} yields
      \begin{align}
         &  {\frac{1}{(2\pi)^{s} \left(\frac s 2\right)!} \frac {\Gamma(\frac {n + s} 2 + l)} {\Gamma(l - 1)} \frac {\Gamma(\frac {n} 2 + 1)} {\Gamma(\frac {n + s} 2 + 1)} \frac {\Gamma(\frac {n - j + s} 2)} {\Gamma(\frac{n - j}{2})} \sum_{i = 0}^{\frac s 2} (-1)^{i} \binom{\frac s 2}{i} \frac {\Gamma(i + l - 1)} {\Gamma(\frac n 2 + l + i)}}\nonumber\\
        & \qquad\qquad \times
				\TenCM{n}{r}{0}{\frac s 2 + l} (P, \beta) \CM{j}(P', \beta')\nonumber\\
				&\qquad =c_{n,j}^s \,\TenCM{n}{r}{0}{\frac s 2 + l} (P, \beta) \CM{j}(P', \beta'),\label{PlusPlus}
      \end{align}
	as required. This completes the proof.

  \appendix
  \numberwithin{equation}{section}
  \numberwithin{Satz}{section}

  \section{Explicit sum expressions}\label{secA}

      In this section, we first establish closed form expressions for sums which are required in the preceding parts.

      \begin{Lemma} \label{Lem_Zeil_yq} Let $q \in \N_0$, $b, c \in \R$. Then
        \begin{equation}\label{lemZeilfund}
          \sum_{y=0}^q \binom{q}{y} \frac {1} {\Gamma(b + y) \Gamma(c - y)} = \frac { \Gamma(b + c + q - 1)} {\Gamma(c) \Gamma(b + q) \Gamma(b + c - 1)}.
        \end{equation}
      \end{Lemma}

      In this work, we often use a consequence of Lemma \ref{Lem_Zeil_yq}: With \eqref{Form_Gam_Cont} we obtain for $a > 0$ and $b\in\R$ the relation
        \begin{equation}\label{lemZeilmod} \tag{\ref{lemZeilfund}$'$}
          \sum_{y=0}^q (-1)^{y} \binom{q}{y} \frac {\Gamma(a + y)} {\Gamma(b + y)} = \frac {\Gamma(a) \Gamma(b - a + q)} {\Gamma(b + q) \Gamma(b - a)},
        \end{equation}
which extends a corresponding lemma in \cite{HugWeis16a} to the range $b\le 0$.

For the convenience of the reader, we include a proof of Lemma \ref{Lem_Zeil_yq}.

      \begin{proof}
        For the proof, we can assume that $b, c \notin \mathbb{Z}$. The general case then follows from
				a continuity argument. We set
        \begin{align*}
           F(q, y) := \binom{q}{y} \frac {1} {\Gamma(b + y) \Gamma(c - y)},\qquad q,y\in\N_0.
        \end{align*}
        Then we have $F(q, y) = 0$ if $y \notin  \{ 0, \ldots, q \}$. We set
        \begin{align*}
          f(q) := \sum_{y = 0}^{q} F(q, y),\qquad q\in\N_0.
        \end{align*}
        Furthermore, for $q,y\in\N_0$ we define
        \begin{align*}
          G(q, y) :=
            \begin{cases}
              \frac{y (b + y - 1)}{q - y + 1} F(q, y), \qquad & \text{ for } y \in \{ 0, \ldots, q\}, \\[1ex]
              G(q, q) -(b + q) F(q + 1, q) & \\
              \qquad + (b + c + q - 1) F(q, q), \qquad & \text{ for } y = q + 1, \\
              0, & \text{ for } y\ge q+2.
            \end{cases}
        \end{align*}
        A direct calculation yields
        \begin{align*}
          -(b + q - 1) F(q, y) + (b + c + q - 2) F(q - 1, y) = G(q - 1, y + 1) - G(q - 1, y)
        \end{align*}
        for $y \in \N_{0}$ and $q\in\N$.
				Summing this relation for all $y \in \{ 0, \ldots, q \}$ gives
        \begin{align*}
          - (b + q - 1) f(q) + (b + c + q - 2) f(q - 1) = 0,
        \end{align*}
        and thus recursively
        \begin{align*}
          f(q) & = \frac {(b + c + q - 3)(b - a + q - 2)} {(b + q - 2) (b + q - 1)} f(q - 2) \\
          & \ \, \vdots \\
          & = \frac {(b + c - 1) \cdots (b + c + q - 2)} {b \cdots (b + q - 1)} f(0) \\
          & =  \frac {\Gamma(b + c + q - 1) \Gamma(b)} {\Gamma(b + q) \Gamma(b + c - 1)} f(0).
        \end{align*}
        With $f(0) = \frac{1}{\Gamma(b)\Gamma(c)}$ we obtain the assertion.
      \end{proof}

    \begin{Lemma}\label{Lem_Zeil_is}
    	Let $\alpha,  \beta, \gamma \in \N$, $0 < j < n$. Then
    	\begin{align*}
    		&\sum_{i = 2x + \alpha}^{s} (-1)^{i + \alpha} \binom{s}{i} \binom{i}{2x} \binom{i - 2x}{\alpha} \left( 1 - z^2 \right)^{i} \\
    		& \qquad = \binom{s}{2x} \binom{s - 2x}{\alpha} z^{2s - 4x - 2\alpha} (1- z^2)^{2x + \alpha}.
				\end{align*}
    \end{Lemma}

    \begin{proof}
      We start with an index shift
      \begin{align*}
        & \sum_{i = 2x + \alpha}^{s} (-1)^{i + \alpha} \binom{s}{i} \binom{i}{2x} \binom{i - 2x}{\alpha} \left( 1 - z^2 \right)^{i} \\
        & \qquad = \left( 1 - z^2 \right)^{2x + \alpha} \sum_{i = 0}^{s - 2x - \alpha} \binom{s}{i + 2x + \alpha} \binom{i + 2x + \alpha}{2x} \binom{i + \alpha}{\alpha} \left( z^2 - 1 \right)^{i}.
      \end{align*}
It is easy to check that
$$\binom{s}{i + 2x + \alpha} \binom{i + 2x + \alpha}{2x} \binom{i + \alpha}{\alpha}  = \binom{s}{2x} \binom{s - 2x}{\alpha} \binom{s - 2x - \alpha}{i} .
$$
Then the binomial theorem yields
      \begin{align*}
        \sum_{i = 0}^{s - 2x - \alpha} \binom{s - 2x - \alpha}{i} \left( z^2 - 1 \right)^{i} & = z^{2s - 4x - 2\alpha},
      \end{align*}
      and thus the assertion.
    \end{proof}
    	
  Finally, we also provide a required measurability result.  	
    	
    	\begin{Lemma}\label{lemmeas1}
			For $j\in\{0,\ldots,n-1\}$, $r,s,l\in\N_0$, and $\eta\in\mathcal{B}(\Sigma^n)$, the mapping $P\mapsto 
			\tilde{\phi}_j^{r,s,l}(P,\eta)$, $P\in\mathcal{P}^n$, is measurable. 
			\end{Lemma}

\begin{proof}
For the proof, it is sufficient to consider the case where $r=s=0$ and $\eta=\beta\times\omega$ 
with $\beta\in\mathcal{B}(\R^n)$ and  $\omega\in\mathcal{B}(\mathbb{S}^{n-1})$. For a locally compact 
Hausdorff space $E$ with a countable base, let $\mathcal{F}(E)$ denote the system of closed subsets of $E$. 
With the Fell topology $\mathcal{F}(E)$ becomes a compact  Hausdorff space with a countable base and 
$\mathcal{F}'(E):=\mathcal{F}(E)\setminus\{\emptyset\}$ is a locally compact subspace. Then $\mathcal{K}^n$ and 
$\mathcal{P}^n$ are measurable subsets of $\mathcal{F}(\R^n)$ and the subspace topology on these subsets coincides 
with the topology induced by the Hausdorff metric (see \cite[Theorem 12.3.4]{SchnWeil08}). 
Further, let ${\sf N}(E)$ 
denote the set of counting measures on $\mathcal{B}(E)$. On ${\sf N}(E)$ we write $\mathcal{N}(E)$ for the sigma algebra generated by the evaluation maps $\eta\mapsto \eta(A)$, where $A\in \mathcal{B}(E)$. We refer to Chapters 3.1 and 12.2 in 
\cite{SchnWeil08} for details.  

In the proof of \cite[Lemma 10.1.2]{SchnWeil08} it is shown that the map 
$\mathcal{P}^n\to\mathcal{F}(\mathcal{F}'(\R^n))$, $P\mapsto 
\mathcal{F}_k(P)$, is measurable. By 
\cite[Lemma 3.1.4]{SchnWeil08} it follows then that the map 
 $\mathcal{P}^n\to\mathcal{P}^n\times{\sf N}(\mathcal{F}'(\R^n))$, 
$P\mapsto (P,\eta_{\mathcal{F}_k(P)})$ is also measurable, where $\eta_{\mathcal{F}_k(P)}$ 
is the simple counting measure with support $\mathcal{F}_k(P)$. Further, if $g:\mathcal{P}^n\times 
\mathcal{F}'(\R^n)\to[0,\infty]$ is measurable, then the map 
$$
\mathcal{P}^n\times {\sf N}(\mathcal{F}'(\R^n))\to[0,\infty],\qquad (P,\eta)\mapsto \int g(P,F)\, \eta(\intd F),
$$
is measurable. Thus, to prove the assertion of the lemma, it is sufficient to show that, for all $\beta\in \mathcal{B}(\R^n)$ and $\omega\in\mathcal{B}(\mathbb{S}^{n-1})$, the map $g$ defined by
$$
g(P,F):=\1\{F\in\mathcal{F}_k(P)\}(Q(F)^l)_{i_1\ldots i_{2l}}\,\mathcal{H}^k(F\cap \beta)\,\mathcal{H}^{n-1-k}(N(P,F)\cap \omega), 
$$
for $(P,F)\in \mathcal{P}^n\times \mathcal{F}'(\R^n)$, 
is measurable, where the definition is to be understood in the sense that $g(P,F):=0$ if $F\notin\mathcal{F}_k(P)$ and 
where $(Q(F)^l)_{i_1\ldots i_{2l}}$ is the coordinate  of the tensor $Q(F)^l$ with respect to some basis of $\mathbb{T}^{2l}$. Note that $\text{aff}(F)^0=\bigcup\{\lambda(F-s(F)):\lambda\in\N\}$, where $s(K)\in\text{relint}(K)$ is the Steiner point of a convex body $K\in\mathcal{K}^n$ (see \cite[p.~50]{Schneider14}). 
Since the maps $\mathcal{P}^n\to\mathcal{F}(\mathcal{F}'(\R^n))$, $P\mapsto 
\mathcal{F}_k(P)$,  and $s:\mathcal{K}^n\to\R^n$  are measurable, the measurability of the mapping  
$(P,F)\mapsto \1\{F\in\mathcal{F}_k(P)\}(Q(F)^l)_{i_1\ldots i_{2l}}$ is implied by 
Theorems 12.2.3, 12.2.7 and 12.3.1 in \cite{SchnWeil08}. 
Moreover, it follows that  $M_k:=\{(P,F)\in \mathcal{P}^n\times\mathcal{F}'(\R^n):F\in \mathcal{F}_k(P)\}$ 
is measurable. 

Next we show 
that the map $M_k\to \mathcal{F}'(\R^n)$, $(P,F)\mapsto N(P,F)\cap \mathbb{S}^{n-1}$, is measurable. For this, 
observe that $s(F)\in\text{relint}(F)$ implies that $N(P,F)=N(P,s(F))$, where $N(K,x)$ denotes 
the normal cone of a convex body $K\in\mathcal{K}^n$ at the point $x\in K$. 
Since $M:=\{(P,x)\in\mathcal{P}^n\times\R^n:x\in P\}$ is a  measurable subset of  $\mathcal{P}^n\times\R^n$, 
and $s:\mathcal{K}^n\to\R^n$ is measurable, it is sufficient to show that the map $T:M\to \mathcal{F}'(\R^n)$, 
$(P,x)\mapsto N(P,x)\cap \mathbb{S}^{n-1}$, is measurable. To see this, let $C\subset\R^n$ be compact. It is sufficient to prove that $M_C:=\{(P,x)\in M:T(P,x)\cap C=\emptyset\}$ is open in $M$. Aiming at a contradiction, we assume that there are $(P_i,x_i)\in M\setminus M_C$, for $i\in\N$, with $(P_i,x_i)\to (P,x)\in M_C$ as $i\to\infty$. Then there are 
$u_i\in  N(P_i,x_i)\cap \mathbb{S}^{n-1}\cap C$ for $i\in\N$. By compactness, there is a subsequence $u_{i_j}$, $j\in\N$, 
which converges to $u\in \mathbb{S}^{n-1}\cap C$. For a convex body $K\in\mathcal{K}^n$, a point $x\in K$, and $v\in\R^n$ we have $v\in N(K,x)$ if and only if $\langle v,x\rangle=h(K,v)$, where $h(K,v)$ is the support function 
$h(K,\cdot)$ of $K$ evaluated at $v$. By assumption, we have $\langle u_{i_j},x_{i_j}\rangle=h(P_{i_j},u_{i_j})$ for $j\in\N$. Since the support function depends continuously on $(K,v)$, it follows that $\langle u,x\rangle=h(P,u)$, and thus $u\in N(P,x)$. This yields $u\in  N(P,x)\cap\mathbb{S}^{n-1}\cap C\neq\emptyset$, that is, $(P,x)\notin M_C$, a contradiction. 

The measurability of $g$ now follows by applying twice \cite[Corollary 2.1.4]{Zaehle82}, since the indicator function 
ensures that effectively the Hausdorff measures  $\mathcal{H}^k(F\cap \cdot)$ and $\mathcal{H}^{n-1-k}(N(P,F)\cap \mathbb{S}^{n-1}\cap  \cdot)$ 
are locally finite.
\end{proof}

\end{document}